\documentclass[letterpaper,psamsfonts,12pt]{article}
\usepackage{pinlabel}
\usepackage{graphicx}
\usepackage{subfigure}
\usepackage[small,bf]{caption}
\usepackage{amsfonts}
\usepackage{amssymb}
\usepackage[all,arc]{xy}
\usepackage{units}
\usepackage{mathrsfs} 
\usepackage{amsmath}
\usepackage{amsthm}
\usepackage{amscd}
\usepackage{verbatim}
\usepackage{fullpage}
\usepackage{cite}
\usepackage{leftidx}
\usepackage{marginnote}
\usepackage{ifthen}
\usepackage{colonequals}
\usepackage{hyperref}

\theoremstyle{plain}
\newtheorem{thm}{Theorem}[section]
\newtheorem{cor}[thm]{Corollary}
\newtheorem{prop}[thm]{Proposition}
\newtheorem{lem}[thm]{Lemma}
\newtheorem{claim}[thm]{Claim}
\newtheorem*{claim*}{Claim}

\newtheorem{mques}[thm]{Motivating Question}
\newtheorem*{fact*}{Fact}
\newtheorem{mythm}[thm]{Theorem}
\newtheorem{thmref}{Theorem} 

\theoremstyle{definition}
\newtheorem{defn}[thm]{Definition}

\newtheorem{exmp}[thm]{Example}

\newtheorem{notn}[thm]{Notation}

\newtheorem{assm}[thm]{Assumptions}
\newtheorem{rem}[thm]{Remark}
\newtheorem*{rem*}{Remark}
\newtheorem{obs}[thm]{Observation}

\newcommand{\abs}[1]{\left\vert#1\right\vert}

\newcommand{\ceil}[1]{\left\lceil#1\right\rceil}

\newcommand{\R}{\mathbb {R}}

\newcommand{\Hyp}{\mathbb {H}}
\newcommand{\hyp}{\Hyp^2}

\newcommand{\Z}{\mathbb {Z}}

\newcommand{\inv}{^{-1}}
\newcommand{\eps}{\varepsilon}
\newcommand{\push}{\mathcal{P}} 
\newcommand{\card}[1]{\abs{#1}}

\DeclareMathOperator{\Mod}{Mod}

\DeclareMathOperator{\pp}{PP}
\DeclareMathOperator{\spec}{spec}
\DeclareMathOperator{\least}{L}
\DeclareMathOperator{\intnum}{i}

\newcommand{\pushdil}[1]{\lambda_{#1}}
\newcommand{\dilat}[1]{\lambda_{#1}}

\newcommand{\lpath}[2]{\leftidx{_{#1}}{#2}}
\newcommand{\liftgamma}{\Gamma}
\newcommand{\marked}{\mathcal{M}}
\newcommand{\baseloop}{\alpha}
\newcommand{\baselift}{\baseloop_{0}}
\newcommand{\basepoint}{p_0}
\newcommand{\turns}{\mathcal{W}}
\newcommand{\turnat}[1]{\tau_{#1}}

\newcommand{\bedge}{\varepsilon_0}
\newcommand{\liftimage}[1]{\baseloop_{#1}}
\newcommand{\liftclass}[2]{[\baseloop_{#1}]_{#2}}
\newcommand{\Gstab}{G_{\bedge}}
\newcommand{\pants}{\mathcal{C}}
\newcommand{\seam}{c_i}
\newcommand{\liftpants}{\tilde{\pants}}
\newcommand{\tree}{\mathcal{T}}
\newcommand{\vertices}{V(\tree)}
\newcommand{\edges}{E(\tree)}
\newcommand{\proj}{\sigma}
\newcommand{\len}{l_{\pants}}

\newcommand{\cone}[1]{V_{#1}}
\newcommand{\weights}[1]{E_{#1}}
\newcommand{\MFS}{\mathscr{M\!F}}

\newcommand{\pretrack}[1]{\tau_{#1}}
\newcommand{\itime}[2]{t_{#1#2}}

\newcommand{\drag}{d}
\newcommand{\lside}{l}
\newcommand{\rside}{r}
\newcommand{\sidecor}{c}
\newcommand{\cross}{a}
\newcommand{\outcor}{b}
\newcommand{\ewt}[1]{
\ifthenelse{\equal{#1}{1}}{a}{}
\ifthenelse{\equal{#1}{2}}{b}{}
\ifthenelse{\equal{#1}{3}}{c}{}
\ifthenelse{\equal{#1}{4}}{e}{}
\ifthenelse{\equal{#1}{5}}{f}{}
\ifthenelse{\equal{#1}{6}}{g}{}
\ifthenelse{\equal{#1}{7}}{h}{}
\ifthenelse{\equal{#1}{8}}{i}{}
\ifthenelse{\equal{#1}{9}}{j}{}
\ifthenelse{\equal{#1}{10}}{u}{}
\ifthenelse{\equal{#1}{11}}{w}{}
\ifthenelse{\equal{#1}{12}}{x}{}
\ifthenelse{\equal{#1}{13}}{y}{}
\ifthenelse{\equal{#1}{14}}{z}{}
}

\renewcommand{\thefootnote}{\fnsymbol{footnote}}

\makeatletter
\let\c@equation\c@thm
\makeatother
\numberwithin{equation}{section}

\title{Dilatation versus self-intersection number for point-pushing pseudo-Anosov homeomorphisms}

\author{Spencer Dowdall}

\date{\today}

\begin{document}

\maketitle

\begin{abstract}
A filling curve $\gamma$ on a based surface $S$ determines a pseudo-Anosov homeomorphism $\push(\gamma)$ of $S$ via the process of ``point-pushing along $\gamma$.'' We consider the relationship between the self-intersection number $\intnum(\gamma)$ of $\gamma$ and the dilatation $\pushdil{\gamma}$ of $\push(\gamma)$; our main result is that $\left(\intnum(\gamma)+1\right)^{\nicefrac{1}{5}} \leq \pushdil{\gamma} \leq 9^{\intnum(\gamma)}$. We also bound the least dilatation of any pseudo-Anosov in the point-pushing subgroup of a closed surface and prove that this number tends to infinity with genus. Lastly, we investigate the minimal entropy of any pseudo-Anosov homeomorphism obtained by pushing along a curve with self-intersection number $k$ and show that, for a closed surface, this number grows like $\log(k)$.\footnotetext{2000 {\normalfont\itshape
  Mathematics Subject Classification\/} 37E30 (primary) 37D20, 37B40, 57M99 (secondary).}
\end{abstract}

\renewcommand{\thefootnote}{\arabic{footnote}}
\section{Introduction}\label{sec:introduction}

In this paper we consider the entropy generated by ``stirring'' a surface $S$ in the following manner. Place a finger on a point $p\in S$ and smoothly deform the surface by pushing $p$ along a closed path $\gamma\colon [0,1]\to S$; the resulting homeomorphism of $S$ mixes the surface just as one stirs a pot of soup. Certainly the amount of entropy introduced in this way depends entirely upon the stirring pattern: pushing along a simple path will have little effect, whereas following a complicated path that winds all over the surface will mix things up in short order. The goal of this paper is to understand \emph{how} the entropy depends on the pushing path.

\subsubsection*{Point-pushing homeomorphisms}

Throughout, $S=S_{g,n}$ will denote the surface obtained from a closed, connected, orientable surface of genus $g$ by removing $n\geq 0$ points or punctures. The \emph{mapping class group} of $S$ is the group $\Mod(S)$ of isotopy classes of orientation-preserving homeomorphisms of $S$. By a \emph{marked} or \emph{based} surface $(S,p)$, we simply mean a surface $S$ with basepoint $p\in S$; the corresponding \emph{based mapping class group} $\Mod(S,p)$ is defined analogously by restricting to homeomorphisms and isotopies of the pair $(S,p)$. These two groups are related by the \emph{Birman exact sequence} \cite[\S4.1]{Birman74}:
\begin{equation*}
\xymatrix{
1 \ar[r] & \pi_1(S,p) \ar[r]^{\push} & \Mod(S,p) \ar[r]^{\text{forget}} & \Mod(S) \ar[r] & 1
}
\end{equation*}
The first map $\push$ constitutes the stirring procedure mentioned above. It is called the \emph{point-pushing homomorphism} and is constructed by pushing the basepoint $p$ around a closed curve $\gamma\in\pi_1(S,p)$; see \S\ref{point-pushingDef} for details. Representatives of the resulting class $\push(\gamma)$ are called \emph{point-pushing homeomorphisms}.

The second map is obtained by simply forgetting the basepoint. Thus the image of $\push$ consists precisely of those mapping classes in $\Mod(S,p)$ that become trivial when one allows isotopies to move the basepoint. This subgroup is called the \emph{point-pushing subgroup} of $\Mod(S,p)$ and will be denoted by
\begin{equation*}
\pp(S) \colonequals \push(\pi_1(S,p)) \leq \Mod(S,p).
\end{equation*}

\subsubsection*{Pseudo-Anosov dilatation}
A mapping class $f\in \Mod(S,p)$ exhibits the mixing behavior that interests us precisely if it is \emph{pseudo-Anosov}, meaning that it has a representative homeomorphism which respectively stretches and contracts a transverse pair of measured foliations on $(S,p)$ by some stretching factor $\dilat{f}>1$; see \cite{FLP79} or \cite{FarbMargalit}. This stretching factor $\dilat{f}$ is called the \emph{dilatation} of $f$; it is an algebraic integer that is an important measure of the dynamical properties of $f$. For instance,

\begin{itemize}
\item $\log(\dilat{f})$ is the minimal topological entropy of any representative homeomorphism in the mapping class of $f$,
\item $\log(\dilat{f})$ is the translation length of the isometric action of $f$ on the Teichm\"uller space of $S\setminus\{p\}$ equipped with the Teichm\"uller metric; thus $\log(\dilat{f})$ also represents the length of the geodesic loop corresponding to $[f]$ in the moduli space of hyperbolic structures on $S\setminus\{p\}$,
\item for any simple closed curve $\alpha\subset S\setminus\{p\}$ and any Riemannian metric $g$ on $S\setminus\{p\}$, the length of $f^k(\alpha)$ grows like $\dilat{f}^k$; more precisely, $\sqrt[k]{l_g(f^k(\alpha))} \to \dilat{f}$ for all $\alpha$ and $g$.
\end{itemize}

According to the Nielsen--Thurston classification \cite{Thurston88,Bers78}, $f\in \Mod(S,p)$ is pseudo-Anosov if and only if no iterate of $f$ fixes the isotopy class of any essential simple closed curve on $(S,p)$. Here, a simple closed curve $\alpha\subset S$ is \emph{essential} if it is neither homotopically trivial nor homotopic into every neighborhood of a puncture; the \emph{(essential) simple closed curves on $(S,p)$} are the same as for $S\setminus\{p\}$. We refer the reader to \cite{FLP79} or \cite{FarbMargalit} for a more thorough discussion of the basic properties of pseudo-Anosov mapping classes.

\subsubsection*{Dilatations in $\pp(S)$}
In the context of the point-pushing subgroup, there is a simple criterion, due to Kra, that determines whether a pushing loop defines a pseudo-Anosov mapping class. Recall that a closed curve $\gamma\subset S$ \emph{fills} $S$ if every loop that is freely homotopic to $\gamma$ intersects every essential simple closed curve in $S$. The following is taken from \cite[Theorem 2']{Kra81}.

\begin{thm}[Kra]\label{thm:Kra}
Let $S = S_{g,n}$ be an orientable surface satisfying $3g + n > 3$, and let $\gamma\in \pi_1(S,p)$ be a closed curve on $S$. Then the mapping class $\push(\gamma)\in\Mod(S,p)$ is pseudo-Anosov if and only if $\gamma$ fills $S$.
\end{thm}

It is clear from the definition that $\gamma$ must fill $S$ in order for $\push(\gamma)$ to be pseudo-Anosov; the point of Theorem~\ref{thm:Kra} is that every sufficiently complicated curve does produce a pseudo-Anosov mapping class under point-pushing. For a quick proof of this result, see the elegant argument given by Farb and Margalit in \cite[Theorem 14.6]{FarbMargalit}.

Each free homotopy class of oriented closed curves on $S$ corresponds to a conjugacy class in $\pi_1(S,p)$. As dilatation is a conjugacy invariant, it follows that to the free homotopy class of each oriented filling curve $\gamma\subset S$ we may assign a number
\begin{equation*}
\pushdil{\gamma} \colonequals \text{the dilatation of $\push(\gamma)$}
\end{equation*}
whose logarithm measures the entropy introduced by stirring the surface along the path $\gamma$. Our primary goal is now to answer the following general question.

\begin{mques}\label{ques:goal}
How does $\pushdil{\gamma}$ depend on the the complexity of the filling curve $\gamma$?
\end{mques}

In order to address Question~\ref{ques:goal}, we need a quantitative measure of the complexity of a closed curve. The most natural choice seems to be self-intersection number.

\begin{defn}[Self-intersection number]\label{defn:selfIntNum}
If $\gamma\colon S^1 \to S$ is a closed curve on the surface $S$, then the \emph{(geometric) self-intersection number} of $\gamma$ is defined to be the quantity
\[ \intnum(\gamma)\colonequals \min_{\mu\in[\gamma]} \frac{1}{2}\card{\left\{ (x,y) \mid \text{$x,y\in S^1$, $x\neq y$, and $\mu(x)=\mu(y)$} \right\}},\]
where the minimum is taken over all closed curves $\mu$ in the free homotopy class of $\gamma$.
\end{defn}

We point out that $\intnum(\gamma)$ is an integer and that it is the ``obvious'' geometric quantity---for representative curves whose intersections are all $4$--valent, the quantity being minimized is just the number of intersection points.

Our goal is to relate $\lambda_\gamma$ and $\intnum(\gamma)$. A simple argument shows that, on a fixed surface $S$, $\pushdil{\gamma}$ tends to infinity as $\intnum(\gamma)$ increases in the sense that
\begin{equation}\label{eqn:toInfinity}
\text{for every $K$ there exists some $N$ so that $\intnum(\gamma)\geq N \implies \pushdil{\gamma} > K$.}
\end{equation}
Indeed, since $f\push(\gamma)f\inv = \push(f(\gamma))$ for any $f\in \Mod(S,p)$, we see that $\intnum(\gamma)$ is a conjugacy invariant of $\push(\gamma)$ and that \eqref{eqn:toInfinity} is a consequence of Ivanov's \cite{Ivanov88} well-known compactness property: For each $K>1$ there are only finitely many conjugacy classes of pseudo-Anosov elements $f\in \Mod(S,p)$ with $\dilat{f} \leq K$. However, this does not explain \emph{how} $\pushdil{\gamma}$ depends on $\intnum(\gamma)$ nor at what rate $\pushdil{\gamma}$ tends to infinity. Our main result addresses these issues by describing an explicit relationship between dilatation and self-intersection number.

\begin{mythm}\label{thm:generalBounds}
Let $S = S_{g,n}$ be a surface satisfying $3g + n > 3$. If $\gamma\colon[0,1]\to S$ is a closed filling curve on $S$ that represents a primitive element of $\pi_1(S)$, then the dilatation $\pushdil{\gamma}$ of the mapping class $\push(\gamma)$ satisfies
\[\sqrt[5]{\intnum(\gamma)+1} \leq \pushdil{\gamma} \leq 9^{\intnum(\gamma)}.\]
Furthermore, the upper bound holds without the assumption that $\gamma$ is primitive.
\end{mythm}

Surprisingly, these bounds are independent of the surface---they only depend on the geometric complexity of the pushing curve $\gamma$. Recalling the context of stirring on a surface, this shows that any stirring path with many self-crossings is guaranteed to generate a lot of entropy and that, in a sense, the entropy is an actual consequence of the complexity of the stirring path.
Since dilatation grows faster than self-intersection number upon taking powers of $\gamma\in\pi_1(S)$, the lower bound in Theorem~\ref{thm:generalBounds} may be extended to non-primitive filling curves as follows.

\begin{cor}\label{cor:nonPrimitiveLowerBounds}
Let $S = S_{g,n}$ be a surface satisfying $3g+n > 3$, and let $\gamma\colon[0,1]\to S$ be a closed filling curve on $S$.
\begin{itemize}
\item[i)] If $S = S_{0,4}$ or $S_{1,2}$ and $\gamma$ is the square of a primitive element in $\pi_1(S)$, then $\pushdil{\gamma}\geq\sqrt[5]{\intnum(\gamma)}$.
\item[ii)]  If $S = S_{1,1}$ and $\gamma$ is the second, third, or fourth power of a primitive element, then $ \pushdil{\gamma} \geq \sqrt[5]{(\intnum(\gamma)+1)/2}$.
\item[iii)] In all other cases, $\pushdil{\gamma} \geq \sqrt[5]{\intnum(\gamma)+1}$.
\end{itemize}
\end{cor} 

\begin{rem*}\label{rem:surfacesWithBoundary}
We have chosen to work in the context of (possibly) punctured surfaces in order to ease the exposition. However, Theorem~\ref{thm:generalBounds} and Corollary~\ref{cor:nonPrimitiveLowerBounds} also hold for compact surfaces with boundary. The conversion between these two contexts is straightforward and left to the reader.
\end{rem*}

It is also interesting to consider Question~\ref{ques:goal} in the context of other measures of complexity. For example, the \emph{lower     central series} $\{G_i\}$ and the \emph{derived series} $\{G^{(i)}\}$ of a group $G=G_1 = G^{(1)}$ are the recursively defined sequences
\begin{equation*}
G_{k+1} = [G_{k},G]\qquad \text{and} \qquad G^{(k+1)} = [G^{(k)},G^{(k)}],
\end{equation*}
respectively. For non-abelian surface groups $G = \pi_1(S_{g,n})$, Malestein and Putman \cite{MalesteinPutman09} have related the self-intersection number of a nontrivial element $\gamma\in G$ to its depth in both the lower central series and the derived series of $G$. More precisely, they showed 
that $\intnum(\gamma)\geq \log_8(k)-1$ for all $\gamma\in G_k$, 
that $\intnum(\gamma)\geq 2^{\lceil k/2\rceil}-2$ for all $\gamma\in G^{(k)}$,
and that $\intnum(\gamma)\geq \tfrac{k}{4g+n-1}-1$ when $\gamma\in G_k$ and $n\geq 1$ \cite{MalesteinPutman09}.
Combining their results with Corollary~\ref{cor:nonPrimitiveLowerBounds} immediately implies the following corollary.

\begin{cor}\label{cor:algebraicConnection}
Let $S = S_{g,n}$ be a surface satisfying $3g + n > 5$, and let $\gamma\in G = \pi_1(S,p)$ be a loop that fills $S$. For $k\geq 1$, let $G_k$ and $G^{(k)}$ denote the $k^{\text{th}}$ terms in the lower central series and the derived series of $G$, respectively.
\begin{enumerate}
\renewcommand{\theenumi}{\arabic{enumi}}
\renewcommand{\labelenumi}{(\theenumi)}
\item[i)] If $\gamma\in G_k$, then 
$\pushdil{\gamma} \geq \left(\log_8(k)\right)^{\nicefrac{1}{5}}$.
\item[ii)] If $\gamma\in G^{(k)}$, then 
$\pushdil{\gamma} \geq \left(2^{\ceil{\nicefrac{k}{2}}-2}+1\right)^{\nicefrac{1}{5}} \geq 2^{\frac{k-4}{10}}$.
\item[iii)] Furthermore, if $\gamma\in G_k$  and $n\geq 1$, then $\pushdil{\gamma} \geq \left(\frac{k}{4g+n-1}\right)^{\nicefrac{1}{5}}$.
\end{enumerate}
(As per Corollary~\ref{cor:nonPrimitiveLowerBounds}, slightly weaker bounds hold for surfaces satisfying $4 \leq 3g+n \leq 5$.)
\end{cor}

\subsubsection*{The spectrum of pseudo-Anosov dilatations}
In addition to studying the dilatation of individual elements, one may also consider the spectrum 
\begin{equation*}
\spec(A) \colonequals \{\log(\dilat{f}) \mid f\in A\text{ is pseudo-Anosov}\}
\end{equation*}
of all entropies attained in a particular subset $A\subseteq \Mod(S)$ (or $A\subset \Mod(S,p)$) of mapping classes. The aforementioned compactness property \cite{Ivanov88} implies that $\spec(A)$ is a discrete closed subset of $\R$, a fact which was previously observed by Arnoux and Yoccoz \cite{ArnouxYoccoz81}. Consequently, this spectrum has a least element
\begin{equation}\label{eqn:leastDilatation}
\least(A) \colonequals \inf \left\{\spec(A)\right\}.
\end{equation}
For example, $\spec(\Mod(S))$ may be thought of as the length spectrum of closed Teichm\"uller geodesics in the moduli space of $S$, and $\least(\Mod(S))$ is the length of the shortest such geodesic.

While explicit calculations of $\least(\Mod(S))$ have only been made in a few low-genus examples (see, e.g., \cite{LanneauThiffeault09} or \cite{Hironaka09}), its asymptotic behavior has been understood for some time. For real-valued functions $f$ and $h$, we write $f \asymp h$ if the quotient $f(x)/h(x)$ is bounded between two positive numbers. For the closed surface $S_g$ of genus $g$, Penner \cite{Penner91} has shown that $\least(\Mod(S_g)) \asymp 1/g$. In particular, by increasing the genus, it is possible to find pseudo-Anosov elements of $\Mod(S_g)$ with dilatation arbitrarily close to $1$.

Farb, Leininger, and Margalit \cite{FarbLeiningerMargalit08} have studied the spectrum of dilatations in the Torelli group $\mathcal{I}_g$, which is the subgroup of $\Mod(S_g)$ consisting of mapping classes that act trivially on $H
_1(S_g;\Z)$. Contrasting Penner's result, they proved that $\least(\mathcal{I}_g)$ is universally bounded between $0.197$ and $4.127$. Since these bounds are independent of genus, this shows that $L(\mathcal{I}_g) \asymp 1$.

\subsubsection*{Least point-pushing dilatations}

In light of these results, it is natural to consider the asymptotics of $\least(\pp(S_g))$ and, more generally, of $\least(\pp(S_{g,n}))$. Since point-pushing homeomorphisms act trivially on homology, it is apparent from \cite{FarbLeiningerMargalit08} that $\least(\pp(S_g))$ is universally bounded below away from zero. Furthermore, since the self-intersection number of any filling curve $\gamma\subset S_{g,n}$ satisfies
\begin{equation}\label{eqn:EulerCharacteristicBound}
\intnum(\gamma) \geq -\chi(S_{g,n}) = 2g+n-2
\end{equation}
one might be tempted to invoke the observation \eqref{eqn:toInfinity} and conclude that $\least(\pp(S_{g,n}))$ tends to infinity with both $g$ and $n$. While this reasoning is invalid, because Ivanov's compactness property only applies to one surface at a time, the conclusion that $\least(\pp(S_{g,n}))\to\infty$ does follow from Theorem~\ref{thm:generalBounds} together with \eqref{eqn:EulerCharacteristicBound}:

\begin{cor}\label{cor:boundOnLeast}
For any surface $S_{g,n}$ satisfying $3g+3>n$, we have
\[\least(\pp(S_{g,n})) \geq \tfrac{1}{5}\log(2g+n-1).\]
In particular $\least(\pp(S_{g,n}))$ tends to infinity with both $g$ and $n$.
\end{cor}
\begin{proof}
We simply note that $\least(\pp(S_{g,n}))$ is realized by a primitive filling curve $\gamma\subset S_{g,n}$.
\end{proof}

Corollary~\ref{cor:boundOnLeast} proves that the least dilatation in $\pp(S_g)$ exhibits drastically different behavior than that for the larger Torelli group---any point-pushing pseudo-Anosov on a high-genus surface must have large dilatation. Interestingly, Corollary~\ref{cor:boundOnLeast} also has the following implication regarding the spectrum of all pseudo-Anosov dilatations of point-pushing homeomorphisms on all surfaces.

\begin{cor}
The infinite union $\bigcup_{3g+n>3} \spec(\pp(S_{g,n}))$ is a discrete closed subset of $\R$.
\end{cor}

This is a marked contrast to the situation for the full mapping class group: although each spectrum $\spec(\Mod(S_{g,n}))$ is discrete, by looking at powers of pseudo-Anosovs $f\in\Mod(S_g)$ with arbitrarily small dilatation, we see that the subspectrum $\bigcup_{g}\spec(\Mod(S_{g}))$ of all pseudo-Anosov dilatations on all closed surfaces is in fact dense in $[0,\infty)$. 

In the case of closed surfaces, we also establish an upper bound on $\least(\pp(S_g))$.

\begin{mythm}\label{thm:leastDilatationFixedGenus}
For the closed surface $S_g$ of genus $g\geq 2$, we have
\[ \tfrac{1}{5}\log(2g) \leq \least(\pp(S_g)) < g\log(11).\]
\end{mythm}

\subsubsection*{Least dilatations and self-intersection number}

We now consider the dependence of least dilatation on self-intersection number. For a nonnegative integer $k$, we define the subset
\begin{equation*}
\pp_k(S) \colonequals \left\{\push(\gamma)\;\middle\vert\; \text{$\gamma\in \pi_1(S)$ with $\intnum(\gamma)=k$}\right\}
\end{equation*}
of point-pushing homeomorphisms coming from pushing curves with self-intersection number $k$. Refining our investigation to this stratification of $\pp(S)$ leads to the following result, which completely describes the asymptotic dependence of least dilatation on self-intersection number.

\begin{mythm}\label{thm:leastDilatationVaryN}
Let $S_g$ be a closed surface of genus $g\geq 3$. For any integer $k \geq 3g-1$, we have that
\[\tfrac{1}{5}\log(k+1) \leq \least(\pp_k(S_g)) < \log(k) + g\log(11).\]
In particular, for a closed surface $S$ of genus at least $3$, this shows that $\least(\pp_k(S)) \asymp \log(k)$.
\end{mythm}

\subsubsection*{Outline}
The theory of train tracks provides a natural means of calculating pseudo-Anosov dilatations. This perspective is investigated in \S\ref{sec:trainTracks}, where we describe a completely straightforward procedure to construct an invariant ``pretrack'' for any point-pushing homeomorphism; see Proposition~\ref{prop:invariantPretrack}. However, this approach is only partially successful because our methods do not produce a train track in general. Nevertheless, the pretrack $\pretrack{\gamma}$ is significant because it comes with an explicit incidence matrix $M_\gamma$ that depends only on the combinatorial structure of $\gamma$. In \S\ref{sec:biggest}, we use this matrix to establish the general upper bound in Theorem~\ref{thm:generalBounds}. In the case that $\pretrack{\gamma}$ is an actual train track, it provides a direct method for calculating the dilatation $\pushdil{\gamma}$. This framework is used in \S\ref{sec:upperbounds} to analyze concrete examples and prove the upper bounds on least dilatations in Theorems~\ref{thm:leastDilatationFixedGenus} and \ref{thm:leastDilatationVaryN}.

The pretrack $\pretrack{\gamma}$ is not able to provide a general lower bound on dilatation; see Remarks~\ref{rem:thePretrack} and \ref{rem:whyTracks}. Thus, in contrast to the our other calculation-based results, the lower bound in Theorem~\ref{thm:generalBounds} presents the primary theoretical difficulty. We prove this inequality in \S\ref{sec:lowerBound} by analyzing the action of $\push(\gamma)$ on simple closed curves and counting the exponential growth rate of intersection numbers. The technique is to lift to the universal cover and control the images of paths by studying the motion of the marked points.

\subsubsection*{Acknowledgments}
The author would like to thank his advisor, Benson Farb, for suggesting this project, for his guidance and insights, and for his constant encouragement and enthusiasm. The author is also grateful to Anna Marie Bohmann and Justin Malestein for their helpful comments on an earlier draft of this paper. Thanks also to the referee for their thorough comments and insightful recommendations.

\section{A lower bound on dilatation}\label{sec:lowerBound}

Our first objective is to establish a general lower bound on the dilatation of a point-pushing pseudo-Anosov map. The train track approach developed in \S\ref{sec:trainTracks} is inadequate for this purpose because it does not produce a train track in general, but only a pretrack. We instead estimate dilatations by examining the action on simple closed curves. Our primary tool in this endeavor is geometric intersection number.

\begin{defn}[Intersection number]\label{defn:intNum}
Let $a$ and $b$ be essential simple closed curves on a marked surface $(S,p)$, and let $\alpha$ and $\beta$ denote their respective isotopy classes in $(S,p)$. The \emph{geometric intersection number} of $a$ and $b$ is then defined as
\[ \intnum(a,b) = \inf \left\{\card{a'\cap b'} :   a'\in\alpha,~ b'\in\beta \right\}.\]
\end{defn}

We emphasize that isotopies of the marked surface $(S,p)$ are required to fix the basepoint $p$, and that essential simple closed curves in $(S,p)$ are the same as in $S\setminus\{p\}$. In these regards, the basepoint plays a similar role as a puncture. The connection between intersection number and pseudo-Anosov dilatation is made precise by the following theorem of Thurston; for a proof, see \cite[Theorem 12.2]{FLP79}.

\begin{thm}[Thurston]\label{thm:pAIntersectionNumber}
Let $f\in\Mod(S,p)$ be pseudo-Anosov with dilatation $\dilat{f} > 1$. For any two essential simple closed curves $a,b\subset (S,p)$, there is a constant $c\in (0,\infty)$ for which
\[\lim_{k\to \infty} \frac{\intnum(f^k(a),b)}{\lambda_f^k} = c.\]
\end{thm}

\begin{assm}\label{assm:Assumptions}
For the remainder of this section, we fix a genus $g$ surface $S = S_{g,n}$ with $n\geq 0$ punctures that satisfies $3g+n>3$. We fix a complete, finite-area hyperbolic metric on $S$ and a filling curve $\gamma \subset S$ that represents a primitive conjugacy class of $\pi_1(S)$; \emph{primitive} here means that $\gamma$ cannot be written as a power $\gamma=\mu^m$ in $\pi_1(S)$ for any $\abs{m}>1$. Upon adjusting $\gamma$ by a homotopy, we may assume that $\gamma$ is geodesic; this geodesic representative realizes the minimum self-intersection number $\intnum(\gamma)> 0$ in Definition~\ref{defn:selfIntNum}. Lastly, we fix an essential, geodesic, simple closed curve $\baseloop\subset S$. Since $\gamma$ is filling, the curve $\baseloop$ necessarily intersects $\gamma$ nontrivially.
\end{assm}

For any two basepoints $q,q'\in S$, the mapping class groups $\Mod(S,q)$ and $\Mod(S,q')$ are naturally isomorphic via an isomorphism that preserves pseudo-Anosov dilatation. Therefore, we are free to choose a basepoint $p\in \gamma$ which is not a self-intersection point of $\gamma$ and such that $p\notin\alpha$. With this basepoint, $\alpha$ becomes an essential simple closed curve in $(S,p)$. After parameterizing $\gamma\colon[0,1]\to S$ so that $\gamma(0) = p = \gamma(1)$, we obtain the point-pushing pseudo-Anosov $\push(\gamma)\in\Mod(S,p)$; our goal is to relate its dilatation $\pushdil{\gamma}$ to the self-intersection number $\intnum(\gamma)$. The hyperbolic metric gives a locally-isometric universal covering $\pi\colon\hyp\to S$, and we fix a preimage $\basepoint\in\pi\inv(p)$ to serve as the basepoint of $\hyp$. This defines an isometric action by deck transformations of the fundamental group $G = \pi_1(S,p)$ on $\hyp$. 

\paragraph{Strategy.} 
Our proof of the lower bound now proceeds in several steps. We first review the definition of $\push$ and build a representative point-pushing homeomorphism $\varphi_\gamma$ in the mapping class $\push(\gamma)$. The ultimate goal is then to study the images of $\baseloop$ under iteration by  $\varphi_\gamma$ and to count their intersection numbers with other curves on the marked surface $(S,p)$. Although it is relatively easy to describe these iterates using, for instance, the train track theory developed in \S\ref{sec:trainTracks} and \S\ref{sec:loopImages}, such representative curves do not aid in calculating intersection numbers because they need not realize the infimum in Definition~\ref{defn:intNum}; see Remark~\ref{rem:overCounting}.

To get around this difficulty, we lift everything to the universal cover $\hyp$ where it will be easier to understand the structure of $\varphi_\gamma^k(\baseloop)$. As discussed in \S\ref{sec:weaving}, the simple closed curves $\varphi_\gamma^k(\baseloop)$ lift to infinite paths $\liftimage{k}$ in $\hyp$, and the process of point-pushing on $S$ lifts to a procedure that we call ``weaving'' in hyperbolic space. In this setting, our goal is to study the paths $\liftimage{k}$ obtained by weaving and to relate their complexity to the self-intersection number $\intnum(\gamma)$.

Making these ideas precise involves many technical tools that we develop over the next several subsections. In \S\ref{sec:divergence} we introduce a tree $\tree$ that will serve as a sort of coordinate system for $\hyp$, and in \S\ref{sec:chainsAndPaths} we develop the technical devices that will help us navigate through $\tree$. As we will see in Observation~\ref{obs:countingIntersections}, bounding the intersection numbers $\intnum(\varphi_\gamma^k(\baseloop),\beta)$ on $(S,p)$ roughly translates into showing that all paths in the isotopy class of $\liftimage{k}$ must cross many edges of $\tree$.

To prove that this is the case, we develop the notion of a constraint on $\liftimage{k}$; this is essentially a marked point in $\hyp$ that forces every path in the isotopy class of $\liftimage{k}$ to visit a particular vertex of $\tree$. The relevant machinery for working with constraints is developed in \S\ref{sec:pushing}; we then give a recursive construction in \S\ref{sec:pointsThatPush} that identifies exponentially many constraints. Finally, in \S\ref{sec:countingIntersections}, we count intersection numbers and establish a lower bound on the dilatation $\pushdil{\gamma}$.

\subsection{Setting the stage: weaving in hyperbolic space}\label{sec:weaving}\label{point-pushingDef}

The first step in our proof is to translate the idea of point-pushing on the surface $S$ to its analogue in the universal cover $\hyp$; it is in this setting that that we will ultimately be able to understand the iterates of $\baseloop$ and count their intersection numbers with other curves. In this subsection we quickly review the construction of the point-pushing homomorphism $\push$ and choose a particular representative $\varphi_\gamma$ of the mapping class $\push(\gamma)$. We then lift this point-pushing homeomorphism to a ``weaving homeomorphism'' $\tilde{\varphi}_\gamma$ in $\hyp$ and introduce the relevant intuition and notation for understanding its structure.

A closed loop $\beta\colon[0,1]\to S$ based at $p = \beta(0) = \beta(1)$ defines an ``isotopy of maps'' $f_t\colon\{p\}\to S$ given by $f_t(p) = \beta(t)$; this may be extended to an isotopy $F_t\colon S\to S$ of the whole surface that effectively ``pushes'' the basepoint $p$ along the path $\beta$ and drags the rest of the surface along. At the end of this isotopy one obtains a point-pushing homeomorphism $\varphi_\beta \colonequals F_1$ that is well-defined up to isotopy in $(S,p)$. Furthermore, as shown by Birman \cite{Birman69,Birman74}, the corresponding isotopy class $[\varphi_\beta]$ depends only on the homotopy class of $\beta$, and the assignment $\beta\mapsto [\varphi_\beta]$ descends to an injective group homomorphism $\push\colon\pi_1(S,p)\to \Mod(S,p)$ called the \emph{point-pushing homomorphism}. We remark that, with our definition, $\push$ is technically an {\em anti}-homomorphism.

For those unfamiliar with point-pushing, it is instructive to consider a simple closed curve $\beta\subset S$, in which case $\push(\beta)$ is just the composition of two Dehn twists (in opposite directions) about the boundary curves of a tubular neighborhood of $\beta$; see \cite[\S4.2]{FarbMargalit}.

Returning now to the filling curve $\gamma \subset S$ chosen in Assumptions~\ref{assm:Assumptions}, we fix, once and for all, an isotopy $F_t\colon S\to S$ satisfying $F_t(p)=\gamma(t)$ and a point-pushing homeomorphism $\varphi_\gamma\colonequals F_1$ that represents the mapping class $\push(\gamma)\in\Mod(S,p)$. Since we are interested in iterating $\varphi_\gamma$, we extend the isotopy $F_t$, via the relation $F_{t+1} =F_t\circ F_1$, so as to be defined for all times $t\in \R$; with this convention we have that $\varphi_\gamma^k = F_k$ for all $k\in \Z$. The isotopy $F_t$ and homeomorphism $\varphi_\gamma$ will remain fixed for the duration of \S\ref{sec:lowerBound}. 

Lifting $F_t$ to the universal cover yields an isotopy $\tilde{F}_t\colon \hyp\to\hyp$ between the identity and a homeomorphism $\tilde{\varphi}_\gamma\colonequals \tilde{F}_1$; this map $\tilde{\varphi}_\gamma$ is the unique lift of the point-pushing map $\varphi_\gamma$ that sends the basepoint $\basepoint$ to its image under the deck transformation $\gamma\in G$. We think of $\tilde{\varphi}_\gamma$ as a ``weaving homeomorphism'' for reasons which will soon become evident. 

In attempts to avoid the confusing situation of ``moving'' the basepoint $p$ throughout the point-pushing procedure, we introduce the notion of a dynamic marked point. The basepoint $p\in S$ remains stationary while the isotopy $F_t$ instead pushes the marked point around $\gamma$. Thus the marked point's location at time $t$ is given by $F_t(p)$, and this location agrees with the basepoint if and only if $t$ is an integer. This concept of a dynamic marked point will be made precise in Definition~\ref{defn:markedPoint} below.

Throughout we will suppress the distinction between a path $\beta\colon[0,1]\to S$ and its image $\beta = \beta([0,1])\subset S$; paths that differ by a reparameterization will not be considered distinct. A \emph{lift} of a closed loop $\beta\colon\R/\Z \to S$ is any path $\tilde{\beta}\colon\R\to \hyp$ that cyclically covers $\beta$; if $\beta$ is a simple loop, then its lifts are exactly the connected components of $\pi\inv(\beta)$. We denote the set of lifts of our chosen filling curve $\gamma$ by 
\begin{equation}\label{eqn:liftGamma}
\liftgamma = \{\tilde{\gamma} \subset\hyp \mid \tilde{\gamma}\colon\R\to \hyp \text{ covers }\gamma\colon \R/\Z\to S\}.
\end{equation}
Since $\gamma$ is a geodesic loop, the elements of $\liftgamma$ are infinite geodesic lines. We will use the following notation to discuss lifts of paths to $\hyp$.

\begin{notn}
If $\eta\colon[0,1]\to S$ is a path in $S$ starting at a point $x_0 = \eta(0)$, then for any preimage $x\in\pi\inv(x_0)$ we let $\lpath{x}{\eta}\colon[0,1]\to \hyp$ denote the unique path lift of $\eta$ starting at $x$. The terminal endpoint of this path will be denoted by $x\cdot\eta \colonequals \lpath{x}{\eta}(1)$. A deck transformation $h\in G$ acts on the set of such paths by changing the starting point: $h(\lpath{x}{\eta}) = \lpath{hx}{\eta}$. 

If we consider loops $\eta\subset S$ based at $x_0$, then the pairing $(x,\eta)\mapsto x\cdot \eta$ defines a right action of $\pi_1(S,x_0)$ on the set $\pi\inv(x_0)$, that is, $(x\cdot \eta)\cdot\mu = x\cdot(\eta\mu)$ for $\eta,\mu\in\pi_1(S,x_0)$. This right action commutes with the left action of $h\in G$ in the sense that $h(x\cdot \eta) = (hx)\cdot \eta$. If $\eta\subset S$ is a loop without a natural basepoint and $x\in\pi\inv(\eta)$ is any point in its preimage, then $\lpath{x}{\eta}$ is understood to mean $\lpath{x}{\hat{\eta}}$, where $\hat{\eta}\colon[0,1]\to S$ is any parameterization of $\eta$ based at $\pi(x)$.
\end{notn}

On the surface, the isotopy $F_t$ pushes the marked point $p$ along the curve $\gamma$. Therefore, in the universal cover, we consider each preimage $x\in \pi\inv(p)$ to be a marked point of $\hyp$ and find that $\tilde{F}_t$ has the effect of pushing $x$ along the path $\lpath{x}{\gamma}$ to the point $x\cdot \gamma$. Since $\gamma$ has self-intersections, each of its lifts $l\in \liftgamma$ intersects infinitely many other lifts. Thus the full preimage $\pi\inv(\gamma)$ is an infinite grid of intersecting geodesics, and the isotopy $\tilde{F}_t$ simultaneously pushes all of the marked points along their corresponding lifts in an intertwining pattern that resembles weaving on an infinite loom.

We are concerned with the images of our simple closed curve $\baseloop\subset (S,p)$ under iteration by $\varphi_\gamma$. Since $\intnum(\baseloop,\gamma) > 0$, each lift of $\baseloop$ is a geodesic line that intersects infinitely many lifts of $\gamma$. We may therefore choose a particular lift $\tilde{\baseloop}$ that intersects the segment $\lpath{\basepoint}{\gamma}$, where $\basepoint$ is the chosen basepoint of $\hyp$. We denote this particular lift by $\baselift$, and denote its image under $\tilde{F}_t$ by
\begin{equation*}
\liftimage{t} = \tilde{F}_t(\baselift).
\end{equation*}

Recall that the closed curves $\varphi_\gamma^k(\baseloop)$ naturally live in the marked surface $(S,p)$, meaning that their isotopy classes are determined up to isotopies that fix the basepoint $p\in S$ (in the unmarked surface $S$, the curves $\varphi_\gamma^k(\baseloop)$ are all isotopic to $\baseloop$). Analogously, the infinite paths $\liftimage{k}=\tilde{\varphi}_\gamma^k(\baselift)$ naturally live in the \emph{marked space} $(\hyp,\pi\inv(p))$---here all homeomorphisms and isotopies are required to preserve the set $\pi\inv(p)$. Equivalently, one may think of the punctured space $\hyp\setminus\pi\inv(p)$. Notice that $\tilde{F}_t$ is not an isotopy of $(\hyp,\pi\inv(p))$, but its terminal homeomorphism $\tilde{\varphi}_\gamma= \tilde{F}_1$ is a homeomorphism of $(\hyp,\pi\inv(p))$. Since $\tilde{\varphi}_\gamma^k$ and $\baselift$ cover $\varphi_\gamma^k$ and $\baseloop$, respectively, it follows that $\pi(\liftimage{k}) = \varphi_\gamma^k(\baseloop)$; therefore we may use the paths $\liftimage{k}\subset (\hyp,\pi\inv(p))$ to study the iterates $\varphi_\gamma^k(\baseloop)\subset (S,p)$ of $\baseloop$.

\begin{figure}
\centering
\subfigure[A pushing curve $\gamma$ with $\intnum(\gamma)=2$, and a simple closed curve $\baseloop\subset S_3$.]
{
\labellist
\small\hair 2pt
\pinlabel $p$ [tr] at 87 53
\pinlabel $\gamma$ [bl] <0pt,-2pt> at 130 96
\pinlabel $\alpha$ [bl] <0.5pt,0pt> at 127 30
\pinlabel $q_1$ [bl] <-2pt,1pt> at 149 66
\pinlabel $q_2$ [r] <1pt,0pt> at 71 73
\endlabellist
\includegraphics[scale=1.0]{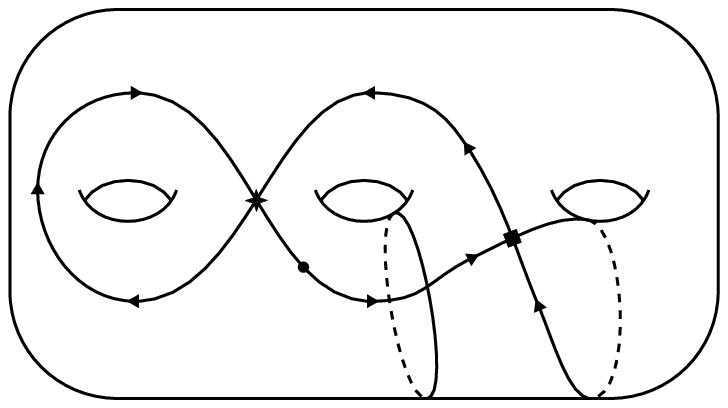}
\label{fig:simpleWeaveSetup}
}
\subfigure[Push $p$ once around $\gamma$ to obtain $\varphi_\gamma(\baseloop)$.]
{
\labellist\small\hair 2pt
\pinlabel {$\varphi_\gamma(\alpha)$} [bl] at 132 98
\endlabellist
\includegraphics[scale=1.0]{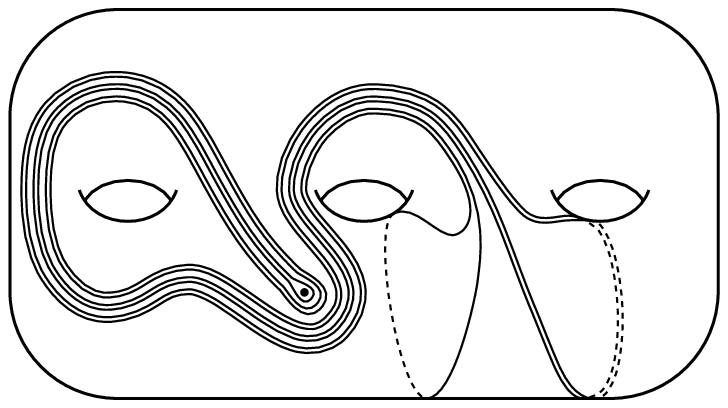}
\label{fig:simpleWeavePush1}
}
\subfigure[Lift to $\hyp$: The marked points $\tilde{p}$ simultaneously slide along the lines $\tilde{\gamma}$ and push the initial lift $\baselift=\tilde{\baseloop}$. The resulting path is $\tilde{\varphi}_\gamma(\baselift)$.]
{
\labellist
\small\hair 2pt
\pinlabel {$\basepoint$} [tr] at 42 51
\pinlabel {$\tilde{\gamma}$} [l] at 160 146
\pinlabel {$\tilde{p}$} [l] at 160 130
\pinlabel {$\tilde{q_1}$} [l] at 200 146
\pinlabel {$\tilde{q_2}$} [l] at 200 130
\pinlabel {$\baselift$} [r] at 42 78
\pinlabel {$\tilde{\varphi}_\gamma(\baselift)$} [bl] at 337 63
\endlabellist
\includegraphics[scale=1.0]{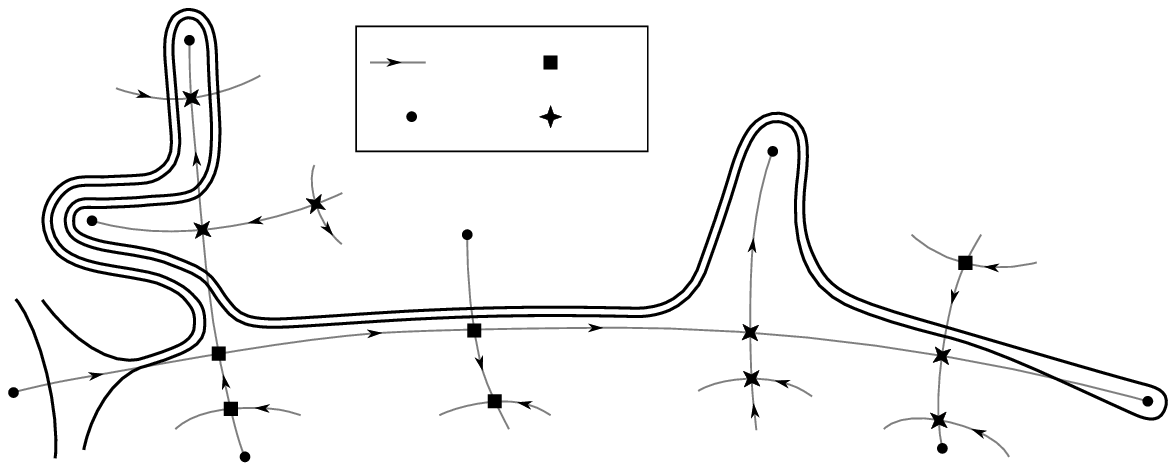}
\label{fig:weavePushUpstairs}
}
\caption{An example of point-pushing on $S_3$ and weaving in $\hyp$.}
\label{fig:simpleWeaveExample}
\end{figure}

Given explicit choices of $\baseloop$ and $\gamma$, it is relatively straightforward to determine $\varphi_\gamma(\baseloop)$ and $\tilde{\varphi}_\gamma(\baselift)$: simply move the marked points along their respective paths and push $\baseloop$ or $\baselift$ along. An illustrative example is depicted in Figure~\ref{fig:simpleWeaveExample}. Here we consider a (non-filling) pushing curve $\gamma\subset S_3$ with two self-intersection points $q_i$ and a simple closed curve $\baseloop$ that intersects $\gamma$ exactly once. As one can check, pushing the marked point once around $\gamma$ transforms $\baseloop$ into the the curve $\varphi_\gamma(\baseloop)$ shown in Figure~\ref{fig:simpleWeavePush1}. In the universal cover, all of the marked points $\tilde{p}\in\pi\inv(p)$ flow simultaneously along the lines $\tilde{\gamma}$ comprising the grid $\pi\inv(\gamma)$. As they travel, some of these points interact with the lift $\baselift$ and
drag it along with them. The resulting path $\tilde{\varphi}_\gamma(\baselift)$, as shown in Figure~\ref{fig:weavePushUpstairs}, is forced to bend around these marked points.

This example exhibits the following key features. The curve $\varphi_\gamma(\baseloop)\subset (S, p)$ is already quite complicated---the marked point is pushing nine strands of $\varphi_\gamma(\baseloop)$, and it is difficult to keep track of which strand is which and how it got there. The structure is more transparent when one unwinds this picture in the universal cover; here we see that these nine strands come from different parts of $\hyp$ and are being pushed in diverging directions by multiple marked points.

This essential observation is the foundation of our entire argument: On the surface there is only one marked point, but in the universal cover there are many marked points pushing the path $\liftimage{t}$ in various directions. By identifying these marked points and keeping track of their locations, we will be able to quantify the complexity of $\varphi_\gamma^k(\baseloop)$, estimate intersection numbers, and establish a lower bound on dilatation. To formalize these ideas, we need a more careful definition of the dynamic marked points.

Intuitively, a dynamic marked point should be a point that moves. For instance, a marked point located at $x\in\hyp$ at time $t$ might move to a new location $y\in \hyp$ at some later time $t' > t$. We need a consistent naming scheme so that a given dynamic marked point has the same name no matter where it is. Clearly the necessary data is captured by a continuous function $\R\to\hyp$, i.e., a path, which records the marked point's position at any given time. In our situation, the relevant marked points are those coming from lifts of the basepoint $p\in S$, and we have a fixed isotopy $\tilde{F}_t\colon \hyp\to\hyp$ describing exactly how these points move.

\begin{defn}[Dynamic marked point]\label{defn:markedPoint}
A \emph{dynamic marked point} in $\hyp$ is a function $\rho\colon \R \to \hyp$ of the form $\rho(t) = \tilde{F}_t(\tilde{p})$, where $\tilde{p}\in\pi\inv(p)$ is any lift of the basepoint. The set of marked points is denoted by $\marked$. The \emph{location} of a marked point $\rho\in\marked$ at a time $t\in \R$ is simply its value $\rho(t)$, and the set of all these locations is denoted by $\marked_t = \{\rho(t) \mid \rho\in\marked\}$.
\end{defn}

While dynamic marked points are ostensibly geodesic paths in $\hyp$, we prefer to think of them as points that physically move with respect to time. To recall the path-like nature of  $\rho\in \marked$, we may simply consider its image $\rho(\R)$, which is a geodesic in the set $\liftgamma$. At each time $k\in \Z$ the set $\marked_k$ of locations is exactly equal to $\pi\inv(p)$; thus the map $\rho\mapsto\rho(k)$ provides a natural bijection $\marked\cong \pi\inv(p)$. However, we stress that all of the bijections obtained in this way are distinct. 

In \S\S\ref{sec:pushing}--\ref{sec:pointsThatPush} we will give precise meaning to the concept of a dynamic marked point that ``pushes $\liftimage{t}$'' and find many points that have this pushing property. In order to do so, it will be important to keep track of the locations of the dynamic marked points and their relative positions to each other. To this end, we construct a tree that will serve as a sort of coordinate system for $\hyp$. This tree and its properties are the business of the next subsection.

\subsection{Uniform divergence in the coordinate tree $\tree$}\label{sec:divergence}

Choose a pants decomposition $\pants$ of $S$, that is, a maximal collection $\pants = \{\seam\}$ of homotopically distinct, disjoint, essential, simple closed curves $\seam\subset S$. Any such collection contains exactly $3g-3+n>0$ curves; in particular, our assumption on $S=S_{g,n}$ ensures that $\pants$ is nonempty. By a slight abuse of notation, the subset $\cup_i\seam\subset S$ will also be denoted by $\pants$. The curves $\seam$ may be chosen to be closed geodesics for the hyperbolic metric on $S$ and, after adjusting the choice of basepoint $p\in \gamma$ if necessary, we may furthermore assume that $p\notin \pants$ so that each $\seam$ defines an essential simple closed curve in the marked surface $(S,p)$. 

The connected components of $\liftpants \colonequals \pi\inv(\pants)$ are geodesic lines that cut $\hyp$ into infinitely many components. Dual to this decomposition of $\hyp$ there is a tree $\tree = \tree_\pants$ whose vertices $v\in\vertices$ are the connected components of $\hyp\setminus\liftpants$ and whose oriented edges $e\in \edges$ are ordered pairs $e=(v_0,v_1)$ of vertices corresponding to regions that share a boundary component.\footnote{Note that $\tree$ is just the Bass--Serre tree for the graph of groups description of $\pi_1(S)$ corresponding to the pants decomposition $\pants$ of $S$; see, for example, \cite{ScottWall79}.} This same edge with the reverse orientation will be denoted by $\bar{e} = (v_1,v_0)$. The unoriented edges of $\tree$ are in bijective correspondence with the components of $\liftpants$; accordingly, we will often suppress the distinction between edges in $\tree$ and these geodesics in $\hyp$. 

An oriented \emph{edge path} in $\tree$ is a (possibly bi-infinite) ordered list $(e_1,\dotsc,e_n)$ of oriented edges $e_i\in \edges$ satisfying the condition that the terminal vertex of $e_i$ is the initial vertex of $e_{i+1}$. An edge path is \emph{geodesic} if it is without backtracking, that is, if $e_{i+1} \neq \bar{e_i}$ for each $i$. Applying the Jordan curve theorem to a geodesic in $\liftpants$, we see that each edge $e\in\edges$ separates $\tree$ into two connected components. Thus $\tree$ is in fact a tree, meaning that there is a unique geodesic between any two vertices. The \emph{length} of a finite edge path $e=(e_1,\dotsc,e_n)$ is denoted by $\len(e) = n$; this defines a path metric on $\tree$. The action of $G$ on $\hyp$ descends to an isometric action on $\tree$, and there is a natural, $G$-equivariant projection $\proj\colon\hyp\to \tree$ that collapses the components of $\hyp\setminus\liftpants$ and $\liftpants$ to vertices and edges, respectively. Since a non-elliptic isometry of $\hyp$ can preserve at most one geodesic line, we see that the fixed set in $\tree$ of a nontrivial deck transformation $h\in G$ contains at most a single edge.

Any oriented path $\mu\subset \hyp$ that is transverse to $\liftpants$ with endpoints in $\hyp\setminus\liftpants$ projects to an edge path in $\tree$. If $\mu$ is a geodesic path, then so is $\proj(\mu)$, and we use $\len(\mu)$ to denote $\len(\proj(\mu))$. The $\tree$-length of a loop $\beta\subset S$ is similarly defined by $\len(\beta) = \len(\proj(\lpath{x}{\beta}))$, where $x$ is any point in $\pi\inv(\beta)\setminus\liftpants$. With this notation, $\len(\gamma) = \sum_{i}\intnum(\gamma,\seam)$ is the number of times the loop $\gamma$ crosses the curves in $\pants$. Since $\gamma$ fills $S$, we have that $\len(\gamma)\geq \card{\pants} \geq 1$.

Geodesic lines $l,l'\subset\hyp$ with distinct endpoints in $\partial\hyp$ necessarily diverge when projected to $\tree$ in the sense that they determine distinct edge paths. Nevertheless, these projections may agree along an arbitrarily long edge path. The following crucial lemma shows that, for geodesics in an equivariant family, the divergence in $\tree$ happens uniformly quickly. In \S\ref{sec:chainsAndPaths}, we will apply this to lifts of the geodesic $\gamma$.

\begin{lem}[Uniform divergence]\label{lem:boundOnDivergence}
Let $l,l'\subset\hyp$ be two distinct lifts of a closed geodesic $\beta\subset S$, and let $X = \proj(l)\cap \proj(l')$ be the intersection of their projections to $\tree$. Then $X$ is a (possibly empty or degenerate) geodesic edge path of length $\len(X)\leq \len(\beta)+1$. In the case that $\len(\beta)\leq 2$, this bound may be improved to $\len(X) \leq \len(\beta)$.
\end{lem}
\begin{proof}
If $\beta\in\pants$, then $l$ and $l'$ correspond to distinct edges of $\tree$ and we have $X = \emptyset$. Therefore, we may assume that $\beta\notin\pants$, in which case $\len(\beta)=\sum\intnum(\beta,\seam)\geq 1$ because $\pants$ is a pants decomposition of $S$. It follows that $\proj(l)$ and $\proj(l')$ are both bi-infinite geodesic edge paths in $\tree$. Their intersection $X$ is clearly a geodesic edge path as well.

Suppose, on the contrary, that $\len(X) \geq n+2$, where $n = \len(\beta)$. Then $X$ contains a subpath of the form $(e_0,e_1,\dotsc,e_n,e_{n+1})$. Writing $e_j = (v_j,v_{j+1})$, we choose a generic point $x_0\in l$ that is contained in $v_1$ and does not project to a self-intersection point of $\beta$ on $S$. Since points in the $G$-orbit of $x_0$ occur with spacing $\len(\beta)=n$ along both $l$ and $l'$, we may find orbit points $x_1\in l$ and $x'\in l'$ with $x_1\in v_{n+1}$ and $x'\in v_{i}$ for some $1\leq i\leq n$. The situation is depicted in Figure~\ref{fig:boundOnTracking}. Letting $f_0,f_1\in G$ be the deck transformations defined by $f_j(x')= x_j$, we see that $h=f_1f_0\inv$ and the hyperbolic translation along $l$ sending $x_0$ to $x_1$ both agree at the point $x_0$---therefore they are equal. Since $f_j\inv(l) = l' \neq l$, the isometry $f_j$ does not preserve the axis $l$ of $h$ and therefore cannot commute with $h$: $f_jh \neq hf_j$. 

\begin{figure}
\centering
\labellist
\small\hair 3pt
\pinlabel {$l$} [tr] at 62 398
\pinlabel {$x_0$} [tr] <3pt,-1pt> at 110 388
\pinlabel {$x_1$} [tl] <-1pt,-1pt> at 481 397
\pinlabel {$l'$} [br] <1pt,-1pt> at 68 283
\pinlabel {$x'$} [bl] <-2pt,1pt> at 243 323
\pinlabel {$e_0$} [l] at 75 257
\pinlabel {$e_1$} [l] at 131 276
\pinlabel {$e_{i-1}$} [l] at 220 288
\pinlabel {$e_{i}$} [l] at 278 293
\pinlabel {$e_n$} [r] <2pt,0pt> at 470 292
\pinlabel {$e_{n+1}$} [l] <1pt,-2pt> at 522 289
\pinlabel {$f_0$} [t] <-1pt,0pt> at 157 352
\pinlabel {$f_1$} [tl] <-4pt,-1pt> at 441 362
\pinlabel {$f_0(x_1)$} [t] <2pt,-1pt> at 342 429
\pinlabel {$f_1(x_0)$} [b] <1pt,1pt> at 344 469
\endlabellist
\includegraphics[width=4in]{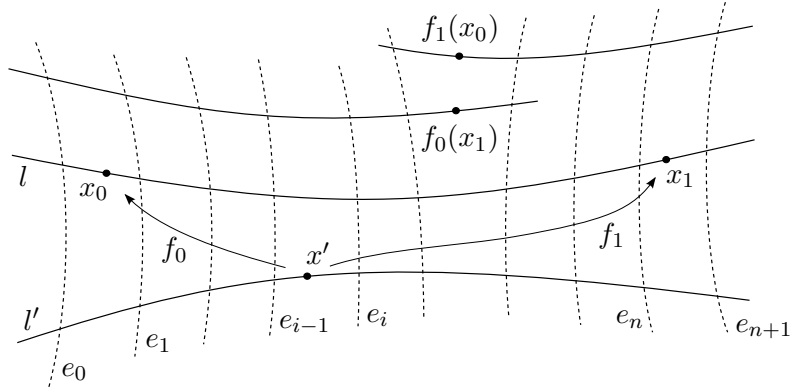}
\caption{The two edge paths $\proj(l)$ and $\proj(l')$ cannot overlap too much, for otherwise the deck transformation sending $f_0(x_1)$ to $f_1(x_0)$ would fix two edges of $\tree$.}
\label{fig:boundOnTracking}
\end{figure}

An orientation on $\beta\colon[0,1]\to S$ lifts to natural orientations on $l$ and $l'$ which in turn induce orientations on $X$. We may assume that $l$ induces the orientation $(e_0,\dotsc,e_{n+1})$ on $X$, so that the path $\lpath{x_0}{\beta}$ crosses the edges $e_1$, $e_2$, \dots, $e_n$ in order. 

\textbf{Case 1:} \emph{The geodesic $l'$ induces the opposite orientation on $X$.} In this case, $\lpath{x'}{\beta}$ projects to a (finite) edge path of the form $(\overline{e_{i-1}},\overline{e_{i-2}},\dotsc,\overline{e_0},\dotsc)$. Consequently, since $f_0$ maps $\lpath{x'}{\beta}$ equivariantly onto $\lpath{x_0}{\beta}$, we see that $f_0((\overline{e_{i-1}},\dotsc,\overline{e_{0}})) = (e_1,\dotsc,e_i)$. In particular, we have $f_0(\overline{e_0}) = e_i$. Reversing direction and considering the first edges crossed by the paths $\lpath{x_0}{\bar{\beta}}$ and $\lpath{x'}{\bar{\beta}}$, where $\bar{\beta}$ denotes $\beta$ with the opposite orientation, we similarly find that $f_0(e_i) = \overline{e_0}$. As $G = \pi_1(S)$ is torsion-free, this implies that $f_0^2$ is a nontrivial deck transformation preserving the distinct edges $\overline{e_0}$ and $e_i$---a contradiction.

\textbf{Case 2:} \emph{The geodesics $l$ and $l'$ induce the same orientation on $X$.} We now have that $\lpath{x'}{\beta}$ and $\lpath{x'}{\bar{\beta}}$ project to edge paths of the form $(e_i,e_{i+1},\dotsc)$ and $(\overline{e_{i-1}},\overline{e_{i-2}},\dotsc)$, respectively. As $f_0$ maps $\lpath{x'}{\beta}$ equivariantly onto $\lpath{x_0}{\beta}$, we see that $f_0((e_{i},\dotsc,e_{n+1})) = (e_1,\dotsc, e_{n+2-i})$. In particular, $f_0((e_{n},e_{n+1})) = (e_{n+1-i},e_{n+2-i})$. Similarly, we have $f_1((\overline{e_{i-1}},\dots,\overline{e_{0}})) = (\overline{e_{n}},\dots,\overline{e_{n+1-i}})$, so that $f_1((\overline{e_1},\overline{e_0}))= (\overline{e_{n+2-i}},\overline{e_{n+1-i}})$. Therefore
\[f_0h((e_0,e_1)) = f_0((e_{n}, e_{n+1})) = (e_{n+1-i},e_{n+2-i}) = f_1((e_0,e_1)).\]
Since $f_1= hf_0$, this shows that the distinct deck transformations $f_0 h$ and $hf_0$ both send $e_0\mapsto e_{n+1-i}$ and $e_1\mapsto e_{n+2-i}$, which is impossible.

It remains to prove the stronger inequality in the case that $n = \len(\beta)\leq 2$. We proceed as above assuming only that $X$ contains a subpath of the form $(e_0,\dotsc,e_n)$; the setup and notation are otherwise unchanged. The arguments in Case 1 are still valid because they do not involve the edge $e_{n+1}$. Thus it suffices to assume that $l$ and $l'$ induce the same orientation on $X$. The case $x'\in v_1$ then yields an immediate contradiction because $f_0$ cannot preserve both $e_0$ and $e_1$. The remaining possibility $x'\notin v_1$ necessitates $n=2$ and $x'\in v_2$, in which case we find that 
$f_1$ sends $(\overline{e_1},\overline{e_0})$ to $(\overline{e_2},\overline{e_1})$ and that $f_0\inv$ sends $(e_0,e_1)$ to $(e_1,e_2)$. It now follows that $f_1f_0((e_1,e_2)) = f_1((e_0,e_1)) = (e_1,e_2)$, contradicting the fact that $f_1f_0$ is nontrivial.
\end{proof}

\subsection{$\liftgamma$-chains and grid paths}\label{sec:chainsAndPaths}

As we flow the isotopy $\tilde{F}_t$, the dynamic marked points travel along lifts of our fixed geodesic $\gamma$ and pass each other at the intersections of these lines. Recall from (\ref{eqn:liftGamma}) that the set of all lifts of $\gamma$ is denoted by $\liftgamma$. In order to describe the locations and interactions of the dynamic marked points that ``push $\liftimage{t}$,'' we will need to consider paths in $\hyp$ that travel along geodesics in $\liftgamma$ and potentially turn at their intersections. These turning paths will lead us to the desired dynamic marked points and show that many such points exist. \

\begin{defn}[Grid path]\label{defn:gridPath}
A \emph{grid path} is a concatenation $\mu= \mu_1\dotsb\mu_n$ of oriented geodesic segments $\mu_i\subset l_i\in \liftgamma$ whose endpoints ``match up,'' that is, the terminal endpoint of $\mu_i$ is the initial endpoint of $\mu_{i+1}$. The $\mu_i$ are called the \emph{straight segments} of $\mu$, and we require that adjacent segments segments $\mu_i$ lie on distinct geodesics $l_i$.
\end{defn}

While the projection of a grid path may backtrack in $\tree$ when it turns at the junction of two straight segments, the extent of this backtracking is universally bounded by Lemma~\ref{lem:boundOnDivergence}. Therefore, by making the straight segments $\mu_i$ sufficiently long, we may effectively disregard any backtracking because it will be contained within a bounded neighborhood of the endpoints of the $\proj(\mu_i)$. This will ensure that the projection of each straight segment $\mu_i$ contributes definite progress in $\tree$. The following definitions and lemma make this precise.

\begin{defn}[$\liftgamma$-chain]\label{defn:chain}
A \emph{$\liftgamma$-chain} is an ordered tuple $(l_1,\dotsc,l_n)$ of
distinct geodesics $l_i\in \liftgamma$ which satisfy the property that $l_i$ and $l_j$ intersect if and only if $\abs{i-j} = 1$. (This is analogous to the familiar notion of a ``chain'' of simple closed curves on a surface.)
\end{defn}

\begin{defn}[Internal edge]\label{defn:internalEdge}
For a finite geodesic segment $\mu\subset \hyp$, an \emph{internal edge of $\proj(\mu)$} is simply an edge $e\in \edges$ whose removal separates $\proj(\mu)$ into two edge paths of length at least $\len(\gamma)+1$. In the case that $\len(\gamma)\leq 2$, we only require these pieces to have length at least $\len(\gamma)$. In either case, $\proj(\mu)$ contains internal edges provided that $\len(\mu)\geq 3\len(\gamma)$.
\end{defn}

\begin{lem}[Long grid paths]\label{lem:lazyChain}
Let $\mu = \mu_1\dotsb\mu_n$ be a grid path with endpoints $x,y\in \hyp$, and let $l_i\in \liftgamma$ be the geodesic containing $\mu_i$. Suppose that $\len(\mu_i) \geq 3\len(\gamma)$ for each $1<i<n$. Then the $\tree$-geodesic from $\proj(x)$ to $\proj(y)$ contains every internal edge of each projection $\proj(\mu_i)$. Furthermore, $(l_1,\dotsc,l_n)$ is a $\liftgamma$-chain.
\end{lem}
\begin{proof}
The condition $\len(\mu_i)\geq 3\len(\gamma)$ implies that $\proj(\mu_i)$ contains an internal edge  when $1 < i < n$. The first and last segments $\proj(\mu_1)$ and $\proj(\mu_n)$ need not have internal edges, but the claim applies if such edges do exist. Let $M$ denote the optimal bound guaranteed by Lemma~\ref{lem:boundOnDivergence}, so $M$ is either $\len(\gamma)+1$ or $\len(\gamma)$ depending on whether or not $\len(\gamma) > 2$. The edge path $\proj(\mu)$ connects $\proj(x)$ to $\proj(y)$ and would be a $\tree$-geodesic except for the fact that backtracking may occur when the geodesic segments $\proj(\mu_i)$ are concatenated. Upon removing all such backtracking by successively cancelling edge pairs $(\dotsc,e,\bar{e},\dotsc)$, we will obtain the desired $\tree$-geodesic. Every edge that is removed because of backtracking at the junction of $\mu_i$ with $\mu_{i+1}$ must be contained in both $\proj(\mu_i)$ and $\proj(\mu_{i+1})$. In light of Lemma~\ref{lem:boundOnDivergence}, it follows that each junction $\proj(\mu_i)\proj(\mu_{i+1})$ can result in at most $M$ cancellations. Since internal edges are, by definition, separated from these junctions by at least $M$ edges on either side, this shows that internal edges cannot cancel with edges from neighboring segments.

If $\proj(\mu_i)$ contains an edge that does not cancel with an edge from either neighboring segment, then it is impossible for edges from $\proj(\mu_{i-1})$ and $\proj(\mu_{i+1})$ to cancel with each other. Thus it is essential that \emph{all} of the segments $\proj(\mu_i)$, $1<i<n$, are long enough to contain internal edges, as this prevents cascading effects and ensures that cancellations only occur between neighboring segments. The first claim now follows from the above observation that such cancellations do not involve internal edges.

As for the second claim, it suffices to show that $\proj(l_i)$ and $\proj(l_j)$ are disjoint whenever $\abs{i-j}\geq2$. If this is not the case, there is a vertex $v\in \proj(l_i)\cap\proj(l_j)$ for some $i+1 \leq j-1$. Consider the shortened grid path $\mu_{i+1}\dotsb\mu_{j-1}$ from $x'$ to $y'$. Let $a$ and $b$ be internal edges of $\proj(\mu_{i+1})$ and $\proj(\mu_{j-1})$, respectively; such internal edges exist because $1 < i+1\leq j-1 < n$. Since $a$ is separated from $\proj(x')$ by at least $M$ edges, the fact that $\len(\proj(l_i)\cap \proj(\mu_{i+1}))\leq M$ implies $a\notin\proj(l_i)$. We similarly have $b\notin\proj(l_j)$. By inducting on $\abs{i-j}$, we may furthermore assume that $a\notin\proj(l_j)$ and $b\notin\proj(l_i)$; this is possible because the base case $\abs{i-j} = 2$ allows one to choose $a=b$. We now see that $\proj(l_i)\cup\proj(l_j)$ contains an edge path from $\proj(x')$ through $v$ to $\proj(y')$ that avoids both $a$ and $b$. Applying the second assertion to the grid path $\mu_{i+1}\dotsb\mu_{j-1}$ yields a contradiction.
\end{proof}

\subsection{The tools for weaving}\label{sec:pushing}

We return to the task of finding dynamic marked points that ``push'' $\liftimage{t}$, where $\liftimage{t} = \tilde{F}_t(\baselift)$ is the image of our initial lift $\baselift$ at time $t$. In this subsection we formalize this notion in terms of ``constraining points'' (Definition~\ref{defn:constrains}) and provide the necessary tools for working with these points. The sought-after points will be described explicitly in the next subsection.

These considerations involve infinite paths in $\hyp$ and their images under isotopies of $\hyp$. We are primarily concerned with lifts of isotopies of the surface $S$; any such isotopy moves points a uniformly bounded distance and, in particular, fixes the boundary at infinity $\partial\hyp$ pointwise. Therefore, we will only consider isotopies of $\hyp$ that fix $\partial\hyp$ pointwise: if $\mu\colon\R\to\hyp$ is path with two endpoints at infinity, then these endpoints remain fixed throughout all isotopies.

We henceforth assume that the pants decomposition $\pants$ is chosen to contain our simple closed curve $\baseloop\subset S$. In this case, $\baselift$ is a component of $\liftpants$ and corresponds to an edge of $\tree$; this edge $\proj(\baselift)$, together with its adjacent vertices, will be denoted by $\bedge\subset\tree$. 

Let $\Gstab\leq G$ be the stabilizer of $\bedge$ in $G$. This is a cyclic subgroup consisting of hyperbolic isometries that act by translation along the geodesic axis $\baselift\subset \hyp$. If $\mu\subset S$ is any simple closed curve isotopic to $\varphi_\gamma^k(\baseloop)=\pi(\liftimage{k})$ in the marked surface $(S,p)$, then $\mu$ has a particular lift $\tilde{\mu}\subset\hyp$ which is isotopic to $\liftimage{k}$ in $(\hyp,\pi\inv(p))$. This lift $\tilde{\mu}$ is characterized by having the same endpoints in $\partial \hyp$ as $\liftimage{k}$. Since the endpoints of $\liftimage{k}$ are the same as those of $\baselift$, we see that $\Gstab$ fixes the endpoints of $\tilde{\mu}$. Thus each $g\in \Gstab$ in fact preserves $\tilde{\mu}$ and acts as a translation along $\tilde{\mu}$ of the form $y\mapsto y\cdot \mu^n$ for some $n\in \Z$. Indeed, since $\mu$ and $\baseloop$ are isotopic in $S$, they determine the same conjugacy class in $\pi_1(S,p)$. Elements of this conjugacy class are in bijective correspondence with the lifts of $\baseloop$ and also with the lifts of $\mu$. The two lifts $\baselift$ and $\tilde{\mu}$ have the same endpoints at infinity and therefore correspond to the same element of $\pi_1(S,p)$; this element is a generator of $\Gstab$. 

The cyclic group $\Gstab$ acts on $\hyp$ on the left with quotient space $\Gstab\backslash \hyp$. In this quotient, the lift $\tilde{\mu}$ projects to a simple closed curve $\Gstab\backslash\tilde{\mu}$ that bijectively covers $\mu$; that is, the natural covering $\Gstab\backslash\hyp\to S$ restricts to a degree one cover $\Gstab\backslash\tilde{\mu}\to \mu$. Therefore, each intersection point $y\in \mu\cap \pants$ with a pants curve $\seam$ lifts to a \emph{unique} intersection point of the loop $\Gstab\backslash\tilde{\mu}$ with a lift of $\seam$ to $\Gstab\backslash \hyp$. Furthermore each such lift exactly corresponds to a $\Gstab$--orbit $\Gstab\tilde{\seam}\subset\hyp$ of geodesics in $\liftpants$. This has the following implication:

\begin{obs}\label{obs:countingIntersections}
Let $\mu\subset (S,p)$ be a simple closed curve that is isotopic to $\varphi_\gamma^k(\baseloop)$ in $(S,p)$, and let $\tilde{\mu}\subset(\hyp,\pi\inv(p))$ be the unique lift whose endpoints agree with those of $\liftimage{k}$. Assuming $\mu$ is transverse to $\pants$, the cardinality of $\mu\cap \pants$ is equal to the number of edges that the loop $\Gstab\backslash\proj(\tilde{\mu})$ crosses in the quotient graph $\Gstab\backslash \tree$. That is, $\card{\mu\cap \pants}$ is equal to the number of $\Gstab$--orbits of edges that $\proj(\tilde{\mu})$ crosses in $\tree$. Therefore, in order to estimate the intersection number
\[\intnum(\varphi_\gamma^k(\baseloop),\pants) \colonequals \sum_i\intnum(\varphi_\gamma^k(\baseloop),\seam) = \inf_{\mu}\card{\mu\cap \pants},\]
it suffices to vary $\mu$ in the isotopy class of $\varphi_\gamma^k(\baseloop)\subset (S,p)$ and bound the number of $\Gstab$--orbits of edges $e\in\edges$ that $\proj(\tilde{\mu})$ crosses in $\tree$.
\end{obs}

It is now apparent that we should consider paths that are isotopic to $\liftimage{k}$ and study their projections to $\tree$. Recall the set $\marked$ of dynamic marked points defined in Definition~\ref{defn:markedPoint}. For a given $t\in \R$ and a subset $V\subseteq \marked$, let $\liftclass{t}{V}$ denote the isotopy class of the path $\liftimage{t}$ in $(\hyp,V_{t})$, where $V_t = \{\rho(t) \mid \rho\in V\}$ is the set of locations of those dynamic marked points in $V$ at time $t$. Equivalently, $\liftclass{t}{V}$ is the isotopy class of $\liftimage{t}$ in $\hyp\setminus V_t$. This isotopy class is obtained from $\liftimage{t}\subset(\hyp,\marked_t)$ by simply ``forgetting,'' at time $t$, all of the dynamic marked that are not in $V$. One may alternately think of forgetting these points at time $0$ and pushing the initial path $\baselift$ by a modified isotopy that only moves those dynamic marked points contained in $V$. The resulting path is a representative of $\liftclass{t}{V}$.

Recall that the vertices and edges of $\tree$ are defined to be subsets of $\hyp$; in particular, it makes sense to say that a path $\mu\subset \hyp$ intersects a vertex $v\in \vertices$. More generally, any subset $A\subset \tree$ may be thought of as a subset of $\hyp$ by looking at the preimage $\proj\inv(A)\subset \hyp$. This identification will be used implicitly in the sequel. 

For a subset $X\subseteq\hyp$ (or $X\subseteq \tree$), we say that $\liftclass{t}{V}$ \emph{intersects $X$} if every path in the isotopy class intersects $X$. Otherwise, there is a representative path that avoids $X$ and we say that $\liftclass{t}{V}$ is \emph{disjoint from $X$}. Since every path that is isotopic to $\liftimage{t}$ in $(\hyp, \marked_t)$ lies in the isotopy class $\liftclass{t}{V}$, we see that if $\liftclass{t}{V}$ intersects $X$, then so does $\liftclass{t}{\marked}$. In particular, we may gain information about $\liftclass{t}{\marked}$ by considering the drastically simplified isotopy classes $\liftclass{t}{V}$ corresponding to certain finite subsets $V\subset \marked$.

As the endpoints of $\baselift$ in $\partial\hyp$ remain fixed throughout all isotopies, we find that $\liftclass{t}{V}$ intersects the base edge $\bedge\subset\tree$ for all times $t$ and all subsets $V\subseteq \marked$ (recall that $\bedge$ contains the edge $\proj(\baselift)$ and its adjacent vertices). Our goal is to show that, as time progresses, $\liftclass{t}{V}$ intersects larger and larger subsets of $\tree$.

\begin{defn}[Constraining points]\label{defn:constrains}
Let $V\subseteq\marked$ be a set of marked points containing a marked point $\rho$, and let $t\geq 0$ be a real number. We say that $\rho$ \emph{constrains} $\liftclass{t}{V}$ if $\liftclass{s}{V}$ intersects $\proj(\rho(s))\in \tree$ for all times $s\geq t$. In this case, every path in $\liftclass{s}{V}$ projects onto (a superset of) the unique $\tree$-geodesic connecting $\proj(\rho(s))$ to $\bedge$. See Figure~\ref{fig:constraint} for an illustration.
\end{defn}

The first thing to check is that such points exist. Recall that $\baselift$ was chosen specifically so that it intersects the path $\lpath{\basepoint}{\gamma}$ starting at the basepoint $\basepoint\in \hyp$; it follows that $\lpath{h\basepoint}{\gamma} = h(\lpath{\basepoint}{\gamma})$ intersects $\baselift =h(\baselift)$ for every $h\in \Gstab$.

\begin{figure}
\begin{minipage}[t]{0.48\linewidth}
\centering
\labellist
\small\hair 2pt
\pinlabel {$\liftimage{t}$} [b] <0pt,1pt> at 107 45
\pinlabel {$\proj(x_0(t))$} [t] <-1pt,-0.5pt> at 185 29
\pinlabel {$x_0$} [r] <0.5pt,1pt> at 135 83
\pinlabel {$x_1$} [b] <0pt,0pt> at 105 81
\pinlabel {$x_2$} [r] at 70 97
\pinlabel {$x_3$} [l] at 59 152
\pinlabel {$x_4$} [r] at 104 150
\pinlabel {$x_5$} [tr] <1pt,0pt> at 157 152
\pinlabel {$x_6$} [l] <0pt,-1pt> at 134 120
\endlabellist
\includegraphics[scale=1.0]{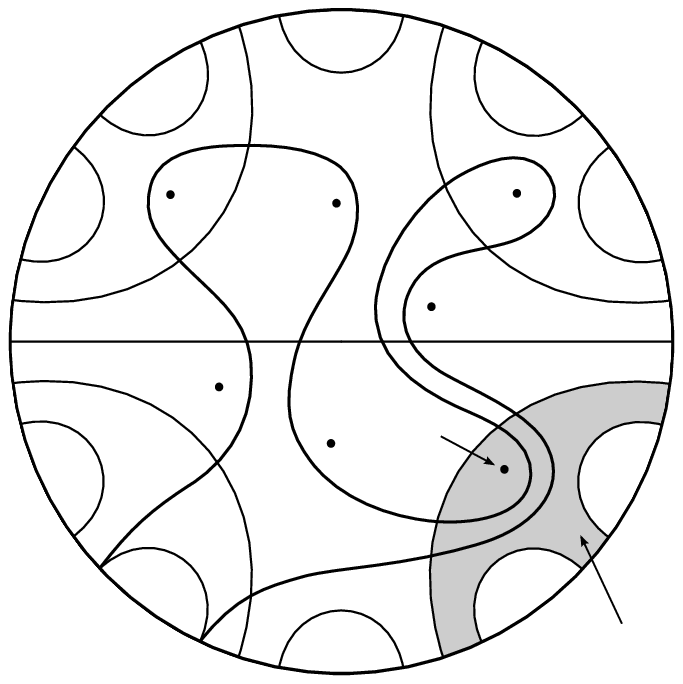}
\caption{The marked point $x_0$ constrains $\liftclass{t}{\{x_0,\dotsc,x_6\}}$ because $\liftimage{s}$ is isotopically forced to intersect the vertex $\proj(x_0(s))\in\tree$ for all $s\geq t$.}
\label{fig:constraint}
\end{minipage}
\hfill
\begin{minipage}[t]{0.48\linewidth}
\centering
\labellist\small\hair 2pt
\pinlabel {$\rho(s)$} [r] <1pt,0pt> at 70 159
\pinlabel {$z$} [l] <1pt,1pt> at 165 198
\pinlabel {$\rho(0)$} [r] <1pt,0pt> at 81 66
\pinlabel {$\proj(\rho(s))$} [t] at 172 126
\pinlabel {$B_0$} [c] at 90 38
\pinlabel {$\liftimage{0}$} [bl] <0pt,-1pt> at 94 73
\pinlabel {$\liftimage{s}$} [l] <1pt,0pt> at 124 76
\pinlabel {$R_0\bigtriangleup R_s$} [c] <1pt,-2pt> at 97 109
\pinlabel {$R_s$} [c] <1pt,-2pt> at 56 134
\endlabellist
\includegraphics[scale=1.0]{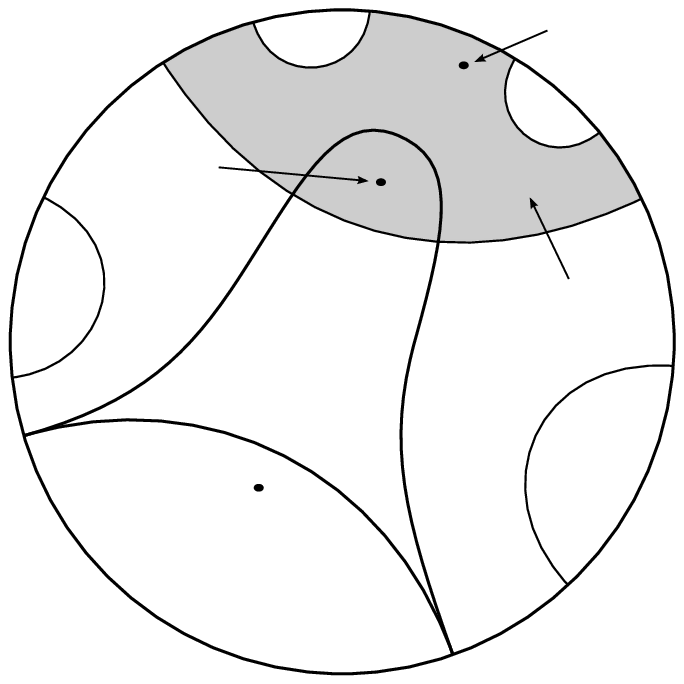}
\caption{Proving Proposition~\ref{prop:initialPushPoints}: The marked point $\rho$ necessarily constrains $\liftclass{s}{\{\rho\}}$.}
\label{fig:initialConstraints}
\end{minipage}
\end{figure}

\begin{prop}[Initial constraints] \label{prop:initialPushPoints}
For a fixed $h\in \Gstab$, let $\rho\in \marked$ be the dynamic marked point whose position at time $t$ is given by $\rho(t) = \tilde{F}_t(h \basepoint)$. Then $\rho$ constrains $\liftclass{1}{\{\rho\}}$.
\end{prop}
\begin{proof}
Fix a time $s\geq 1$. After forgetting about all other dynamic marked points and adjusting the pushing isotopy $\tilde{F}_t$ accordingly, we may assume that $\liftimage{s}$ represents an arbitrary path in the isotopy class $\liftclass{s}{\{\rho\}}$. By the Jordan Curve Theorem, each path $\liftimage{t}=\tilde{F}_t(\baselift)$ divides $\hyp$ into two path connected components, which we denote by $R_t$ and $B_t$ for the ``red'' and ``blue'' sides, respectively. These names are assigned consistently in $t$ so that they are preserved by the pushing isotopy, that is, $\tilde{F}_t(B_0) = B_t$. Assuming that $\rho$ initially lies in the blue side, we have that $\rho(t)\in B_t$ for all $t$. Since the path $\rho([0,1]) = \lpath{h\basepoint}{\gamma}$ intersects $\baselift$, the marked point $\rho$ evidently crosses over $\baselift$ during the time interval $[0,1]$. Recalling that $\proj(\rho(s))\in\tree$ is a subset of $\hyp$, we have that $\rho(s)\in R_0$ and $\proj(\rho(s))\subseteq R_0$ at time $s$; see Figure~\ref{fig:initialConstraints}.

The paths $\baselift$ and $\liftimage{s}$ have the same endpoints in $\partial\hyp$ and are therefore contained within bounded neighborhoods of each other. Thus the symmetric difference of $R_0$ and $R_s$ is contained in a bounded neighborhood of $\baselift$. On the other hand, $\proj(\rho(s))\subset R_0$ contains points that are arbitrarily far from $\baselift$. Therefore, by avoiding the symmetric difference, it is possible to choose a point $z\in \proj(\rho(s))\cap R_s$. The fact that $\rho(s)\in B_s$ and $z\in R_s$ lie in opposite components of $\hyp\setminus\liftimage{s}$ implies that every path from $\rho(s)\in\proj(\rho(s))$ to $z\in\proj(\rho(s))$ must intersect $\liftimage{s}$. 
Since $\proj(\rho(s))$ is path connected, this shows that $\liftimage{s}$ intersects $\proj(\rho(s))$.
\end{proof}

Once there are some constraints on $\liftimage{t}$, the weaving pattern of the dynamic marked points creates more in a recursive manner. The relevant interaction occurs when a marked point $x_1$ \emph{passes in front} of another marked point $x_2$, meaning that the intersection $x_1(\R)\cap x_2(\R) = \{y\}$ is a single point and that $x_1$ reaches $y$ before $x_2$ does. The basic intuition is this: if $x_1$ constrains $\liftimage{t}$ while it passes in front of $x_2$, then it drags $\liftimage{t}$ across the path in front of $x_2$. Since $\liftimage{t}$ is now blocking its way, $x_2$ is forced to push $\liftimage{t}$ ahead as it progresses through $\hyp$.

To make this recursive step precise, we formulate it in the context of $\tree$. Suppose that $x_1$ passes in front of $x_2$, and let $X = \proj(x_1(\R))\cap \proj(x_2(\R))$ be the intersection of their $\tree$-geodesics; according to Lemma~\ref{lem:boundOnDivergence}, $X$ contains at most $\len(\gamma)+1$ edges. For a time $s$, consider the two rays $x_j(\R_{\leq s})$; we think of these rays as \emph{tails} connecting the marked points $x_j$ to $\partial\hyp$. We say that $x_1$ and $x_2$ have \emph{diverged in $\tree$ at time $s$} if $X$ separates each ray $\proj(x_j(\R_{\leq s}))$ into two connected components, neither of which is a single vertex. Since a dynamic marked point crosses $\len(\gamma)$ edges of $\tree$ per unit time, we see that each time $s\geq t_2+2$ satisfies this criterion, where $t_2\in\R$ is the time at which $x_2$ reaches $\{y\} = x_1(\R)\cap x_2(\R)$. 

Assuming that $x_1$ and $x_2$ have diverged in $\tree$ at time $s$, choose any two edges $e,e'\in \edges$ in different components of $\proj(x_1(\R))\setminus X$ and consider their relationship to the tail $\eta = x_2(\R_{\leq s})$ of $x_2$. Thinking of $\eta$ as a wall or a barrier, it is apparent that the only way to get from $e$ to $e'$ is to ``go around'' the marked point $x_2(s)$ at the end of $\eta$. More precisely, any path $\mu\subset\hyp\setminus \eta$ that intersects both $e$ and $e'$ must also intersect $\proj(x_2(s))$; see Figure~\ref{fig:separationIdea}. The purpose of the next lemma is to prove that $\eta$ enjoys this same separation property on the level of isotopy classes of paths.

\begin{figure}
\begin{minipage}[b]{0.48\linewidth}
\centering
\labellist
\small\hair 2pt
\pinlabel {$\eta$} [br] <0pt,1pt> at 82 97
\pinlabel {$\mu$} [bl] at 166 130
\pinlabel {$e'$} [bl] at 53 155
\pinlabel {$e$} [br] <0pt,-1pt> at 46 72
\pinlabel {$x_2(s)$} [r] <1pt,1pt> at 160 111 
\pinlabel {$\proj(x_2(s))$} [t] <-9pt,-1pt> at 166 59
\endlabellist
\includegraphics[scale=1.0]{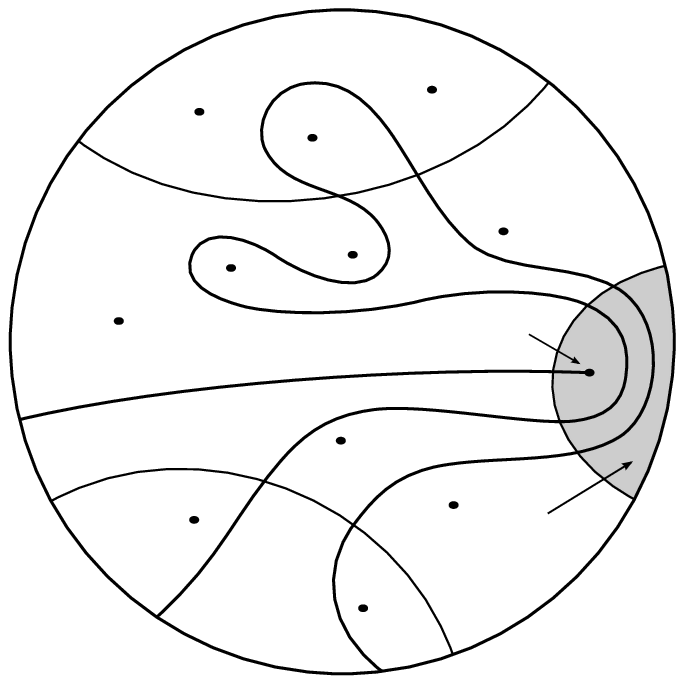}
\caption{$\eta$ separates $e$ from $e'$: any path that intersects $e$ and $e'$ must go through $\proj(x_2(s))$.}
\label{fig:separationIdea}
\end{minipage}
\hfill
\begin{minipage}[b]{0.48\linewidth}
\centering
\labellist\small\hair 2pt
\pinlabel {$\tilde{c}$} [tr] at 166 62
\pinlabel {$\eta$} [b] at 140 100
\pinlabel {$\eta'$} [t] <2pt,2pt> at 198 59
\pinlabel {$\proj(\rho(s))$} [b] at 32 188
\pinlabel {$\mu_0$} [tl] at 154 161
\pinlabel {$e$} [bl] at 128 39
\pinlabel {$\proj(x(s))$} [b] <0pt,1pt> at 180 189
\pinlabel {$W$} [c] at 69 37
\endlabellist
\includegraphics[scale=1.0]{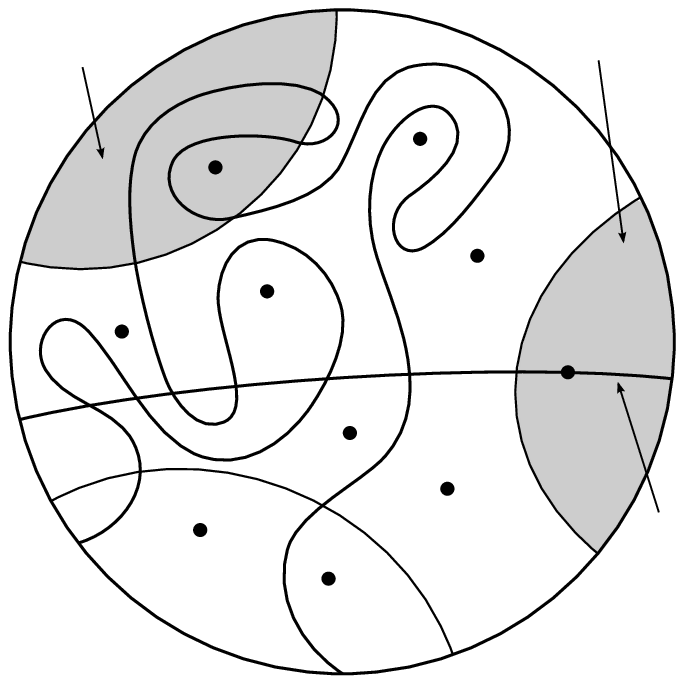}
\caption{The representative arc $\mu_0\in\liftclass{t}{V}$ is isotopically disjoint from $\eta\cup\eta'$ in $W$.}
\label{fig:separationProof}
\end{minipage}
\end{figure}

\begin{lem}[Tails separate]\label{lem:separation}
Suppose that the marked point $\rho\in\marked$ passes in front of another marked point $x\in\marked$ and that they have diverged in $\tree$ by time $s$. Let $\eta = x(\R_{\leq s})$ be the tail of $x$ and let $e\in \edges$ be any edge in the infinite component of $\proj(\rho(\R_{\leq s}))\setminus\proj(x(\R))$. Then $\eta$ separates $\proj(\rho(s))$ from $e$ in the following sense: If $V\subset \marked$ is a finite subset of marked points that contains $\{x,\rho\}$ and the isotopy class $\liftclass{s}{V}$ intersects both $\proj(\rho(s))$ and $e$ but is disjoint from $\eta$, then $\liftclass{s}{V}$ also intersects $\proj(x(s))$.
\end{lem}
\begin{proof}
Assuming that $x(s)$ is contained in a component $U$ of $\hyp\setminus \liftpants$, we let $\tilde{c}\subset\liftpants$ denote the boundary component of $U$ that intersects $\eta$. Thus $\tilde{c}$ defines an edge of $\tree$ that is adjacent to the vertex $\proj(x(s))\in\vertices$. If $x(s)\in \liftpants$ so that $\proj(x(s))$ is an edge of $\tree$, then we take $\tilde{c} = \proj(x(s))$. The fact that $x$ and $\rho$ have diverged in $\tree$ implies that $\tilde{c}$ is disjoint from $\proj(\rho(\R))$. Let $\eta'=x(\R_{\geq s})$ be the ``future ray'' of $x$. Recalling that $\tilde{c}$ is a geodesic component of $\liftpants$, it follows from the Jordan Curve Theorem that $\tilde{c}$ divides $\hyp$ into two components, one containing $\eta'$ and the other containing both $\proj(\rho(s))$ and $e$. 

Suppose now that $\liftclass{s}{V}$ intersects $\proj(\rho(s))$ and $e$ but is disjoint from $\eta$. Our goal is to show that $\liftclass{s}{V}$ also intersects $\proj(x(s))$. Since every path in $\hyp$ that intersects both $\proj(\rho(s))$ and $\eta'$ must intersect $\tilde{c}$ and, consequently, $\proj(x(s))$, it suffices to show that $\liftclass{s}{V}$ intersects $\eta'$. Supposing this is not the case, there is a representative path $\mu_0\in \liftclass{s}{V}$ that is disjoint from $\eta'$. The situation is illustrated in Figure~\ref{fig:separationProof}.

Consider the compactified disk $W_0 = \hyp\cup\partial\hyp$ and the corresponding punctured surface $W=W_0\setminus V_s$. Topologically, $W$ is a genus zero surface with one boundary component and $2\leq \card{V}<\infty$ punctures. A \emph{simple arc} in $W$ is a continuous injection $\beta\colon[0,1]\to W_0$ with $\beta\inv(\partial\hyp\cup V_s) = \{0,1\}$. All isotopies of a simple arc $\beta$ in $W$ are required to fix its endpoints in $\partial\hyp\cup V_s$ pointwise. It is a basic fact from surface topology that if a simple arc is isotopically disjoint from two other arcs, then it can be isotoped to be simultaneously disjoint from both. For example, this may be achieved by taking geodesic representatives in any hyperbolic metric on $W$ in which $\partial W$ is geodesic and the punctures are modeled on finite-volume cusps; see also  \cite[Expos\'e 3]{FLP79} or \cite[Corollary 1.9 and \S1.2.7]{FarbMargalit}. 

The two rays $\eta$ and $\eta'$ now define disjoint simple arcs in $W$, and $\liftclass{s}{V}$ becomes an isotopy class of simple arcs in $W$. By the Jordan Curve Theorem, the union $\eta\cup\eta'$ separates $W$ into two path connected components, and the choice of $e$ ensures that $\proj(\rho(s))$ and $e$ are contained in opposite components of $W\setminus(\eta\cup\eta')$. By hypothesis, arcs in $\liftclass{s}{V}$ are isotopically disjoint from both $\eta$ and $\eta'$; therefore, we may find a representative $\mu\in \liftclass{s}{V}$ that is simultaneously disjoint from both $\eta$ and $\eta'$. Since any such $\mu$ is contained in one component of $W\setminus(\eta\cup\eta')$, this contradicts the fact that $\mu$ necessarily intersects both $\proj(\rho(s))$ and $e$.
\end{proof}

In order to apply Lemma~\ref{lem:separation} recursively, we need to identify dynamic marked points $x\in \marked$ and subsets $V\subset \marked$ for which the tail $\eta = x(\R_{\leq s})$ is disjoint from $\liftclass{s}{V}$. This is easily accomplished by ensuring that the marked points in $V$ never cross the ray $x(\R_{\leq t})$.

\begin{lem}[Finding disjoint tails]\label{lem:disjointTails}
Let $V = (y_1,\dotsc,y_n,x)\subset \marked$ be a finite set of  marked points, and suppose that $(y_1(\R),\dotsc,y_n(\R),x(\R))$ is a $\liftgamma$-chain. If $y_n$ passes in front of $x$ and $\liftclass{0}{V}$ is disjoint from $x(\R)$, then, at each time $t\geq 0$,  $\liftclass{t}{V}$ is disjoint from the ray $x(\R_{\leq t})$.
\end{lem}
\begin{proof}
The result will follow easily from the following basic principle.

\begin{claim*}\label{lem:disjointSubsets}
Let $[s,s']$ be a time interval, and let $U\subset \hyp$ be a connected open set with the property that for each $\rho\in V$, the image $\rho([s,s'])$ is either contained in $U$ or is disjoint from the closure of $U$. If $\liftclass{s}{V}$ is disjoint from $U$, then $\liftclass{t}{V}$ is disjoint from $U$ for all $t\in [s,s']$.
\end{claim*}
\begin{proof}[Proof of Claim]
The isotopy of $\hyp$ that pushes the marked points in $V$ may be taken to be the identity away from the marked points, that is, off of an open neighborhood of $\cup_{\rho\in V}\rho([s,s'])$. In particular, we may assume that this isotopy is the identity on $\partial U$ throughout the time interval $[s,s']$. Applying this isotopy to a representative path $\mu\in\liftclass{s}{V}$ that is disjoint from $U$, we see that each isotopy class $\liftclass{t}{V}$ has a representative that is disjoint from $\partial U$ and therefore from $U$.
\end{proof}
We now complete the proof of Lemma~\ref{lem:disjointTails}.
Let $t_1,t_2\in \R$ be the times defined by $y_n(t_1) = z = x(t_2)$, where $\{z\} = y_n(\R)\cap x(\R)$. The hypothesis on passing is that $t_1 < t_2$, and we choose a point $t_0\in (t_1,t_2)$. Assuming that $t_0 > 0$, let $U$ be a small open neighborhood of $x(\R_{\leq t_0})$ whose closure is disjoint from $y_n(\R)$ and $\liftclass{0}{V}$. The chain condition implies we may choose $U$ so that, for $i<n$, the marked point $y_i$ avoids the closure of $U$ throughout all of time. Since $x$ remains inside $U$ during the interval $[0,t_0]$, the above claim implies that $\liftclass{t}{V}$ is disjoint from $U$, and thus from $x(\R_{\leq t})$, for each time $t\in [0,t_0]$. In the case that $t_0 \leq 0$, we simply note that the hypotheses ensure that $\liftclass{0}{V}$ is disjoint from $x(\R_{\leq 0})$.

It remains to consider a time $t \geq s_0 =\max\{t_0,0\}$. At time $s_0$, $y_n$ has already crossed $x(\R)$; thus there is no obstruction to sliding any intersections of $\liftclass{s_0}{V}$ with $x(\R_{\geq s_0})$ forward along $x(\R)$ to obtain a representative path $\mu\in \liftclass{s_0}{V}$ that is disjoint from $x(\R_{\leq t})$. Enlarging $U$ to a neighborhood $U'$ of $x(\R_{\leq t})$, we find that $\liftclass{s_0}{V}$ is disjoint from $U'$ and that all of the marked points $y_i$ avoid $U'$ throughout the interval $[s_0,t]$. A second application of the claim now shows that $\liftclass{t}{V}$ is disjoint from $x(\R_{\leq t})$.
\end{proof}

\subsection{The points that push}\label{sec:pointsThatPush}

The stage is set: we have developed the navigational tools and built up the machinery for pushing in the universal cover $\hyp$. Everything is in place to exhibit exponentially many dynamic marked points that constrain $\liftclass{t}{\marked}$.

Recall our pushing curve $\gamma\colon[0,1]\to S$, which is a geodesic loop based at $\gamma(0) = p$. We need a way to refer to the self-intersection points of $\gamma$. If each point of intersection on $S$ corresponded to a double intersection of $\gamma$, then we could simply label the self-intersections by their corresponding points in $S$ (this is the approach we will take in \S\S\ref{sec:trainTracks}--\ref{sec:upperbounds} below). However, it may be that that some points in $S$ correspond to, say, triple intersections of $\gamma$ or, in the the extreme case, that all self-intersections occur at a single point of $S$. To accommodate such possibilities, a \emph{self-intersection point} $q$ of $\gamma$ will mean an ordered pair $q =(\itime{q}{1},\itime{q}{2})$ of times $0\leq\itime{q}{1}<\itime{q}{2}\leq 1$ for which $\gamma(\itime{q}{1})=\gamma(\itime{q}{2})$. By the Definition \ref{defn:selfIntNum} of self-intersection number and the fact that $\gamma$ is a geodesic, $\gamma$ has exactly $\intnum(\gamma)$ self-intersection points.

For each self-intersection point $q$ of $\gamma$, we form the decomposition $\gamma = \beta_q\delta_q\nu_q$, where
\begin{equation}\label{eqn:gammaDecomp}
\beta_q = \gamma\vert_{[0,\itime{q}{1}]}, \qquad \delta_q = \gamma\vert_{[\itime{q}{1},\itime{q}{2}]}, \quad\text{and}\quad \nu_q = \gamma\vert_{[\itime{q}{2},1]}.
\end{equation} 
Skipping over the subloop $\delta_q$ based at $\gamma(\itime{q}{1})$, we form the concatenation $\turnat{q} = \beta_q\nu_q\in\pi_1(S,p)$. This is a piecewise geodesic loop based at $p$ which may be given the explicit parameterization
\begin{equation}\label{eqn:theTurns}
\turnat{q}(t) = \beta_q\nu_q =\begin{cases}
\gamma(\itime{q}{1}), & t_{q1} \leq t \leq t_{q2} \\
\gamma(t), & \text{otherwise.}
\end{cases}
\end{equation}
Each path lift $\lpath{x}{\turnat{q}}$ to $\hyp$ is a grid path that follows along a geodesic lift of $\gamma$ and then turns, at a lift of $\gamma(\itime{q}{1})$, onto a new geodesic. Accordingly, we refer to $\turnat{q}$ as the ``turn at $q$.''

Let $\turns = \{\gamma\}\cup\left(\bigcup_q \{\turnat{q}\}\right)\subset \pi_1(S,p)$ be the collection consisting of the ``straight loop'' $\gamma$ and these $\intnum(\gamma)$ turns. We will find dynamic marked points that constrain $\liftclass{t}{\marked}$ by following paths in $\hyp$ corresponding to words $\omega_1\dotsb\omega_n\in\pi_1(S,p)$ in the letters $\omega_i\in\turns$. For each $x\in\pi\inv(p)$, every such word $\omega_1\dotsb\omega_n$ lifts to a grid path $\mu=\lpath{x}{(\omega_1\dotsb\omega_n)}$ in $\hyp$. When $\mu$ is decomposed $\mu= \mu_1\dotsb\mu_j$ as a grid path, the number $j$ of straight segments is one more than the number of turns in $\mu$, that is, $j - 1  = \card{\{i : \omega_i\neq \gamma\}}$. In order to apply the theory we have developed, we need to consider grid paths whose straight segments define $\liftgamma$-chains. According to Lemma~\ref{lem:lazyChain}, this may be accomplished by ``padding'' the word $\omega_1\dotsb\omega_n$ with copies of $\gamma$ in order to ensure that the straight segments are long enough. 

Recall that $\bedge\subset \tree$ consists of the edge $\proj(\baselift)\in\edges$ and its adjacent vertices, that $\Gstab\leq G$ is the stabilizer of $\bedge$, and that $\basepoint$ is the fixed basepoint of $\hyp$.

\begin{prop}[The points]\label{prop:thePaths}
Let $k\geq 1$ be an integer, and let $h\in \Gstab$ be a deck transformation that preserves the base edge $\bedge$. Set $\omega_1 = \gamma$, and choose loops $\omega_2,\dotsc,\omega_k\in \turns$. For $1\leq i \leq k$, let $x_i\in \marked$ be the dynamic marked point whose location at time $5i$ is given by
\[x_i(5i) = (h\basepoint)\cdot\left[(\gamma^2\omega_1\gamma^2)\dotsb(\gamma^2\omega_i\gamma^2)\right].\]
If $V = \{x_1\dotsc,x_k\}$, then $x_k$ constrains $\liftclass{5k}{V}$ and, in particular, $\liftclass{5k}{\marked}$.
\end{prop}
\begin{proof}
For notational convenience, we let $x_0\in \marked$ denote the dynamic marked point whose initial location is $x_0(0) = h\basepoint$. Notice that the marked points $x_i$ are defined so that
\[x_{i+1}(5i+5) = \left((h\basepoint)\cdot\left[(\gamma^2\omega_1\gamma^2)\dotsb(\gamma^2\omega_i\gamma^2)\right]\right)\cdot(\gamma^2\omega_{i+1}\gamma^2) = x_i(5i)\cdot(\gamma^2\omega_{i+1}\gamma^2).\]
Since the marked point $x_i$ travels from $x_i(5i)$ to $x_i(5i)\cdot\gamma^5$ during the time interval $[5i,5i+5]$, choosing $\omega_{i+1} = \gamma$ results in the equation $x_{i+1}(5i+5) = x_{i}(5i+5)$. As distinct marked points cannot be at the same place at the same time, this shows that choosing $\omega_{i+1} = \gamma$ is equivalent to setting $x_{i+1} = x_i$.

We proceed by induction on $k$, starting with the case $k= 1$ and $V = \{x_1\}$. The assignment $\omega_1 = \gamma$ ensures that $x_1 = x_0$ so that $V = \{x_0\}$. Proposition~\ref{prop:initialPushPoints} now shows that $x_1 = x_0$ constrains $\liftclass{1}{V}$ and, consequently, $\liftclass{5}{V}$.

For $k > 1$, we inductively assume that the marked point $x_{k-1}$ constrains $\liftclass{5(k-1)}{V'}$, where $V'=\{x_1,\dotsc,x_{k-1}\}$. Since $5(k-1) \leq 5k$ and $V'\subseteq V = \{x_1,\dotsc,x_k\}$, this implies that $x_{k-1}$ constrains $\liftclass{5k}{V}$. We must show that $x_k$ constrains $\liftclass{5k}{V}$ as well. The result is immediate if $x_k = x_{k-1}$, so it suffices to consider the case $\omega_k \neq \gamma$. The proof proceeds as a series of steps that establish the properties needed to apply Lemmas~\ref{lem:separation} and \ref{lem:disjointTails} and conclude the result.
\bigskip

\noindent
\textbf{Step 1:} \emph{Notation.} Each index $i\geq 2$ with $\omega_i = \gamma$ results in repeated entries in the list $x_1,\dotsc,x_k$; upon deleting all neighboring repeats, we obtain an ordered list $y_1,\dotsc,y_j$ of marked points $y_i\in \marked$ that satisfy $y_i \neq y_{i+1}$ and $\{y_1,\dotsc,y_j\} = V$. Here $k-j$ is the number of indices $i\in\{2,\dotsc,k\}$ for which $\omega_i = \gamma$. Since we assumed $x_k\neq x_{k-1}$, we have $y_j=x_k$ and $y_{j-1} = x_{k-1}$.

The path $\mu = \lpath{h\basepoint}{\left[(\gamma^2\omega_1\gamma^2)\dotsb(\gamma^2\omega_k\gamma^2)\right]}$ defines a grid path in $\hyp$ that may be decomposed as a concatenation $\mu = \mu_1\dotsb\mu_n$ of geodesic segments $\mu_i$ along lifts $l_i\in \liftgamma$ of $\gamma$ that satisfy $l_i \neq l_{i+1}$. The number $n$ of straight segments $\mu_i$ in this decomposition is equal to $1+m$, where $m = \card{\{1 \leq i \leq k : \omega_i \neq \gamma\}}$ is the number of turns in $\mu$. Noting that $k-j = (k-m)-1$, we find that $j = n$. It is clear from the definitions that the marked points $x_1,\dotsc,x_k$ travel along the straight segments of $\mu$. Upon reindexing them as $y_1,\dotsc,y_j$, the resulting marked point $y_i$ travels along the geodesic $y_i(\R) = l_i$ containing the segment $\mu_i$. Indeed, since each turn $\turnat{q}$ along $\mu$ results in both a new segment $\mu_i$ and a distinct marked point $y_i$, this follows inductively from the observation that $y_1(\R) = l_1$.
\bigskip

\noindent
\textbf{Step 2:} \emph{$x_{k-1}$ passes in front of $x_k$.} The assumption $x_k \neq x_{k-1}$ implies that the geodesics $x_{k-1}(\R)$ and $x_k(\R)$ intersect in a single point. Suppose that $\omega_k=\turnat{q}\in\turns$, where $q=(\itime{q}{1},\itime{q}{2})$ is a self-intersection point of $\gamma$, and $\turnat{q}=\beta_q\nu_q$ is the concatenation of the two geodesic segments $\beta_q$ and $\nu_q$ defined in (\ref{eqn:gammaDecomp}). It then takes $\itime{q}{1}$ time units for a marked point to travel across $\beta_q$ and $(1-\itime{q}{2})$ time units to cross $\nu_q$. Since $x_k(5k) = x_{k-1}(5k-5)\cdot\gamma^2\beta_q\nu_q\gamma^2$,
we see that 
\[x_{k-1}(5k-5 + 2 + \itime{q}{1}) = x_{k}(5k-2-(1-\itime{q}{2})).\] Therefore $x_{k-1}$ does pass in front of $x_k$ because $5k- 3 + \itime{q}{1} < 5k- 3 + \itime{q}{2}$.
\bigskip

\noindent
\textbf{Step 3:} \emph{$(l_1,\dotsc,l_j)$ is a $\liftgamma$-chain.} Each straight segment $\mu_i$, with $1 < i < j$, has length $\len(\mu_i) \geq 3\len(\gamma)$ because there are at least four copies of $\gamma$ between any two turns along $\mu$. (We in fact have $\len(\mu_i)\geq 4\len(\gamma)$; this will be used in the proof of Proposition~\ref{prop:distinctness} below.) Therefore $\mu=\mu_1\dotsb\mu_j$ satisfies the hypotheses of Lemma~\ref{lem:lazyChain}, and it follows that $(l_1,\dotsc,l_j)$ is a $\liftgamma$-chain. 
\bigskip

\noindent
\textbf{Step 4:} \emph{$\liftclass{s}{V}$ is disjoint from $y_j(\R_{\leq s})$.} The proof of Lemma~\ref{lem:lazyChain} shows that $\proj(l_1)$ and $\proj(l_j)$ are disjoint if $j >2$. Since $\proj(\mu_1)$ contains the base edge $\proj(\baselift)$ of $\tree$, this shows that $l_j$ and $\baselift$ are disjoint when $j > 2$. While $\proj(\mu_1)\cap \proj(l_j)$ is nonempty when $j=2$, the fact that $\baselift$ intersects the initial subpath $\lpath{h\basepoint}{\gamma}$ of $\mu_1=\lpath{h\basepoint}{(\gamma^5\dotsb)}$ implies that $\baselift$ is not one of the last $\len(\gamma)+1$ edges of $\proj(\mu_1)$ and therefore cannot be contained in the intersection $\proj(\mu_1)\cap \proj(l_2)$. In any case, we find that $\liftclass{0}{V}$ is disjoint from $y_j(\R)=l_j$. Since $(y_1(\R),\dotsc,y_j(\R))$ is a $\liftgamma$-chain and $y_{j-1}=x_{k-1}$ crosses in front of $y_j=x_k$, Lemma~\ref{lem:disjointTails} now implies that $\liftclass{s}{V}$ is disjoint from $y_j(\R_{\leq s})$ for all times $s\geq 0$.
\bigskip

\noindent
\textbf{Step 5:} \emph{$y_j(\R_{\leq s})$ enjoys the separation property of Lemma~\ref{lem:separation}.} Fix a time $s \geq 5k$ and consider the path
\[\mu' = \lpath{h\basepoint}{\left[(\gamma^2\omega_1\gamma^2)\dotsb(\gamma^2\omega_{k-1}\gamma^2)(\gamma^5\dotsb)\right]}\]
from $h\basepoint$ to $x_{k-1}(s) = y_{j-1}(s)$. Comparing this path with $\mu$, we find that $\mu'$ decomposes as a grid path $\mu' = \mu_1\dots\mu_{j-2}\mu_{j-1}'$, where $\mu_{j-1}$ is the initial subpath of $\mu_{j-1}'$. Let $e$ be an internal edge of $\proj(\mu_{j-1})$, in which case $e$ must lie in the infinite component of $\proj(y_{j-1}(\R_{\leq s}))\setminus\proj(y_j(\R))$. Since $\mu_{j-1}'\supseteq \mu_{j-1}$, $e$ is also an internal edge of $\proj(\mu_{j-1}')$. As $y_{j-1}$ passes in front of $y_j$ and these two marked points have diverged in $\tree$ by the time $s\geq 5k$, it now follows that the ray $\eta = y_j(\R_{\leq s})$ separates $\proj(y_{j-1}(s))$ from $e$ in the sense of Lemma~\ref{lem:separation}
\bigskip

\noindent
\textbf{Step 6:} \emph{$x_k$ constrains $\liftclass{5k}{V}$.} By our induction hypothesis, the isotopy class $\liftclass{s}{V}$ intersects $\proj(y_{j-1}(s))$. It follows that every path in $\liftclass{s}{V}$ projects onto the $\tree$-geodesic from $\proj(y_{j-1}(s))$ to the base edge $\bedge$. By Lemma~\ref{lem:lazyChain}, this $\tree$-geodesic contains the edge $e$, so it must be that $\liftclass{s}{V}$ intersects $e$ as well. Applying the separation property from Lemma~\ref{lem:separation}, we finally conclude that $\liftclass{s}{V}$ intersects $\proj(y_j(s))$. This proves that $x_{k}=y_j$ constrains $\liftclass{5k}{V}$.
\end{proof}

For each $h\in \Gstab$, we have now described $(\intnum(\gamma)+1)^{k-1} = \card{\turns}^{k-1}$ dynamic marked points that constrain the isotopy class $\liftclass{5k}{\marked}$. However, it remains to be seen that these marked points are distinct and that they project to distinct vertices in the quotient graph $\Gstab\backslash \tree$.

\begin{prop}[Distinctness]\label{prop:distinctness}
Let $k\geq 1$ be an integer. For each $h\in \Gstab$ and each ordered list $\omega = (\omega_1,\omega_2,\dotsc,\omega_k)$ of $k$ elements $\omega_i\in \turns$ satisfying $\omega_1 = \gamma$, consider the point
\[\Psi(h,\omega) = (h\basepoint)\cdot[(\gamma^2\omega_1\gamma^2)\dotsb(\gamma^2\omega_k\gamma^2)]\in \hyp.\]
For each distinct choice of $h$ and $\omega$, this point projects to a distinct vertex $\proj(\Psi(h,\omega))$ in $\tree$. In particular, the orbits $\Gstab\proj(\Psi(h,\omega))$ and $\Gstab\proj(\Psi(h',\omega'))$ are equal if and only if $\omega = \omega'$.
\end{prop}
\begin{proof}
We first deal with the dependence on $h\in \Gstab$. Choose any list $\omega$ and consider the path $\mu=\lpath{\basepoint}{[(\gamma^2\omega_1\gamma^2)\dotsb(\gamma^2\omega_k\gamma^2)]}$ from the basepoint $\basepoint$ to $\Psi(1,\omega)$. This defines a grid path $\mu=\mu_1\dotsb\mu_n$ whose straight segments $\mu_i$ satisfy the hypotheses of Lemma~\ref{lem:lazyChain}. Since $\omega_1=\gamma$, the first straight segment $\mu_1$ contains the initial subpath $\lpath{\basepoint}{\gamma^5}$ and has length $\len(\mu_1)\geq 5\len(\gamma)$. The beginning $\lpath{\basepoint}{\gamma}$ of this path intersects the base edge $\baselift\in\edges$ of $\tree$; therefore $\proj(\lpath{\basepoint}{\gamma^5})$ must contain a geodesic edge path of the form $(\baselift,b_2,\dotsc,b_m)$, where $b_m$ is an internal edge of $\proj(\lpath{\basepoint}{\gamma^5})$. Notice that $b_m$ is also an internal edge of $\proj(\mu_1)$ and that the choice of $b_m$ does not depend on $\omega$. If $v\in \vertices$ denotes the initial vertex of this edge path, then both $\baselift$ and $v$ are fixed by every element of $\Gstab$.  Lemma~\ref{lem:lazyChain} now implies that the $\tree$-geodesic from $\proj(\basepoint)$ to $\proj(\Psi(1,\omega))$ contains the edge $b_m$; in particular, the $\tree$-geodesic from $v$ to $\proj(\Psi(1,\omega))$ must have the form \[(\baselift,b_2,\dotsc,b_m,\dotsc).\]
Applying a deck transformation $h\in \Gstab$, we see that the $\tree$-geodesic from $hv= v$ to $h\proj(\Psi(1,\omega)) = \proj(\Psi(h,\omega))$ has the form $(\baselift,hb_2,\dotsc,hb_m\dotsc)$. That is, the $m^{\text{th}}$ edge of the $\tree$-geodesic from $v$ to $\proj(\Psi(h,\omega))$ is $hb_m$. This feature is independent of $\omega$. Since $hb_m \neq h'b_m$ for distinct $h,h'\in \Gstab$, this proves that the vertices $\proj(\Psi(h,\omega))$ and $\proj(\Psi(h',\omega'))$ are distinct when $h\neq h'$.

It remains to consider the dependence on $\omega$. The following notation will aid our analysis. Let $M$ denote the optimal upper bound from Lemma~\ref{lem:boundOnDivergence}; thus $M$ is either $\len(\gamma)+1$ or $\len(\gamma)$ depending on whether or not $\len(\gamma) > 2$. Furthermore, in the case that $\len(\gamma)\geq 3$, we take a decomposition $\gamma = \xi_1\xi_2$ of $\gamma$ into two subpaths which satisfy $\len(\xi_1) = 2$ and $\len(\xi_2) \geq 1$. If $\len(\gamma)\leq 2$, we instead choose this decomposition such that $\len(\xi_1) = 1$ and $\len(\xi_2) \leq 1$. Notice that $M+1 = \len(\gamma)+\len(\xi_1)$.  

For the remainder of the proof, we may consider a fixed element $h\in \Gstab$. Let $\omega=(\omega_1,\dotsc,\omega_k)$ and $\omega' = (\omega_1',\dotsc,\omega_k')$ be two distinct lists, and let $j$ be the smallest index with $\omega_j\neq \omega'_j$. Set $e\in \edges$ to be the last edge that the path $\eta = \lpath{h\basepoint}{[(\gamma^2\omega_1\gamma^2)\dotsb(\gamma^2\omega_j\gamma^2)\xi_1]}$ crosses, and let $y = (h\basepoint)\cdot\eta$ be the endpoint of this path. Define $e'\in \edges$ and $y' = (h\basepoint)\cdot \eta'$ similarly. The bulk of our argument is devoted to proving the following claim.

\begin{claim*} The $\tree$-geodesic from $\proj(y)$ to $\proj(y')$ has length at least $\len(\gamma)+2$ and contains both $\bar{e}$ and $e'$. This essentially means that the two geodesics connecting $\proj(h\basepoint)$ to $\proj(y)$ and $\proj(y')$ have diverged in $\tree$.
\end{claim*}
\begin{proof}[Proof of Claim] First consider the case that neither $\omega_j$ nor $\omega_j'$ is equal to $\gamma$. Using the notation of (\ref{eqn:theTurns}), we then have $\omega_j = \turnat{q}=\beta_q\nu_q$ and $\omega'_j = \turnat{q'}=\beta_{q'}\nu_{q'}$ for two distinct self-intersection points $q,q'$ of $\gamma$. Let
\[z = (h\basepoint)\cdot[(\gamma^2\omega_1\gamma^2)\dotsb(\gamma^2\omega_{j-1}\gamma^2)(\gamma^2\beta_q)]\]
be the point where $\eta$ makes its final turn towards $y$, and let $A = \lpath{z}{(\nu_q\gamma^2\xi_1)}$ be the geodesic from $z$ to $y$. Define $z'$ and $A'$ similarly. The geodesic segment $X$ from $z$ to $z'$ is then a subpath of a segment of the form $\lpath{x}{\gamma}$; as such, it has $\len(X)\leq \len(\gamma)$. On the other hand, the edge paths $\proj(A)$ and $\proj(A')$ both have length at least $2\len(\gamma)+\len(\xi_1)$, and their intersection $\proj(A)\cap\proj(A')$ contains at most $M$ edges.

As in the proof of Lemma~\ref{lem:lazyChain}, $\proj(\bar{A}XA')$ is an edge path from $\proj(y)$ to $\proj(y')$ that can be made into a $\tree$-geodesic by successively cancelling edge pairs $(\dotsc,d,\bar{d},\dotsc)$ to remove any backtracking. The path $\proj(X)$ can contribute to at most $\len(X)$ cancellations, and, assuming all of these edges cancel, we can then have at most $M$ cancellations involving edges of $\proj(\bar{A})$ with edges of $\proj(A')$. Therefore, the $\tree$-geodesic $L$ from $\proj(y)$ to $\proj(y')$ will be obtained from $\proj(\bar{A}XA')$ after at most $\len(X)+M$ cancellations. Since each cancellation removes two edges, it follows that
\begin{align*}
\len(L)& \geq 2\big(2\len(\gamma)+\len(\xi_1)\big)+\len(X)-2\big(\len(X)+M\big)\\
 & = 2\len(\gamma)-\len(X)+2
 \geq \len(\gamma)+2.
\end{align*}
Furthermore, since $\proj(\bar{A})$ and $\proj(A')$ each contain at least $2\len(\gamma)+\len(\xi_1)\geq \len(X)+M+1$ edges, we see that the first edge of $\proj(\bar{A})$ and the last edge of $\proj(A')$ do not cancel. As these edges are exactly $\bar{e}$ and $e'$, the claim holds when neither $\omega_j$ nor $\omega_j'$ is equal to $\gamma$.

The argument for the case $\omega_j \neq \omega_j' = \gamma$ is similar: Define $z$ as above and again let $A = \lpath{z}{(\nu_q\gamma^2\xi_1)}$ be the geodesic from $z$ to $y$. The geodesic from $z$ to $y'$ is then given by $A' = \lpath{z}{(\delta_q\nu_q\gamma^2\xi_1)}$, where $\delta_q$ is as in (\ref{eqn:gammaDecomp}). Since these are both segments along geodesics in $\liftgamma$, the concatenation $\proj(\bar{A})\proj(A')$ can result in at most $M$ cancellations. The resulting $\tree$-geodesic from $\proj(y)$ to $\proj(y')$ has length at least
\begin{align*}
2\big(2\len(\gamma)+\len(\xi_1)\big)-2\len(M)\geq 2\len(\gamma)+2
\end{align*}
and still contains the initial and terminal edges $\bar{e}$ and $e'$. This proves the claim.
\end{proof}

We now complete the proof of Proposition~\ref{prop:distinctness}. To prove that $\Psi(h,\omega)$ and $\Psi(h,\omega')$ lie in distinct vertices of $\tree$, it suffices to show that the $\tree$-geodesic between these vertices is nondegenerate. To ease the notation, set $x=\Psi(h,\omega)$ and $x'=\Psi(h,\omega')$. First suppose that $j=k$, in which case we have $y = x \cdot \xi_1$ and $y' = x'\cdot \xi_1$. Together with the above claim, the triangle inequality then implies that
\[d(\proj(x),\proj(x'))\geq d(\proj(y),\proj(y'))-2\len(\xi_1) \geq \len(\gamma)+2-2\len(\xi_1) \geq 1,\]
where $d$ is the path metric in $\tree$. In the case that $j < k$, we instead consider the path $\mu = \lpath{h\basepoint}{[(\gamma^2\omega_1\gamma^2)\dotsb(\gamma^2\omega_k\gamma^2)]}$ from $h\basepoint$ to $x$. This is a grid path whose straight segments satisfy the hypotheses of Lemma~\ref{lem:lazyChain}. Let $\mu_i$ be the straight segment of $\mu$ containing the edge $e$. Then $\mu_i$ contains a subpath of the form $\gamma^2\xi_1\xi_2\gamma$, where $e$ is the last edge that the $\xi_1$ factor crosses. It is now evident that $e$ separates $\mu_i$ into two edge paths of lengths at least $2\len(\gamma)$ and $\len(\gamma)+\len(\xi_2)\geq M$. Therefore, the definition of $\xi_2$ ensures $e$ is an internal edge of $\proj(\mu_i)$. Lemma~\ref{lem:lazyChain} now implies that the $\tree$-geodesic from $\proj(h\basepoint)$ to $\proj(x)$ contains $e$ and, similarly, that $\tree$-geodesic from $\proj(h\basepoint)$ to $\proj(x')$ contains $e'$. Writing these geodesic edge paths as $(b_1,\dots,b_n,e,a_1,\dots,a_m)$ and $(b'_1,\dotsc,b'_{n'},e',a'_1,\dots,a'_{m'})$, and combining them with the geodesic $(\overline{e},d_1,\dots,d_l,e')$ from $\proj(y)$ to $\proj(y')$, we find that
\[(\overline{a_m},\dotsc,\overline{a_1},\overline{e},d_1,\dotsc,d_l,e',a'_1,\dotsc,a'_{m'})\]
is a nondegenerate, non-backtracking edge path from $\proj(x)$ to $\proj(x')$.
\end{proof}

\subsection{The point of pushing: proof of the lower bound}\label{sec:countingIntersections}

Now that we have found distinct orbits of dynamic marked points that constrain $\liftclass{t}{\marked}$, it is a simple matter to count intersection numbers and bound the dilatation $\pushdil{\gamma}$. We first state the following corollary to the above propositions.

\begin{cor}[Intersection numbers]\label{cor:intersectionNumbers}
Let $\gamma\in\pi_1(S,p)$ be a filling loop that represents a primitive element of $\pi_1(S,p)$, and let $\baseloop$ be an essential simple closed curve on $S$ that is contained in $S\setminus\{p\}$. Choose any pants decomposition $\pants = \{\seam\}$ of $S$ that contains $\baseloop$ and consists of curves contained in $S\setminus\{p\}$. Then for all integers $k\geq 1$, the iterates $\varphi_\gamma^{5k}(\baseloop)\subset(S,p)$ of $\baseloop\subset(S,p)$ under the point-pushing homeomorphism $\varphi_\gamma$ satisfy
\[ \intnum(\varphi_\gamma^{5k}(\baseloop),\pants) = \sum_i \intnum(\varphi_\gamma^{5k}(\baseloop),\seam)\geq (\intnum(\gamma)+1)^{k-1}.\]
\end{cor}
\begin{proof}
After fixing a hyperbolic metric on $S$, modifying each curve by an isotopy to make it geodesic, and adjusting the basepoint accordingly, we may assume that $\gamma$ and $\alpha$ satisfy Assumptions~\ref{assm:Assumptions} and that each simple closed curve $\seam\in\pants$ is geodesic. We then have the corresponding tree $\tree = \tree_\pants$ described in \S\ref{sec:divergence} and may apply the theory developed in \S\S\ref{sec:divergence}--\ref{sec:pointsThatPush}. It follows that all of the dynamic marked points described in Proposition~\ref{prop:thePaths} constrain $\liftclass{5k}{\marked}$. Let $\mu\subset S$ be any simple closed curve that is isotopic to $\varphi_\gamma^{5k}(\baseloop)$ in $(S,p)$, and let $\tilde{\mu}\subset (\hyp,\pi\inv(p))$ be the lift of $\mu$ whose endpoints on $\partial\hyp$ agree with those of $\baselift$. We assume that $\mu$ is transverse to the curves in $\pants$. Since $\tilde{\mu}$ is in the isotopy class $\liftclass{5k}{\marked}$, 
its projection $\proj(\tilde{\mu})$ to $\tree$ is an edge path that necessarily visits all of the vertices described by Proposition~\ref{prop:distinctness}. Since these project to $(\intnum(\gamma)+1)^{k-1}$ distinct vertices in the quotient graph $\Gstab\backslash \tree$, it is apparent that $\proj(\tilde{\mu})$ projects to a closed loop in $\Gstab\backslash \tree$ that crosses at least $(\intnum(\gamma)+1)^{k-1}$ edges. Observation~\ref{obs:countingIntersections} now implies that
\[\card{\mu\cap \pants} \geq (\intnum(\gamma)+1)^{k-1}.\]
Since $\mu$ is an arbitrary representative in the isotopy class of $\varphi_\gamma^{5k}(\baseloop)$, this proves the claim.
\end{proof}

We remark that our proof of Corollary~\ref{cor:intersectionNumbers} is essentially an elaboration of the proof of Kra's theorem given by Farb and Margalit in \cite[Theorem 14.6]{FarbMargalit}. Indeed, their technique of point-pushing in the universal cover provided both the inspiration and the foundation for our above analysis of the intersection numbers $\intnum(\varphi_\gamma^k(\baseloop),\pants)$. In light of the connection between dilatation and intersection numbers (Theorem~\ref{thm:pAIntersectionNumber}), this analysis of $\intnum(\varphi_\gamma^k(\baseloop),\pants)$ easily implies a lower bound on the dilatation of $\varphi_\gamma$. 

\begin{mythm}[The lower bound]\label{thm:lowerBound}
Let $S = S_{g,n}$ be a surface satisfying $3g + n > 3$, and let $\mu\colon[0,1]\to S$ be a closed filling curve on $S$ based at $p =\mu(0)$. Then the dilatation $\pushdil{\mu}$ of the mapping class $\push(\mu)\in\Mod(S,p)$ is bounded below as follows:
\begin{itemize}
\item[i)] If $S = S_{0,4}$ or $S_{1,2}$ and $\mu$ is the square of a primitive element in $\pi_1(S)$, then $\pushdil{\mu} \geq \sqrt[5]{\intnum(\mu)}$.
\item[ii)]  If $S = S_{1,1}$ and $\mu$ is the second, third, or fourth power of a primitive element, then $ \pushdil{\mu} \geq \sqrt[5]{(\intnum(\mu)+1)/2}$.
\item[iii)] In all other cases, $\pushdil{\mu} \geq \sqrt[5]{\intnum(\mu)+1}$.
\end{itemize}
\end{mythm} 
\begin{proof}
Decompose $\mu$ as a power $\mu =\gamma^m$, $m\geq 1$, of some primitive filling curve $\gamma\in \pi_1(S,p)$. Let $\varphi_\gamma$ denote a representative homeomorphism for $\push(\gamma)$ and let $\baseloop\subset S\setminus\{p\}$ be any essential simple closed curve on $S$. After choosing a pants decomposition $\pants$ as in Corollary~\ref{cor:intersectionNumbers}, it follows that the simple closed curves $\varphi_\gamma^k(\baseloop)\subset(S,p)$ satisfy $\intnum(\varphi_\gamma^{5k}(\baseloop),\pants)\geq (\intnum(\gamma)+1)^{k-1}$ for all $k\geq 1$. Upon manipulating this inequality, we find that
\[
\left(\frac{(\intnum(\gamma)+1)^{\nicefrac{1}{5}}}{\lambda_\gamma}\right)^{5k}
\leq (\intnum(\gamma)+1)\frac{\intnum(\varphi_\gamma^{5k}(\baseloop),\pants)}{\lambda_\gamma^{5k}}
= (\intnum(\gamma)+1)\sum_{\seam\in\pants}\frac{\intnum(\varphi_\gamma^{5k}(\baseloop),\seam)}{\lambda_\gamma^{5k}}.
\]
Theorem~\ref{thm:pAIntersectionNumber} implies that the rightmost expression has a finite limit. Therefore the leftmost expression remains bounded as $k$ tends to infinity, which is only possible if $\pushdil{\gamma}\geq (\intnum(\gamma)+1)^{\nicefrac{1}{5}}$. This proves the theorem when $\mu = \gamma$ is a primitive element of $\pi_1(S)$.

It remains to bound the dilatation $\pushdil{\mu} = \pushdil{\gamma}^m$ in the case that $m\geq 2$. Isotope $\gamma$ to attain the minimum intersection number $\intnum(\gamma) > 0$ in Definition~\ref{defn:selfIntNum}. By building a single closed curve out of $m$ offset copies of this loop, we obtain a representative for $\mu=\gamma^m$ that has $m^2\intnum(\gamma)+(m-1)$ transverse self intersections. This gives an upper bound
\begin{equation}\label{eqn:muIntBound}
\intnum(\mu)+1 \leq m^2\intnum(\gamma)+m
\end{equation}
on the self-intersection number of $\mu$. To relate this to the dilatation of $\push(\mu)$, we need to consider several inequalities involving the numbers $\intnum(\gamma)$ and $m$. Firstly, the inequality
\begin{equation}\label{eqn:numberInequality}
m^2k +m \leq (k+1)^m
\end{equation}
holds for all integers $m\geq 2$ and $k\geq 3$. Secondly, thinking of $\gamma\subset S$ as a four-valent graph, we have that
\begin{equation}\label{eqn:intnumVsEuler}
-\intnum(\gamma) = \chi(\gamma) \leq \chi(S_{g,n}) = 2-2g-n,
\end{equation}
where this inequality is strict in the case that $S$ is closed. The following three cases now account for all surfaces $S_{g,n}$ satisfying $3g+n > 3$.
\bigskip

\noindent
\textbf{Case 1:} \emph{$\chi(S)\leq -3$ or $S= S_{2,0}$.} In this case \eqref{eqn:intnumVsEuler} implies that $\intnum(\gamma)\geq 3$ (note that $S_{2,0}$ is closed). Combining \eqref{eqn:muIntBound} and \eqref{eqn:numberInequality} then yields the desired inequality
\begin{equation*}
\pushdil{\mu} = \pushdil{\gamma}^m \geq (\intnum(\gamma)+1)^{\nicefrac{m}{5}} \geq \left(m^2\intnum(\gamma) + m\right)^{\nicefrac{1}{5}} \geq \sqrt[5]{\intnum(\mu)+1}.
\end{equation*}

\noindent
\textbf{Case 2:} \emph{$S = S_{0,4}$ or $S= S_{1,2}$.} In this case \eqref{eqn:intnumVsEuler} only ensures that $\intnum(\gamma)\geq 2$. When $k=2$, the inequality \eqref{eqn:numberInequality} remains valid provided that $m\geq 3$. In such cases, we obtain the bound $\pushdil{\mu}\geq (\intnum(\mu)+1)^{\nicefrac{1}{5}}$ as above. When $m=2$, the modified inequality $2^2k +2 - 1 \leq (k+1)^2$ holds provided $k\geq 2$. Combining this with \eqref{eqn:muIntBound}, we find that 
\begin{equation*}
\pushdil{\mu} = \pushdil{\gamma}^2 \geq (\intnum(\gamma)+1)^{\nicefrac{2}{5}}
\geq (2^2\intnum(\gamma) + 2 -1)^{\nicefrac{1}{5}}  \geq \sqrt[5]{\intnum(\mu)}
\end{equation*}
in the case that $\mu$ is the square of a primitive element of either $\pi_1(S_{0,4})$ or $\pi_1(S_{1,2})$.
\bigskip

\noindent
\textbf{Case 3:} \emph{$S = S_{1,1}$.} The Euler characteristic now guarantees that $\intnum(\gamma) \geq -\chi(S_{1,1}) = 1$. When $k=1$, \eqref{eqn:numberInequality} still holds provided that $m\geq 5$; therefore we may conclude the general bound $\pushdil{\mu}\geq (\intnum(\mu)+1)^{\nicefrac{1}{5}}$ in these cases. For $2 \leq m \leq 4$, we instead have the inequality $(m^2k+m)/2 \leq (k+1)^m$, which holds for all $k\geq 1$. Together with \eqref{eqn:muIntBound}, this shows that 
\begin{equation*}
\pushdil{\mu} = \pushdil{\gamma}^m\geq (\intnum(\gamma)+1)^{\nicefrac{m}{5}} \geq \left(\frac{m^2\intnum(\gamma) + m}{2}\right)^{\nicefrac{1}{5}}\geq \sqrt[5]{\frac{\intnum(\mu)+1}{2}}
\end{equation*}
in the case that $\mu$ is the second, third, or fourth power of a primitive element in $\pi_1(S_{1,1})$.
\end{proof}

\begin{rem}
The above complications due to non-primitive loops are unfortunate but unavoidable. If $\mu$ is any closed loop that realizes the minimum $\intnum(\mu)$ and has transverse self-intersections, then $\intnum(\mu)$ is exactly equal to the number of distinct lifts $\tilde{\mu}$ that a single path lift $\lpath{x}{\mu}$ intersects. However if $\mu$ is the $k^{\text{th}}$ power of a primitive element, then the lifts of $\mu$ are grouped into families of $k$ ``parallel'' lifts with the same endpoints at $\partial\hyp$. Thus, following along a path $\lpath{x}{\mu}$ and turning onto various other lifts $\tilde{\mu}$ does not lead to exponential branching out in $\tree$, as many of these lifts now fellow-travel in $\tree$ forever.
\end{rem}

\section{An invariant pretrack}\label{sec:trainTracks}

Train tracks are an invaluable tool in the study of pseudo-Anosov homeomorphisms, and they will play an essential role in in our investigation. While it is a nontrivial matter to find an invariant train track for an arbitrary pseudo-Anosov map (there is, however, an algorithm due to Bestvina and Handel \cite{BestvinaHandel95}), \S\ref{sec:invariantTracks} describes a simple method for constructing invariant pretracks for pseudo-Anosov elements of the the point-pushing subgroup. We will use this construction to establish the upper bounds in Theorems~\ref{thm:generalBounds}, \ref{thm:leastDilatationFixedGenus}, and \ref{thm:leastDilatationVaryN}. We begin our discussion by recalling the relevant train track theory, which is developed more thoroughly in \cite{PennerHarer92, MosherMonograph, Penner91, PapadopoulosPenner87}.

\subsection{Preliminary train track theory}\label{sec:trainTrackPrelims}

A \emph{pretrack} on $S = S_{g,n}$ is a nonempty, smooth, closed $1$-complex $\tau \subset S$, whose edges are called \emph{branches} and whose vertices are called $\emph{switches}$, with the property that all branches incident on a given switch $v\in \tau$ share a common tangent line $L_v\leq T_vS$ at $v$; in this way, the branches incident at $v$ are divided into two sides depending on whether their tangent vectors at $v$ (oriented into the branch) are parallel or antiparallel. The local picture around a switch is shown in Figure~\ref{fig:switch}. The closure $C$ of a component of $S\setminus \tau$ is naturally a surface with some number $k\geq 0$ of cusps on its boundary, and we define the \emph{Euler index} of such a surface to be $\chi(C)-\nicefrac{k}{2}$. For instance, a \emph{$k$-gon}, that is, a topological disk with $k$ cusps on its boundary, has Euler index $\frac{2-k}{2}$. A $k$-gon with $k\leq 3$ will usually be referred to as a nullgon, monogon, bigon, or trigon.

\begin{figure}
\begin{minipage}[b]{0.48\linewidth}
\centering
\labellist\small\hair 2pt
\pinlabel {$b_1$} [tr] at 49 101
\pinlabel {$b_2$} [br] <0pt,1pt> at 43 46
\pinlabel {$b_3$} [br] <0pt,1pt> at 167 72
\endlabellist
\includegraphics[scale=1.0]{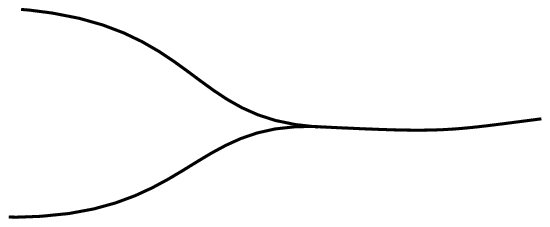}
\caption{A switch.}
\label{fig:switch}
\end{minipage}
\begin{minipage}[b]{0.48\linewidth}
\centering
\includegraphics[scale=1.0]{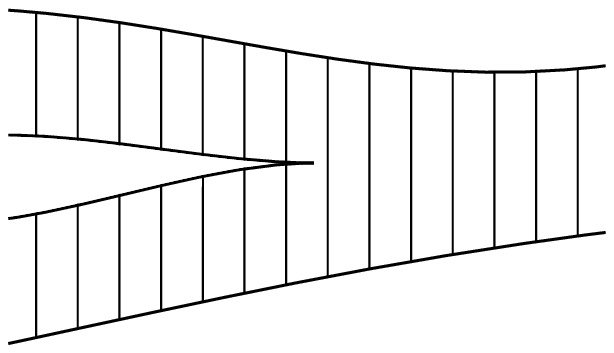}
\caption{The local tie neighborhood.}
\label{fig:tieNeighborhood}
\end{minipage}
\end{figure}

A \emph{train track} on $S$ is simply a pretrack whose complementary components all have negative Euler index. This amounts to ruling out complementary nullgons, monogons, bigons, smooth annuli, and once-punctured nullgons. In the context of a marked surface $(S,p)$, the marked point $p$ counts as a puncture and will be treated as such. We will use the term ``track'' to refer to both pretracks and train tracks. 

A \emph{weight function} on a track $\tau$ is an assignment of a nonnegative real number to each branch of $\tau$ in such a way that the net weights incident on either side of each switch agree. For instance, if the weight $w_i$ is assigned to branch $b_i$ in Figure~\ref{fig:switch}, then these weights must satisfy the equation $w_1 + w_2 = w_3$. The set of weight functions on $\tau$ is denoted by $\weights{\tau}$; it is a convex cone in $\R^B$, where $B$ is the set of branches in $\tau$. 

Let $\MFS(S)$ denote the space of equivalence classes of measured foliations on $S$; see \cite{FLP79} for the theory of measured foliations. Because of the train track condition on complementary components, there is natural injection $\rho\colon\weights{\tau}\to \MFS(S)$ from the set of weights on a train track $\tau$ onto a convex cone $\cone{\tau}\subseteq \MFS(S)$ consisting of those measured foliations which are ``carried'' by $\tau$ \cite[pp.~360--361]{PapadopoulosPenner87}. As every measured foliation is carried by some train track, the cones $\cone{\tau}$ are sometimes regarded as parameterized coordinate patches in $\MFS(S)$. A pretrack that fails to be a train track only due to the existence of complementary bigons will be called a \emph{bigon track}. The natural function $\rho\colon\weights{\tau}\to\cone{\tau}$ still makes sense for a bigon track, but it may fail to be injective \cite[p.~183]{Penner88}.

Associated to $\tau$ is a \emph{local tie neighborhood} $N\subset S$; this is a small neighborhood of $\tau$ equipped with a retraction $N\to\tau$ whose fibers form a foliation of $N$ by \emph{ties} that are transverse to $\tau$, as in Figure~\ref{fig:tieNeighborhood}. If $\sigma$ is another track on $S$, then $\tau$ \emph{carries} $\sigma$, denoted $\sigma \prec \tau$, if $\sigma$ may be smoothly isotoped into $N$ while remaining transverse to the ties. Such an isotopy $\Phi_t\colon S\to S$ with $\Phi_1(\sigma)\subset N$ is called a \emph{supporting map for the carrying} $\sigma\prec\tau$; it defines a corresponding \emph{incidence matrix} $M = \left(M_{ij}\right)$ as follows: For each branch $b_i$ of $\tau$ choose a distinguished fiber $x_i\subset N$ over an interior point of $b_i$. Then, for each branch $c_j$ of $\sigma$, set $M_{ij} = \abs{\Phi_1\inv(x_i)\cap c_j}$ to be the number of times $\Phi_1(c_j)$ crosses the distinguished tie $x_i$. Although the incidence matrix $M$ depends on the supporting map $\Phi_t$, the matrix nevertheless induces a canonical linear transformation $M\colon\weights{\sigma}\to\weights{\tau}$ from the set of weight functions on $\sigma$ to the set of weight functions on $\tau$. In the case that $\sigma$ and $\tau$ are train tracks, the carrying $\sigma\prec \tau$ implies that $\cone{\sigma}\subseteq \cone{\tau}$, and any incidence matrix $M$ describes the corresponding transition function between these parameterizations of $\cone{\sigma}$ \cite[p.~362]{PapadopoulosPenner87}.

Because of the switch conditions, a weight function $\mu\in \weights{\tau}$ may be specified by its values on a proper subset $B'\subset B$ of the branches of $\tau$; for example, in the situation of Figure~\ref{fig:switch}, the weight $\mu(b_3)$ is determined by the values of $\mu(b_1)$ and $\mu(b_2)$. This means that the natural projection $\R^{B}\to \R^{B'}$ is injective on $\weights{\tau}\subset\R^{B}$. If $\weights{\tau}'\subset \R^{B'}$ denotes image of $\weights{\tau}$, then inverting this projection gives a linear bijection $A_\tau\colon \weights{\tau}'\to \weights{\tau}$. 
In the case of a carrying $\sigma\prec \tau$ with incidence matrix $M$, we can make use of these bijections and instead consider the $\card{B_\tau'}\times\card{B_\sigma'}$ matrix $M' = A_\tau\inv M A_\sigma$. This smaller matrix gives a linear transformation $M'\colon \weights{\sigma}'\to \weights{\tau}'$ that contains all of the information of the carrying. To ease calculations, we will work with incidence matrices of this smaller form.

In addition to isotopy, we will make use of three elementary moves on a track $\tau$ which produce a new track $\tau'$ that carries $\tau$. The moves are illustrated in Figure~\ref{fig:moves}, and they consist of: sliding one switch past another, collapsing a bigon, or pinching branches together in the manner illustrated. For each move there is a natural choice of supporting map for the carrying $\tau \prec \tau'$ whose corresponding incidence matrix has the obvious effect on weights.

\begin{figure}
\centering
\labellist\small\hair 2pt
\pinlabel {slide} [B] <0pt,6pt> at 100 158
\pinlabel {pinch} [B] <0pt,6pt> at 332 110
\pinlabel {collapse} [B] <0pt, 6pt> at 100 62
\endlabellist
\includegraphics[scale=1.0]{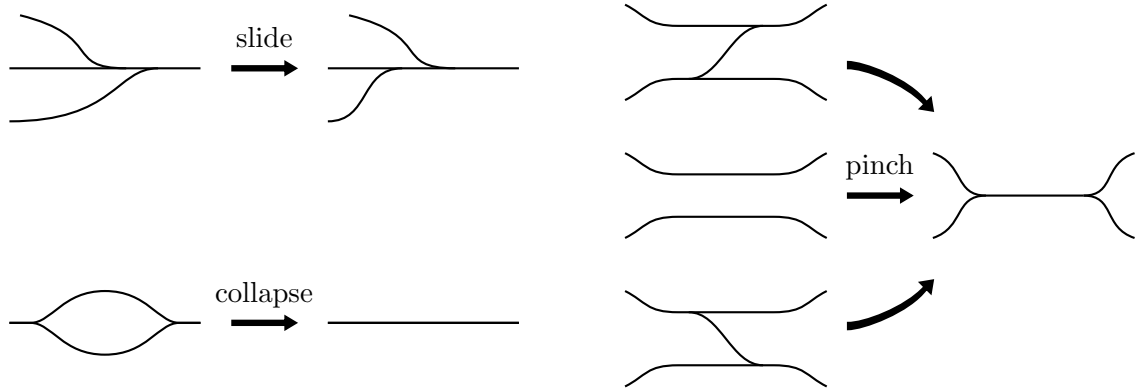}
\caption{The elementary carrying moves.}
\label{fig:moves}
\end{figure}

If $f\in \Mod(S)$ is a pseudo-Anosov mapping class and $\tau$ is a track on $S$, then the image $f(\tau)$ is well-defined up to isotopy. We say that $\tau$ is an \emph{invariant track} for $f$ if $f(\tau)\prec \tau$. If, additionally, $\tau$ is a train track or bigon track, then any incidence matrix $M$ for the carrying $f(\tau)\prec \tau$ describes the induced map $f_*\colon\MFS(S)\to \MFS(S)$ in the coordinate chart $\rho\colon\weights{\tau}\to\cone{\tau}$, that is, we have $f_*(\rho(\mu)) = \rho(M\mu)$ for any weight function $\mu\in \weights{\tau}$ \cite[p.~444]{Penner91}. The projective class of the unstable measured foliation $\mathscr{F}_+$ of $f$ is an attracting fixed point for the action of $f_*$ on the space $\mathbb{P}(\MFS(S))$ of projective classes of measured foliations. Since $f_*$ preserves $\cone{\tau}$, it follows that the projectivized coordinate chart $\mathbb{P}(\cone{\tau})$ contains sequences that converge to the projective class of $\mathscr{F}_+$. Since $\cone{\tau}$ is closed, this implies that $\mathscr{F}_+$ is contained in $\cone{\tau}$ and necessarily corresponds to an eigenvector of $M$.  In particular, the dilatation of $f$ is an eigenvalue of the incidence matrix $M$.

A square integer matrix $A$ is \emph{Perron--Frobenius} if it has nonnegative entries and some power $A^k$ has strictly positive entries (such matrices are also known as ``primitive irreducible''). In this case, the eigenvalue of $A$ with maximum modulus is positive real and its corresponding eigenvector has strictly positive entries \cite[Ch XIII \S2 Theorem 2]{Gantmacher59v2}. It follows that the modulus of any eigenvalue of a Perron--Frobenius matrix is bounded above by the largest row-sum of the matrix. Applying this classical result to the case of a Perron--Frobenius incidence matrix $M$, we may conclude the following key lemma.

\begin{lem}\label{lem:PerronFrobenius}
Let $f$ be a pseudo-Anosov mapping class, and suppose that $\tau$ is an invariant train track or bigon track for $f$. If $M$ is a Perron--Frobenius incidence matrix for the carrying $f(\tau)\prec \tau$, then the dilatation of $f$ is bounded above by the largest row-sum of $M$.
\end{lem}

\subsection{Invariant tracks for point-pushing homeomorphisms}\label{sec:invariantTracks}

In this subsection we describe a simple procedure for producing a pretrack from a curve; this construction will be used in \S\ref{sec:biggest} and \S\ref{sec:upperbounds} to analyze explicit examples and prove the upper bounds in Theorems~\ref{thm:generalBounds}, \ref{thm:leastDilatationFixedGenus} and \ref{thm:leastDilatationVaryN}. Let $\gamma\colon[0,1]\to S$ be a smooth curve on $S$ with $\gamma(0)=\gamma(1)=p$. We say that such a loop is \emph{generic} if it is simple except for finitely many transverse double-intersection points in the interior of $\gamma$ (i.e., not at $p$). If $\gamma$ is generic and $q\in S$ is a self-intersection point of $\gamma$, we let
$\itime{q}{1}$ and $\itime{q}{2}$ denote the two preimages of $q$ under $\gamma$, that is, we have $\gamma\inv(q) = \{\itime{q}{1},\itime{q}{2}\}$ with $0  < \itime{q}{1} < \itime{q}{2}<  1$.

Notice that a generic loop $\gamma\subseteq S$ is naturally a smooth, closed $1$-complex that only fails to be a pretrack because its intersection points are transverse rather than tangential. Thus we can build a pretrack that is intrinsically related to $\gamma$ by simply adjusting this $1$-complex around its intersection points to ensure that it satisfies the tangential condition at switches.

Locally around an intersection point $q$, the curve $\gamma$ cuts $S$ into four quadrants which have corners incident at $q$ and boundaries given by arcs of $\gamma$. The quadrant whose two boundary edges agree with the tangent vectors $\gamma'(\itime{q}{1})$ and $\gamma'(\itime{q}{2})$ is the \emph{outbound quadrant}, and its diagonal opposite is the \emph{inbound quadrant}; the situation is depicted in Figure \ref{fig:transverseIntersection}.

\begin{figure}
\centering
\subfigure[A self-intersection point of $\gamma$.]
{
\labellist\small\hair 2pt
\pinlabel {\shortstack[c]{ {\scriptsize \it inbound} \\ {\scriptsize \it quadrant}}} [tr] at 80 43
\pinlabel {\shortstack[c]{ {\scriptsize \it outbound}\\{\scriptsize \it quadrant}}} [bl] at 111 90
\pinlabel {$\gamma$} [br] <0pt,-1pt> at 47 58
\pinlabel {$\gamma$} [tl] <2pt,0pt> at 104 17
\pinlabel {$q$} [tr] <1pt,-2pt> at 102 71
\pinlabel* {$\gamma'(\itime{q}{1})$} [tr] <3pt,-3pt> at 174 78
\pinlabel* {$\gamma'(\itime{q}{2})$} [tl] <2pt,2pt> at 99 130
\endlabellist
\includegraphics[scale=1.0]{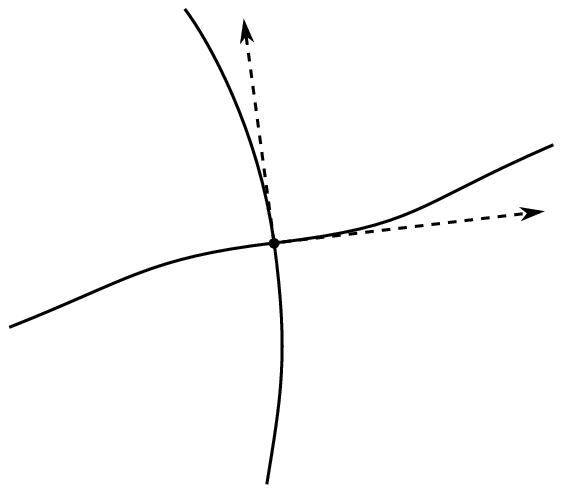}
\label{fig:transverseIntersection}
}
\subfigure[The local pretrack $\pretrack{\gamma}$ around $q$.] 
{
\labellist\small\hair 2pt
\pinlabel {\shortstack[c]{ {\scriptsize \it inbound} \\ {\scriptsize \it quadrant}}} [tr] at 80 43
\pinlabel {\shortstack[c]{ {\scriptsize \it outbound}\\{\scriptsize \it quadrant}}} [bl] at 117 96
\pinlabel {$q$} [tr] at 102 71
\pinlabel {$a_1^-$} [br] <0pt,-2pt> at 76 71
\pinlabel {$a_1^+$} [tl] <-2pt,1pt> at 139 79
\pinlabel {$a_2^-$} [l] <1pt,0pt> at 106 33
\pinlabel {$a_2^+$} [bl] at 96 108
\endlabellist
\includegraphics[scale=1.0]{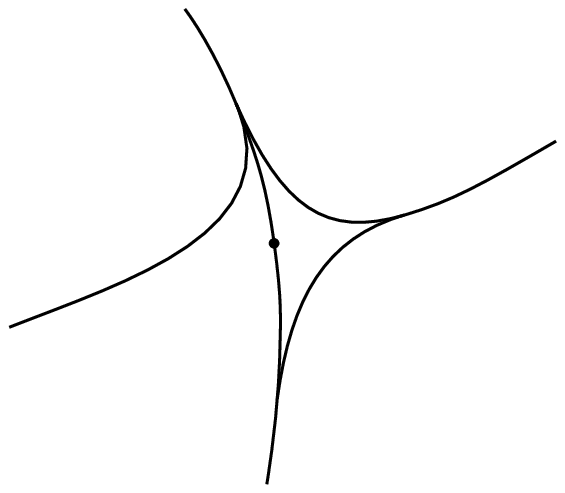}
\label{fig:localPretrack}
}
\caption{Constructing a pretrack $\pretrack{\gamma}$ from a generic curve $\gamma$.}
\label{fig:pretrackConstruction}
\end{figure}

We now describe how to adjust the $1$-complex $\gamma$ around $q$ to obtain a pretrack. The path $\gamma$ crosses $q$ twice, first at time $\itime{q}{1}$ and then at $\itime{q}{2}$. For a sufficiently small $\eps > 0$, we consider the four nearby points $a_{i}^{\pm} = \gamma(\itime{q}{i}\pm\eps)$ on the edges of $\gamma$ incident at $q$. Add three short, curved segments connecting the three pairs of points $(a_1^-,a_2^+)$, $(a_1^+,a_2^-)$, and $(a_1^+,a_2^+)$; this has the effect of cutting off the corner of every quadrant except for the inbound quadrant. Removing the segment of $\gamma$ between $a_1^-$ and $a_1^+$, we obtain the local pretrack illustrated in Figure \ref{fig:localPretrack}; it has a trigon located at $q$, the inbound quadrant has a cusp, and the other quadrants have smooth corners. Notice that this $1$-complex is still tangent to $\gamma'(\itime{q}{2})$ but is no longer tangent to $\gamma'(\itime{q}{1})$. 

After making these adjustments at each self-intersection point, we obtain a pretrack on the surface $S$. The last step is to create a track on $S\setminus\{p\}$. Consider the branch that passes through $p$ and split it (e.g., at $\gamma(1-\eps)$ and $\gamma(0+\eps)$) into a bigon around $p$. Finally, as illustrated in Figure~\ref{fig:basedPretrack}, add a smooth arc across the front of the bigon separating it into a trigon and a monogon containing $p$. We refer to this section of track that surrounds $p$, i.e., the section shown in Figure~\ref{fig:basedPretrack}, as the \emph{eye} of the track.

\begin{figure}
\centering
\labellist\small\hair 2pt
\pinlabel {$p$} [tr] at 93 75
\pinlabel {$\gamma'(0)$} [b] <0pt,-1pt> at 110 79
\endlabellist
\includegraphics[scale=1.0]{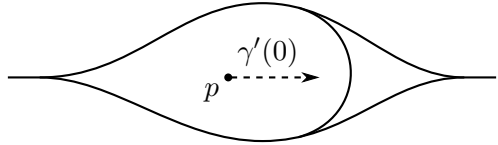}
\caption{Adjusting for the basepoint: the eye of the pretrack $\pretrack{\gamma}$.}
\label{fig:basedPretrack}
\end{figure}

\begin{defn}[Induced pretrack]\label{defn:inducedPretrack}
Let $\gamma\colon [0,1]\to S$ be a generic closed curve on the surface $S$. The corresponding pretrack, as constructed above, is called the \emph{pretrack induced by $\gamma$} and will be denoted by $\pretrack{\gamma}$. See Figures~\ref{fig:examplePiece} and \ref{fig:exampleTrack} for examples of such pretracks.
\end{defn}

\begin{prop}[Invariance of the induced pretrack] \label{prop:invariantPretrack}
Let $\gamma\colon[0,1]\to S$ be a generic loop representing a nontrivial element of the fundamental group $\pi_1(S,p)$. Then the induced pretrack
$\pretrack{\gamma}$ is invariant under the mapping class $\push(\gamma) = [\varphi_\gamma]$.
\end{prop}

\begin{proof}
Observe that the pretrack $\pretrack{\gamma}$ depends on the choice of basepoint $p\in S$. For $t\in [0,1]$, we may choose a reparameterization $\widehat{\gamma}$ of $\gamma$ based at the point $\widehat{\gamma}(0) = \gamma(t)$, and we let $\pretrack{t} = \pretrack{\widehat{\gamma}}$ denote the corresponding pretrack. This new track differs from $\pretrack{\gamma}$ in two important ways: Firstly, the eye of the track is now located at $\gamma(t)$ instead of at $p=\gamma(0)$. Secondly, the order of traversal at a self-intersection point $q$ might have changed; this would have the effect of reversing the orientation of the local picture around $q$---the inbound and outbound quadrants would be unchanged but the segment connecting $a_2^-$ and $a_2^+$ would be replaced by a segment joining $a_1^-$ and $a_1^+$ (see Figure~\ref{fig:localPretrack}). In this case we say that the branch containing $q$ has ``flipped'' in order to satisfy the condition that it is always transverse to the direction of travel for $\widehat{\gamma}$'s \emph{first} intersection with $q$.

Before proceeding with the proof, we highlight the key idea: pushing across a self-intersection point has the effect of moving the eye and flipping the branch containing that point. Since one full loop around $\gamma$ crosses each intersection point twice, each branch flips twice and there is no net effect. To make this precise, we argue as follows.

Recall from \S\ref{sec:weaving} that the point-pushing homeomorphism $\varphi_\gamma$ is obtained at the end of an isotopy $F_t\colon S\to S$ that pushes the point $p$ around $\gamma$ via the formula $F_t(p)=\gamma(t)$. Recall also that the pretrack $\pretrack{\beta}\subset S$ induced by a closed curve $\beta\colon[0,1]\to S$ is, by definition, contained in the punctured surface $S\setminus\{\beta(0)\}$. Therefore, for each $t\in[0,1]$, the reparameterized track $\pretrack{t}$ defined above satisfies $\pretrack{t}\subset S\setminus\{\gamma(t)\}$. Notice that we also have $F_t(\pretrack{\gamma})\subset S \setminus\{\gamma(t)\}$. 

Let $[s,s']\subseteq [0,1]$ be a time interval during which $F_t$ either pushes $p$ along an edge of $\gamma$ or pushes $p$ through a self-intersection point of $\gamma$. We will prove that if $F_{t}(\pretrack{\gamma})\prec \pretrack{t}$ in the punctured surface $S\setminus\{\gamma(t)\}$ when $t=s$, then the same is true when $t=s'$. Since $[0,1]$ may be covered by finitely many of these intervals, it will then follow that $\varphi_\gamma(\pretrack{\gamma}) = F_1(\pretrack{\gamma}) \prec \pretrack{1} = \pretrack{\gamma}$ in the punctured surface $S\setminus\{p\}$.

Strictly speaking, we should start with the track $F_{s}(\pretrack{\gamma})$ and apply the isotopy over the interval $[s,s']$ to obtain $F_{s'}(\pretrack{\gamma})$. Instead, we will use the given carrying $F_{s}(\pretrack{\gamma})\prec \pretrack{s}$ and start with the track $\pretrack{s}$. We then apply the isotopy to $\pretrack{s}$ and obtain a new track $\sigma$. Since we could have alternately taken note of the steps in the carrying and simply performed them after completing the isotopy, we see that $F_{s'}(\pretrack{\gamma})\prec \sigma$. Thus it suffices to show $\sigma\prec \pretrack{s'}$.

We first consider the case where $F_t$ pushes $p$ along an edge of $\gamma$ during the interval $[s,s']$. Consider a small neighborhood $U$ of $\gamma([s,s'])$ that does not contain any self-intersection points of $\gamma$. The isotopy may be chosen so that the complement of $U$ is unchanged throughout the interval $[s,s']$, that is, such that $F_t\circ F_{s'}\inv\vert_{S\setminus U}$ is equal to the identity for all $t\in[s,s']$. Furthermore, the pretrack $\pretrack{s}$ may be constructed so that the eye of $\pretrack{s}$ is contained in $U$. As the isotopy pushes $\gamma(s)$ along the path $\gamma([s,s'])$, we may assume that the eye of $\pretrack{s}$ retains its structure as it slides through $U$. Since the rest of the track remains unchanged, the resulting track at time $s'$ is exactly $\pretrack{s'}$. Thus $F_{s'}(\pretrack{\gamma})\prec \pretrack{s'}$. 

\begin{figure}
\centering
\subfigure[Coming up to an intersection point.]
{
\labellist\small\hair 2pt
\pinlabel {$p$} [r] <0pt,0pt> at 67 53
\pinlabel {$\drag$} [br] <1pt,-1pt> at 52 77
\pinlabel {$\lside$} [bl] <2pt,-1pt> at 94 58
\pinlabel {$\rside$} [tl] <1pt,0pt> at 92 45
\pinlabel {$\cross$} [r] at 152 50
\pinlabel {$\outcor$} [bl] <1pt,0pt> at 170 66
\pinlabel {$\sidecor$} [tl] at 172 44
\endlabellist
\includegraphics[scale=1.0]{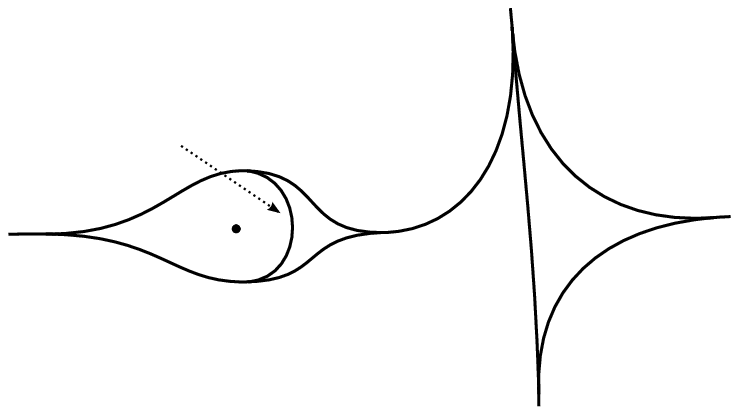}
\label{fig:beforeIntersection}
}
\subfigure[Push through the intersection point.]
{
\labellist\small\hair 3pt
\pinlabel {$p$} [r] at 157 56
\pinlabel {$\drag$} [b] <0pt,1pt> at 96 63
\pinlabel {$\lside$} [br] at 80 76
\pinlabel {$\rside$} [t] at 90 43
\pinlabel {$\cross$} [r] <1pt,2pt> at 107 24
\pinlabel {$\outcor$} [b] <-1pt,0pt> at 148 87
\pinlabel {$\sidecor$} [t] at 148 27
\endlabellist
\includegraphics[scale=1.0]{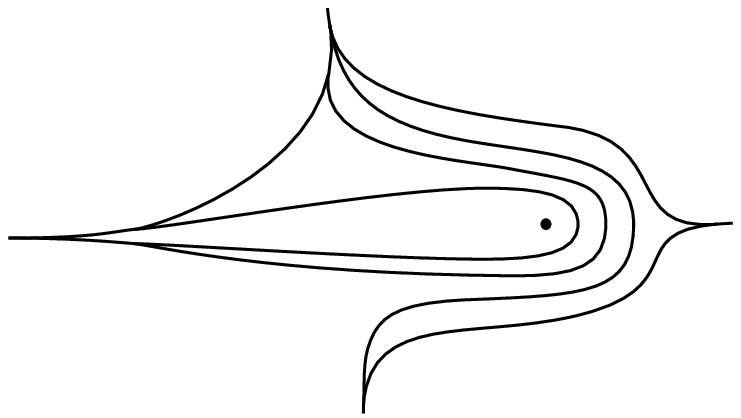}
\label{fig:throughIntersection}
}
\subfigure[Pinch branches together.]
{
\labellist\small\hair 2pt
\pinlabel {$p$} [r] <-1pt,0pt> at 159 56
\endlabellist
\includegraphics[scale=1.0]{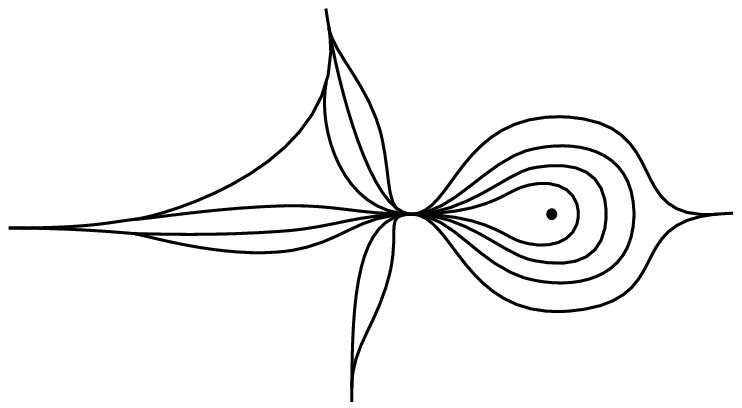}
\label{fig:throughIntersectionBigons}
}
\subfigure[Collapse the bigons.]
{
\labellist\small\hair 2pt
\pinlabel {$p$} [r] <-1pt,0pt> at 167 56
\pinlabel {$\drag'$} [r] <2pt,0pt> at 146 34
\pinlabel {$\lside'$} [bl] <0.5pt,0.5pt> at 190 60
\pinlabel {$\rside'$} [tl] <0pt,1pt> at 188 51
\pinlabel {$\cross'$} [t] <0pt,-1pt> at 81 57
\pinlabel {$\outcor'$} [bl] <1pt,-1pt> at 101 69
\pinlabel {$\sidecor'$} [br] <1pt,-1pt> at 75 71
\endlabellist
\includegraphics[scale=1.0]{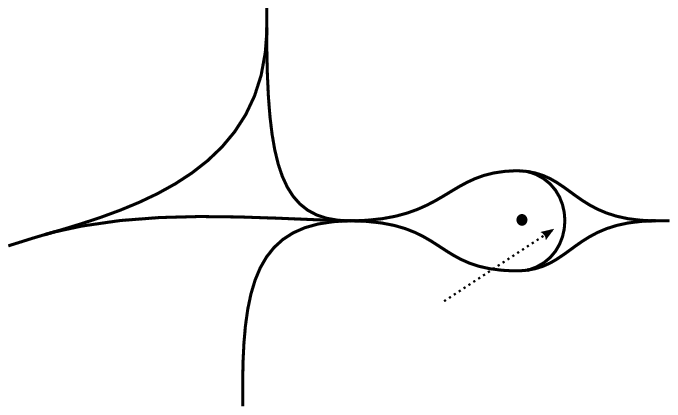}
\label{fig:afterIntersection}
}
\caption{Proving that $\pretrack{\gamma}$ carries $\push(\gamma)(\pretrack{\gamma})$: navigating a self-intersection point of $\gamma$.}
\label{fig:invarianceAtIntersection}
\end{figure}

It remains to examine the case where $F_t$ pushes $p$ through a self-intersection point $q$ during the interval $[s,s']$. 
At time $t=s$ we have the track $\pretrack{s}$, and the initial situation is as depicted in Figure~\ref{fig:beforeIntersection}. Pushing $p$ through $q$ results in the track $\sigma$ illustrated in Figure \ref{fig:throughIntersection}. After pinching several branches together as in Figure~\ref{fig:throughIntersectionBigons}, we may collapse the the resulting bigons to obtain the track $\sigma'$ shown in Figure~\ref{fig:afterIntersection}. Note that $\sigma\prec \sigma'$. A comparison of $\pretrack{s}$ and $\sigma'$ shows that pushing through $q$ has the effect of moving the eye to $\gamma(s')$ and flipping the branch containing $q$. Since changing the starting point from $\gamma(s)$ to $\gamma(s')$ switches the order of traversal at $q$---the direction we just pushed is the \emph{second} direction of traversal if we start at $\gamma(s')$---$\sigma'$ evidently satisfies the defining characteristics of $\pretrack{s'}$; whence $\sigma\prec \pretrack{s'}$.
\end{proof}

The proof of Proposition~\ref{prop:invariantPretrack} exhibits an explicit carrying $\varphi_{\gamma}(\pretrack{\gamma})\prec \pretrack{\gamma}$, and it is straightforward to determine the corresponding incidence matrix. Notice that any weight function is locally determined by its values on the six branches that are labeled in Figure~\ref{fig:beforeIntersection}. We identify these distinguished branches in the following way. In the eye of the pretrack, $\drag$ is the branch immediately in front of the marked point, while $\lside$ and $\rside$ form the left and right sides (in the direction of travel) of the trigon at the front of the eye. In the local picture around a self-intersection point $q$, $\cross_q$ is the branch containing $q$, $\outcor_q$ is the curved branch through the outbound quadrant, and $\sidecor_q$ is the curved branch forming the other side of this trigon.

Suppose that $w$ is a weight function on $\pretrack{\gamma}$, and let $q_1,\dotsc,q_n$ denote the $n = \intnum(\gamma)$ self-intersection points of $\gamma$. We say that $w_e = (w(\drag),w(\lside),w(\rside))$ is the weight vector at the eye, and that $w_i = (w(\cross_{q_i}),w(\outcor_{q_i}),w(\sidecor_{q_i}))$ is the weight vector at $q_i$. The weight function is then completely determined by its weight vectors $w_e,w_1,\dotsc,w_n$.

By keeping track of the weights throughout the carrying illustrated in in Figure~\ref{fig:invarianceAtIntersection}, one finds that pushing the marked point through $q_i$ transforms the weight vector $(w_e,w_1,\dotsc,w_n)$ according to the block matrix
\begin{equation*}
A_{i}^{\text{right}}=
\begin{pmatrix}
\left(\begin{smallmatrix}
1 & 0 & 1 \\
0 & 0 & 0 \\
0 & 0 & 0 \end{smallmatrix}\right)& 
\dotsb & 
\left(\begin{smallmatrix}
1 & 0 & 0\\
0 & 1 & 0\\
0 & 0 & 1 \end{smallmatrix}\right)&
\dotsb & 0\\
\vdots & \ddots & \vdots &  & \vdots\\
\left(\begin{smallmatrix}
2 & 0 & 1 \\
0 & 0 & 1 \\
0 & 1 & 0 \end{smallmatrix}\right)& 
\dotsb & 
\left(\begin{smallmatrix}
0 & 0 & 0\\
1 & 1 & 0\\
0 & 0 & 0 \end{smallmatrix}\right)&
\dotsb & 0 \\
\vdots & & \vdots & \ddots & \vdots\\
0 & \dotsb & 0 & \dotsb & I
\end{pmatrix}
\quad
\text{or}
\quad
A_{i}^{\text{left}} =
\begin{pmatrix}
\left(\begin{smallmatrix}
1 & 1 & 0 \\
0 & 0 & 0 \\
0 & 0 & 0 \end{smallmatrix}\right)& 
\dotsb & 
\left(\begin{smallmatrix}
1 & 0 & 0\\
0 & 0 & 1\\
0 & 1 & 0 \end{smallmatrix}\right)&
\dotsb & 0\\
\vdots & \ddots & \vdots & &  \vdots\\
\left(\begin{smallmatrix}
2 & 1 & 0 \\
0 & 1 & 0 \\
0 & 0 & 1 \end{smallmatrix}\right)& 
\dotsb & 
\left(\begin{smallmatrix}
0 & 0 & 0\\
1 & 1 & 0\\
0 & 0 & 0 \end{smallmatrix}\right)&
\dotsb & 0 \\
\vdots & & \vdots & \ddots & \vdots\\
0 & \dotsb & 0 & \dotsb & I 
\end{pmatrix}
\end{equation*}
depending on whether the inbound quadrant is on the right (as in Figure~\ref{fig:invarianceAtIntersection}) or left side of $p$. Here $A_i^{\text{right}}$ and $A_{i}^{\text{left}}$ are each the identity matrix except for the four indicated blocks in the $e$ and $i$ positions. The full incidence matrix for the carrying is then an appropriate product of these matrices.

\begin{cor}[Incidence matrix]\label{cor:incidenceMatrix}
Suppose that a generic closed curve $\gamma\colon[0,1]\to S$ crosses its $n=\intnum(\gamma)$ self-intersection points $\{q_i\}$ in the order $q_{i_1},q_{i_2},\dotsc,q_{i_{2n}}$ and that the handedness of the $j^{\text{th}}$ crossing is $\mathcal{O}_j\in\{\text{right},\text{left}\}$. Let $M_\gamma\colon\weights{{\pretrack{\gamma}}}\to \weights{{\pretrack{\gamma}}}$ be the linear transformation induced by the carrying $\varphi_\gamma(\pretrack{\gamma})\prec \pretrack{\gamma}$ described in Proposition~\ref{prop:invariantPretrack}. Then the action of $M_\gamma$ on a weight vector $(w_e,w_1,\dotsc,w_n)$ is given by the matrix product
\begin{equation*}
M_\gamma = A_{i_{2n}}^{\mathcal{O}_{2n}}\dotsm A_{i_1}^{\mathcal{O}_1}.
\end{equation*}
\end{cor}

\begin{rem}\label{rem:thePretrack}
It is an obvious drawback that $\pretrack{\gamma}$ is only a pretrack and not, in general, a train track. There is a beautiful algorithm, due to Bestvina and Handel \cite{BestvinaHandel95}, that will find an invariant train track for any pseudo-Anosov mapping class. However, in the case of a point-pushing map, it is not clear how the resulting track depends on the pushing curve. The point of our construction is that $\pretrack{\gamma}$ depends visibly on $\gamma$, and, as seen in Corollary~\ref{cor:incidenceMatrix}, the corresponding incidence matrix depends quantifiably on the self-intersection number $\intnum(\gamma)$. Thus $\pretrack{\gamma}$ provides a connection between $\intnum(\gamma)$ and the dilatation $\pushdil{\gamma}$. It would be interesting if our construction could be modified to produce a train track while maintaining these key features. 
\end{rem}

\section{The largest dilatations}\label{sec:biggest}

Having established a general lower bound on the dilatation of a point-pushing homeomorphism, we now turn our attention to upper bounds. In this section we use the train track theory developed in \S\ref{sec:trainTracks} to bound $\pushdil{\gamma}$ from above and complete the proof of Theorem~\ref{thm:generalBounds}.

\subsection{The image of a loop}\label{sec:loopImages}

As in the proof of the lower bound, we will estimate dilatation by studying the action on simple closed curves and counting intersection numbers. Rather than lifting to the universal cover, as we did in \S\ref{sec:lowerBound}, we will use train tracks to analyze curves directly on the surface. We begin with some general observations.

A simple closed curve can naturally be given the structure of a pretrack. Thus it makes sense to talk about simple closed curves being carried by a track. Let $f\in \Mod(S,p)$ be a pseudo-Anosov map, and suppose that $\tau$ is an invariant track for $f$. We then have the $\card{B}\times \card{B}$ incidence matrix $M$ for the carrying $f(\tau)\prec \tau$, where $B = \{b_i\}$ is the set of branches of $\tau$. If $\alpha\subset (S,p)$ is a simple closed curve that is carried by $\tau$, then the carrying $\alpha\prec \tau$ defines a $\card{B}\times 1$ incidence matrix that we think of as weight vector $v\in \weights{\tau}$. Since $\tau$ is invariant under $f$, it follows that $\tau$ carries $f^k(\alpha)$ for $k\geq 0$ and that the weight vector for the carrying $f^k(\alpha)\prec f^k(\tau)\prec \tau$ is given by $u = M^kv$. This means that $f^k(\alpha)$ is isotopic to a simple closed curve that is contained in the tie neighborhood of $\tau$ and intersects the central tie over $b_i$ exactly $u(b_i)$ times. Conversely, such a representative for $f^k(\alpha)$ may be constructed directly from the weight vector $u$: For each edge $b_i$ of $\tau$, place $u(b_i)$ disjoint segments running parallel to $b_i$. Because $u\in \weights{\tau}$ satisfies the switch conditions, the endpoints incident at each switch match up in pairs to produce a simple closed curve that is isotopic to $f^k(\alpha)$ in $(S,p)$.

Now suppose that $\eta\subset (S,p)$ is another simple closed curve. After adjusting $\eta$ by an isotopy, we may assume that its intersections with $\tau$ are all transverse and contained in the interiors of the branches $b_i$. Corresponding to this setup, we then have the $1\times \card{B}$ \emph{intersection matrix} $N = (N_{1j})$ given by $N_{1j} =\card{\eta\cap b_j}$. If $a\subset (S,p)$ is the simple closed curve constructed from the weight vector $u = M^kv$ as above, it follows that the cardinality of $\eta\cap a$ is given by the product $Nu$. In particular, for all integers $k\geq 0$, the matrix products
\begin{equation}\label{eqn:boundOnIntersections}
NM^kv = Nu = \card{\eta\cap a} \geq \intnum(\eta,f^k(\alpha))
\end{equation}
give convenient upper bounds on the intersection numbers of $f^k(\alpha)$ and $\eta$.

\subsection{An upper bound on dilatation}

We now carry out such an estimate in the case of a point-pushing pseudo-Anosov map. Let $\gamma\colon[0,1]\to S$ be a filling curve on an oriented surface $S = S_{g,n}$, where $3g+n> 3$. After adjusting $\gamma$ by an isotopy, we may assume it is generic and that it realizes the minimum self-intersection number $\intnum(\gamma)> 0$ in Definition~\ref{defn:selfIntNum}. It follows from Kra's theorem that the point-pushing homeomorphism $\varphi_\gamma$ is pseudo-Anosov, and we consider its induced invariant pretrack $\pretrack{\gamma}$ as constructed in Definition~\ref{defn:inducedPretrack}. In order to apply the above discussion and relate intersection numbers to dilatation, we need to find an essential simple closed curve that is carried by $\pretrack{\gamma}$.

Suppose that $\intnum(\gamma)=n$, and let $q_1,\dotsc,q_n\in S$ be the $n$ self-intersection points of $\gamma$. The path $\gamma\colon[0,1]\to S$ crosses each of these points twice, and they are ordered according to their first crossing times. Thus $q_n$ is the intersection point that is \emph{reached} last, and not necessarily the intersection point that is crossed last. Using $q_n\in \gamma$ as a break point, we form the decomposition $\gamma = \alpha\delta\beta$,
where $\alpha$ is the initial portion of $\gamma$ from $p=\gamma(0)$ to $q_n$, $\delta$ is the subsequent loop based at $q_n$, and $\beta$ is the final segment from $q_n$ back to $p=\gamma(1)$.

\begin{lem}\label{lem:simpleCurveCarrying}
The subloop $\delta\subset \gamma$ is a simple closed curve that is carried by $\pretrack{\gamma}$.
\end{lem}
\begin{proof}
For each $1\leq i \leq n$, let $0 < \itime{i}{1} < \itime{i}{2} < 1$ be the two preimages of $q_i$ under $\gamma$. Thus $\delta$ is the restriction of $\gamma$ to the interval $[\itime{n}{1},\itime{n}{2}]$. The ordering on the self-intersection points implies that $\itime{i}{1}<\itime{n}{1}$ for all $i < n$. Therefore $\delta$ is simple because it crosses each self-intersection point of $\gamma$ at most once.

To prove that $\delta\prec \pretrack{\tau}$, we must isotope $\delta$ into the tie neighborhood of $\pretrack{\gamma}$ while keeping $\delta$ transverse to the ties. We first argue that this condition already holds everywhere along $\delta$ except near its basepoint $q_n$. This is clear away from the self-intersection points of $\gamma$, so consider the situation near a self-intersection point $q_i$. There is nothing to prove unless $\delta$ crosses $q_i$, in which case we have $\itime{n}{1} < \itime{i}{2} < \itime{n}{2}$. Therefore $\delta$ crosses $q_i$ while travelling in the $\gamma'(\itime{i}{2})$ direction. Since, by Definition~\ref{defn:inducedPretrack}, $\pretrack{\gamma}$ contains a branch through $q_i$ that is tangent to this direction, it follows that $\delta$ is tangent to $\pretrack{\gamma}$ and that the carrying condition does hold near $q_i$.

It remains to adjust $\delta$ near its basepoint $q_n$ so that the carrying condition is satisfied. Notice that $\delta$ leaves $q_n$ going in the direction $\gamma'(\itime{n}{1})$ and returns travelling in the direction $\gamma'(\itime{n}{2})$. These two edges of $\gamma$ do not bound the inbound quadrant at $q_n$, so they are not separated by a cusp of $\pretrack{\gamma}$. Using the notation of Figure~\ref{fig:pretrackConstruction}, we find that one may isotope $\delta$ so that is starts at the point $a_1^+$, leaves in the direction $\gamma'(\itime{n}{1})$, travels once around $\delta$ to the point $a_2^-$, and then follows the curved branch of $\pretrack{\gamma}$ back to $a_1^+$. The resulting curve, as illustrated in Figure~\ref{fig:essentialCurve}, satisfies the carrying condition and proves that $\delta\prec\pretrack{\gamma}$.
\end{proof}

\begin{figure}
\centering
\labellist\small\hair 2pt
\pinlabel {$\sigma$} [br] at 221 36
\pinlabel {$\delta$} [br] at 170 51
\pinlabel {$\varepsilon$} [tr] <1pt,0pt> at 304 83
\pinlabel {$q_n$} [t] <-2pt,-0.5pt> at 182 23
\pinlabel {$\pretrack{\gamma}$} [br] at 111 90
\endlabellist
\includegraphics[scale=1.0]{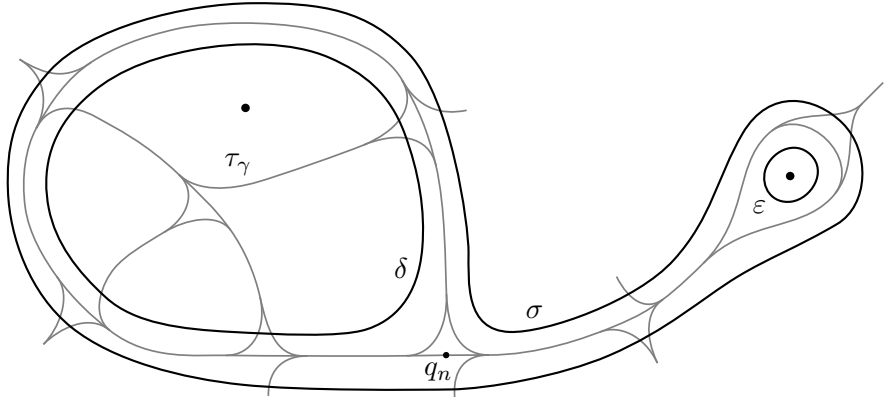}
\caption{The track $\pretrack{\gamma}$ carries both $\delta$ and $\sigma$; at least one of these curves is essential.}
\label{fig:essentialCurve}
\end{figure}

\begin{lem}\label{lem:aCurveExists}
There exists an essential simple closed curve $\sigma\subset (S,p)$ that is carried by $\pretrack{\gamma}$.
\end{lem}
\begin{proof}
Any nullhomotopy of $\delta$ would provide a homotopy between $\gamma=\alpha\delta\beta$ and $\alpha\beta$ that reduces the self-intersection number of $\gamma$. Since this is not possible, $\delta$ is nontrivial in $\pi_1(S)$. If $\delta$ is also essential, then we are done by Lemma~\ref{lem:simpleCurveCarrying}. Otherwise, $\delta$ is puncture-parallel and we construct an essential curve as follows.

The definition of $q_n$ ensures that the restriction of $\gamma$ to $(\itime{n}{1},1]$ crosses each self-intersection point $q_i$ at most once. Therefore $\beta = \gamma\vert_{[\itime{n}{2},1]}$ is a simple path from $q_n$ to $p$ whose interior is disjoint from $\delta$. Note also that $\beta$ is everywhere tangent to $\pretrack{\gamma}$ for the same reason that $\delta$ is. Let $\varepsilon\subset S$ be a simple closed curve around the marked point $p$; for example, $\varepsilon$ may be chosen to be the boundary of a small neighborhood of $p$. Using $\beta$ as a guide, form the connect sum $\sigma=\delta\#\varepsilon$ of $\delta$ with $\varepsilon$. More precisely, remove an interval from each loop and glue in two parallel copies of $\beta$ to form a single closed curve that is homotopic to $\delta\beta\varepsilon\bar{\beta}$. As illustrated in Figure~\ref{fig:essentialCurve}, the structure of the pretrack near $q_n$ and $p$ ensures that $\sigma$ is carried by $\pretrack{\gamma}$. Furthermore, $\sigma$ is simple because it is the union of four simple segments whose interiors are disjoint.

Recalling that the marked point $p$ counts as a puncture, we see that $\delta$ and $\varepsilon$ are both puncture-parallel, separating curves. Therefore $\sigma$ separates $S$ into two components, one of which is a twice-punctured disk. Since $S$ is neither $S_{0,2}$ nor $S_{0,3}$, the other component of $S\setminus \sigma$ cannot be a disk or a once-punctured disk. This proves that $\sigma$ is essential.
\end{proof}

Now that we have found an essential, simple closed curve that is carried by $\pretrack{\gamma}$, we may proceed to estimate intersection numbers and bound the dilatation $\pushdil{\gamma}$. Let $B = \{b_i\}$ be the set of branches of $\pretrack{\gamma}$, and let $M$ be the $\card{B}\times \card{B}$ incidence matrix for the carrying $\varphi_\gamma(\pretrack{\gamma})\prec \pretrack{\gamma}$ given in Proposition~\ref{prop:invariantPretrack}. Take $\sigma\subset S$ to be the essential curve from Lemma~\ref{lem:aCurveExists}, and let $v\in \weights{\pretrack{\gamma}}$ be the $\card{B}\times 1$ weight vector associated to the carrying $\sigma\prec \pretrack{\gamma}$.

Let $B'\subset B$ be the set of $3(\intnum(\gamma)+1)$ distinguished branches defined after the proof of Proposition~\ref{prop:invariantPretrack}, and let $\weights{\pretrack{\gamma}}'\subset \R^{B'}$ be the image of $\weights{\pretrack{\gamma}}$ under the projection $P\colon\R^{B} \to \R^{B'}$. Since any weight vector is determined by its values on these edges, there is a linear bijection $Q\colon\weights{\pretrack{\gamma}}'\to \weights{\pretrack{\gamma}}$ that inverts $P$ (see the discussion in \S\ref{sec:trainTrackPrelims}). This map may be realized, in a non-unique way, by a $\card{B}\times\card{B'}$ matrix $Q$ whose $i^{\text{th}}$ row expresses a weight function's value on $b_i$ as a linear combination of its values on the branches $b'\in B'$. Since the action of $\varphi_\gamma$ on $\weights{\pretrack{\gamma}}'$ is given by the $\card{B'}\times\card{B'}$ incidence matrix $M_\gamma$ in Corollary~\ref{cor:incidenceMatrix}, we evidently have that $M = QM_\gamma P$.

Let $v'= Pv$ denote the projection of $v$ to $\weights{\pretrack{\gamma}}'$. This may be expressed as a vector $v' = (v_1',\dotsc,v_l')$ of size $l = \card{B'}$ whose entries are nonnegative integers. According to Corollary~\ref{cor:incidenceMatrix}, $M_\gamma$ is a product of $2\cdot\intnum(\gamma)$ matrices of the form $A_i^{\text{right}}$ or $A_i^{\text{left}}$. This structure gives us good control on the size of $M_\gamma^kv'$ in terms of $k$ and $\intnum(\gamma)$.

\begin{lem}\label{lem:matrixBound}
Let $u = (u_1,\dotsc,u_l)$ be a vector of size $l =\card{B'}$, and let $c = \max_i\{u_i\}$ denote its largest entry. If $D$ is a product of $k\geq 0$ matrices of the form $A_i^{\text{right}}$ or $A_i^{\text{left}}$ given in Corollary~\ref{cor:incidenceMatrix}, then every entry of the vector $Du$ is bounded above by $3^kc$.
\end{lem}
\begin{proof}
If $A = (A_{ij})$ is any matrix of the specified form, then each row of $A$ has sum at most $3$. Assuming inductively that the claim holds for any given $k$-fold product $D$, we find that
\[\left(ADu\right)_i = \sum_j A_{ij}\left(Du\right)_j \leq \sum_j A_{ij}3^kc = 3^kc\sum_j A_{ij} \leq 3^kc\cdot 3.\]
Therefore the claim also holds for the $(k+1)$-fold product $AD$.
\end{proof}

We now give a general upper bound on the dilatation $\pushdil{\gamma}$. Together with Theorem~\ref{thm:lowerBound}, the following result completes the proof of Theorem~\ref{thm:generalBounds}.

\begin{thm}[The upper bound]\label{thm:upperBound}
Let $S = S_{g,n}$ be a surface satisfying $3g + n > 3$, and let $\gamma\colon[0,1]\to S$ be a filling loop based at $\gamma(0)=p$. Then the dilatation $\pushdil{\gamma}$ of the mapping class $\push(\gamma)$ is bounded above by $9^{\intnum(\gamma)}$.
\end{thm}
\begin{proof}
Retaining the notation of the preceding discussion, we see that for $k\geq 0$ the weight vector associated to the carrying $\varphi_\gamma^k(\sigma)\prec \pretrack{\gamma}$ is given by the matrix product
\begin{equation*}
M^k v = QM_\gamma^kv'.
\end{equation*}
Let $\eta\subset (S,p)$ be any essential simple closed curve that is transverse to the track $\pretrack{\gamma}$ and disjoint from the switches. We may then form the $1\times\card{B}$ intersection matrix $N = (N_{1j})$ whose entries are given by $N_{1j} = \card{\eta\cap b_j}$. It now follows from \eqref{eqn:boundOnIntersections} that the intersection numbers $\intnum(\varphi_\gamma^k(\sigma),\eta)$ are bounded above by the following matrix products
\begin{equation*}
\intnum(\varphi_\gamma^k(\sigma),\eta) \leq NM^kv = NQM_\gamma^kv'.
\end{equation*}

Let $d$ denote the sum of the entries in the $1\times\card{B'}$ matrix $NQ$. By Corollary~\ref{cor:incidenceMatrix}, $M_\gamma^k$ is a product of $2k\cdot\intnum(\gamma)$ matrices of the form $A_i^{\text{right}}$ or $A_i^{\text{left}}$. Therefore, we may apply Lemma~\ref{lem:matrixBound} and conclude that
\begin{equation*}
\intnum(\varphi_\gamma^k(\sigma),\eta) \leq \sum_j (NQ)_j (M_\gamma^kv')_j \leq 3^{2k\intnum(\gamma)}c\sum_j (NQ)_j = dc3^{2k\intnum(\gamma)},
\end{equation*}
where $c = \max\{v_i'\}$ is the largest entry in $v'$. Dividing by $\pushdil{\gamma}^k$, we see that the inequality
\[ \frac{\intnum(\varphi_\gamma^k(\sigma),\eta)}{\pushdil{\gamma}^k} \leq dc\left(\frac{9^{\intnum(\gamma)}}{\pushdil{\gamma}}\right)^k\]
holds for all integers $k\geq 0$. Since $\sigma$ and $\eta$ are both essential, simple closed curves, Theorem~\ref{thm:pAIntersectionNumber} implies that the left hand side has a positive limit as $k$ tends to infinity. It follows that right hand side is bounded away from zero, which necessitates $\pushdil{\gamma}\leq 9^{\intnum(\gamma)}$.
\end{proof}

\begin{rem}\label{rem:overCounting}
It seems unlikely that this upper bound is optimal. When $\tau$ is a  train track, the matrix products in \eqref{eqn:boundOnIntersections} actually grow like $\intnum(f^k(\alpha),\eta)$ and $\dilat{f}^k$. However, if $\tau$ has complementary monogons, then the loops constructed from the weight vectors $M^kv$ will be very inefficient representatives of $f^k(\alpha)$, and the matrix products $NM^kv$ will drastically overestimate the intersection numbers $\intnum(f^k(\alpha),\eta)$. Indeed, in order to minimize the cardinality of $f^k(\alpha)\cap \eta$, one would need to straighten the representative loop $f^k(\alpha)$ by pulling its strands across all of the complementary monogons.

In the case of a pushing curve $\gamma\subset S$ with very high self-intersection number, Euler characteristic considerations imply that $\pretrack{\gamma}$ will have on the order of $\chi(S) + \intnum(\gamma)/2$ complementary monogons and nullgons. Consequently, we expect that our upper bound is far from sharp when $\intnum(\gamma)$ is large.
\end{rem}

\begin{rem}\label{rem:whyTracks}
The above remark illustrates why the pretrack $\pretrack{\gamma}$ is not able to provide a general lower bound on the dilatation $\pushdil{\gamma}$. Namely, without the train track condition on complementary components, the pretrack $\pretrack{\gamma}$ does not aid in calculating the infimum in the Definition~\ref{defn:intNum} of intersection number. The pretrack $\pretrack{\gamma}$ does provide an upper bound on $\pushdil{\gamma}$ precisely because one can bound $\intnum(\varphi_\gamma^k(\baseloop),\eta)$ from above without taking infimums.
\end{rem}

The following simple argument shows that any general upper bound must be at least on the order of $\exp(C\intnum(\gamma)^{\nicefrac{1}{2}})$ for some constant $C$. In particular, the optimal upper bound on $\pushdil{\gamma}$, in terms of $\intnum(\gamma)$, must lie somewhere between $\exp(C\intnum(\gamma)^{\nicefrac{1}{2}})$ and $\exp(\log(9)\intnum(\gamma))$.

\begin{prop}\label{prop:orderOfUpperBound}
There exists an infinite family $\{\push(\gamma_k)\}$ of point-pushing pseudo-Anosovs and a constant $C$ such that $\intnum(\gamma_k)\to \infty$ and $\pushdil{\gamma_k}\geq \exp(C\sqrt{\intnum(\gamma_k)})$.
\end{prop}
\begin{proof}
Let $\mu\subset S$ be any filling curve on $S$ and let $\lambda = \pushdil{\mu}$ be the dilatation of the pseudo-Anosov mapping class $\push(\mu)$. For example, $\mu$ may be chosen to fill $S$ efficiently so that $\intnum(\mu)$ is relatively small. Powers $\gamma_k = \mu^k$ of $\mu$ may then be represented by closed loops that are built by concatenating $k$ offset copies of $\mu$. These offset copies contribute $k^2$ intersection points for each self-intersection point of $\mu$, so we have the upper bound
\[\intnum(\gamma_k) = \intnum(\mu^k) \leq k^2\intnum(\mu)+(k-1)\leq 2k^2\intnum(\mu).\]
Since the dilatation of $\push(\gamma_k)$ is given by $\pushdil{\gamma_k} = \lambda^k$, the above inequality implies that
\[\pushdil{\gamma_k} = \lambda^k \geq \lambda^{\sqrt{\intnum(\gamma_k)/2\intnum(\mu)}} = \left(\lambda^{\nicefrac{1}{\sqrt{2\intnum(\mu)}}}\right)^{\sqrt{\intnum(\gamma_k)}}.\qedhere\]
\end{proof}

\section{Bounds on least dilatations}\label{sec:upperbounds}

It remains to consider the least dilatations attained in the point-pushing subgroup $\pp(S)$. In this section, we use the tools of \S\ref{sec:trainTracks} to examine two concrete examples on the closed surface $S_g$ of genus $g$. These calculations will prove Theorems~\ref{thm:leastDilatationFixedGenus} and \ref{thm:leastDilatationVaryN}.

\subsection{A low-dilatation example}

Recall from \eqref{eqn:leastDilatation} that $\least(\pp(S)) = \inf\left\{\spec(\pp(S))\right\}$ is the the least element in the spectrum of all entropies of point-pushing pseudo-Anosov homeomorphisms on the surface $S$. Our goal in this subsection is to bound the least dilatation $\least(\pp(S_g))$ on the closed surface of genus $g$ and to estimate its asymptotic dependence on genus. More precisely, we prove the following.

\renewcommand{\thethmref}{\ref{thm:leastDilatationFixedGenus}}
\begin{thmref}
For the closed surface $S_g$ of genus $g\geq 2$, we have
\[ \tfrac{1}{5}\log(2g) \leq \least(\pp(S_g)) < g\log(11).\]
\end{thmref}
\begin{proof}
As one may easily check, the inequality \eqref{eqn:EulerCharacteristicBound} is strict in the case of a closed surface. Therefore the inequality in Corollary~\ref{cor:boundOnLeast} may be improved to give the lower bound $\least(\pp(S_g)) \geq \tfrac{1}{5}\log(2g)$. To bound $\least(\pp(S_g))$ from above, it suffices to consider an example.

\begin{exmp}\label{ex:fixedSurface}
Consider the surface $S_2$ of genus $2$, and let $\gamma_0\colon[0,1]\to S_2$ be the closed curve illustrated in Figure~\ref{fig:exampleCurvePiece}. Notice that $\gamma_0$ has three self-intersection points and that $S_2\setminus\gamma_0$ is a single topological disk. Applying Definition~\ref{defn:inducedPretrack} to $\gamma_0$ produces the invariant pretrack $\pretrack{\gamma_0}$ shown in Figure~\ref{fig:exampleTrackPiece}. Since the complementary components of $\pretrack{\gamma_0}$ consist of five trigons and one punctured monogon, $\pretrack{\gamma_0}$ is, in fact, a train track.

\begin{figure}
\centering
\subfigure[The filling loop $\gamma_0$.]
{
\labellist\small\hair 2pt
\pinlabel {$\gamma_0$} [br] at 64 108
\pinlabel {$p$} [tr] <0pt,-1pt> at 40 48
\endlabellist
\includegraphics[scale=1.0]{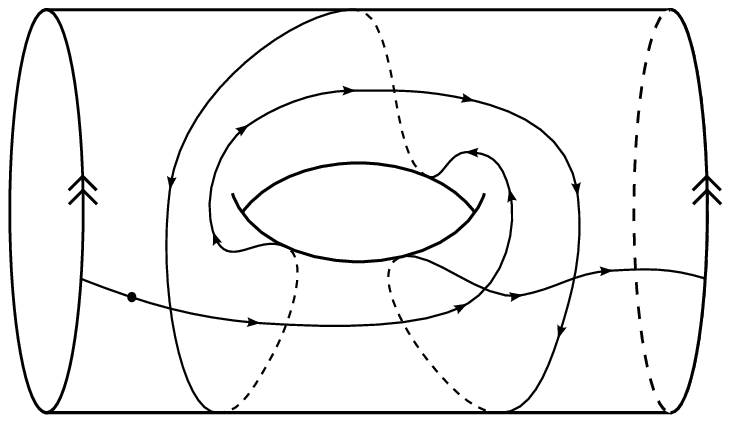}
\label{fig:exampleCurvePiece}
}
\subfigure[The corresponding pretrack $\pretrack{\gamma_0}$.]
{
\labellist\scriptsize\hair 2pt
\pinlabel {$\pretrack{\gamma_0}$} [bl] <0pt,-1pt> at 151 102
\pinlabel {$\drag$} [br]  <2pt,0pt> at 37 51
\pinlabel {$\lside$} [b]  <1pt,-0.5pt> at 48 45
\pinlabel {$\rside$} [tr] <0pt,-0.5pt> at 43 34
\pinlabel {$\ewt{1}$} [r]  <1pt,-1pt> at 61 42
\pinlabel {$\ewt{2}$} [tl] <0pt,-0.5pt> at 69 36
\pinlabel {$\ewt{3}$} [bl] at 67 45
\pinlabel {$\ewt{4}$} [tr] <1pt,0pt> at 137 53
\pinlabel {$\ewt{5}$} [l]  <0pt,1pt> at 144 56
\pinlabel {$\ewt{6}$} [br] <1pt,-1pt> at 140 60
\pinlabel {$\ewt{7}$} [br] <1.2pt,-0.2pt> at 173 54
\pinlabel {$\ewt{8}$} [tl] at 175 50
\pinlabel {$\ewt{9}$} [tr] <1pt,0pt> at 168 49
\endlabellist
\includegraphics[scale=1.0]{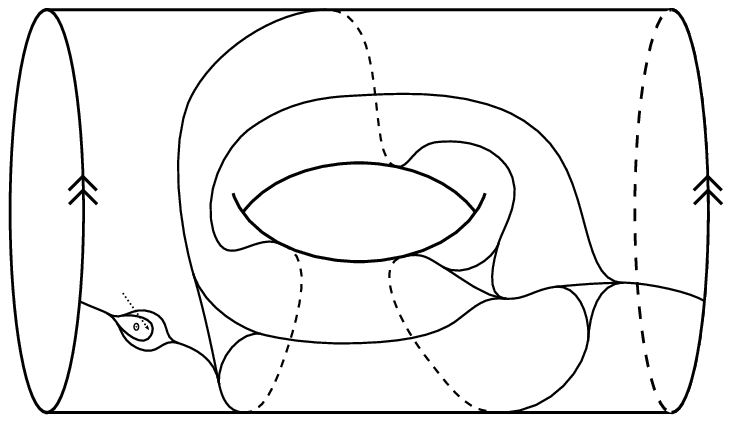}
\label{fig:exampleTrackPiece}
}
\caption{A point-pushing example on $S_2$ (the left and right boundaries are identified).}
\label{fig:examplePiece}
\end{figure}
 
Assigning weights as indicated in Figure~\ref{fig:exampleTrackPiece} determines a weight function on $\pretrack{\gamma_0}$. If $v_e = (\drag,\lside,\rside)$ denotes the vector of weights around the eye and $v = (\ewt{1},\ewt{2},\ewt{3},\ewt{4},\ewt{5},\ewt{6},\ewt{7},\ewt{8},\ewt{9})$ denotes the vector of weights at the intersection points of $\gamma_0$, then Corollary~\ref{cor:incidenceMatrix} implies that the mapping class $\push(\gamma_0)\in \Mod(S_2,p)$ transforms the weight vector $(v_e,v)$ according to the block matrix

\begin{equation}\label{eqn:matrixFormula}
\begin{pmatrix} A & B \\ C & D \end{pmatrix}
=
\begin{pmatrix}
\begin{pmatrix}
11 & 8 & 2\\
0  & 1 & 0\\
0  & 0 & 1\end{pmatrix} &
\begin{pmatrix}
5 & 0 & 5 & 4 & 4 & 0 & 1 & 0 & 1 \\
1 & 1 & 0 & 0 & 0 & 0 & 0 & 0 & 0 \\
0 & 0 & 0 & 0 & 0 & 0 & 1 & 1 & 0 \end{pmatrix}
\vspace{8 pt}\\

\begin{pmatrix}
2  & 2 & 0 \\
2  & 2 & 0 \\
0  & 0 & 0 \\
6  & 4 & 2 \\
2  & 2 & 0 \\
0  & 0 & 0 \\
10 & 8 & 2 \\
6  & 4 & 2 \\
0  & 0 & 0 \end{pmatrix} &
\begin{pmatrix}
2 & 0 & 2 & 2 & 1 & 0 & 0 & 0 & 0 \\
1 & 1 & 0 & 0 & 1 & 0 & 0 & 0 & 0 \\
0 & 0 & 0 & 0 & 0 & 1 & 0 & 0 & 0 \\
2 & 0 & 2 & 2 & 2 & 0 & 2 & 0 & 1 \\
2 & 0 & 2 & 1 & 1 & 0 & 0 & 0 & 1 \\
0 & 0 & 0 & 0 & 0 & 0 & 0 & 1 & 0 \\
6 & 0 & 5 & 3 & 3 & 0 & 2 & 0 & 2 \\
2 & 0 & 3 & 3 & 3 & 0 & 1 & 1 & 0 \\
0 & 1 & 0 & 0 & 0 & 0 & 0 & 0 & 0 \end{pmatrix}
\end{pmatrix}.
\end{equation}


The surface $S_g$ of genus $g$ is a natural $(g-1)$-fold, cyclic cover of $S_2$, which one may visualize as follows: Take $g-1$ copies of the torus with two boundary components, as in Figure~\ref{fig:exampleCurvePiece}, and glue them end-to-end to form a ring; the result is a rotationally symmetric surface with $g-1$ genera cyclically arranged around one central genus. The covering map $S_g \to S_2$ is the quotient by the $\Z/(g-1)\Z$ rotational symmetry. Explicitly, this is the cover corresponding to the kernel of the map
\[\pi_1(S_2)\to H_1(S_2;\Z) \to H_1\left(S_2;\Z/(g-1)\Z\right)\stackrel{\psi}{\to}\Z/(g-1)\Z,\]
where $\psi$ is Poincar\'e dual to the $\Z/(g-1)\Z$ homology class of the curve that is cut along in Figure~\ref{fig:exampleCurvePiece}. We now lift $\gamma_0^{g-1}\in \pi_1(S_2)$ to a filling loop $\gamma$ on $S_g$ and consider the corresponding mapping class $\push(\gamma)\in \pp(S_g)$. As illustrated in Figure~\ref{fig:exampleTrack}, the induced invariant pretrack $\pretrack{\gamma}$ is a train track that cuts $S_g$ into a single punctured monogon and $4g-3$ trigons---one at each of the $3g-3$ self-intersection points of $\gamma$, one for each of the $g-1$ components of $S_g\setminus \gamma$, and one surrounding the punctured monogon. Although the mapping class $\push(\gamma)$ is not a lift of $\push(\gamma_0)$, we note that $\gamma\subset S_g$ is precisely the preimage of $\gamma_0\subset S_2$ and that $\pretrack{\gamma}$ is essentially just the preimage of $\pretrack{\gamma_0}$.

\begin{figure}
\centering
\includegraphics[scale=1.0]{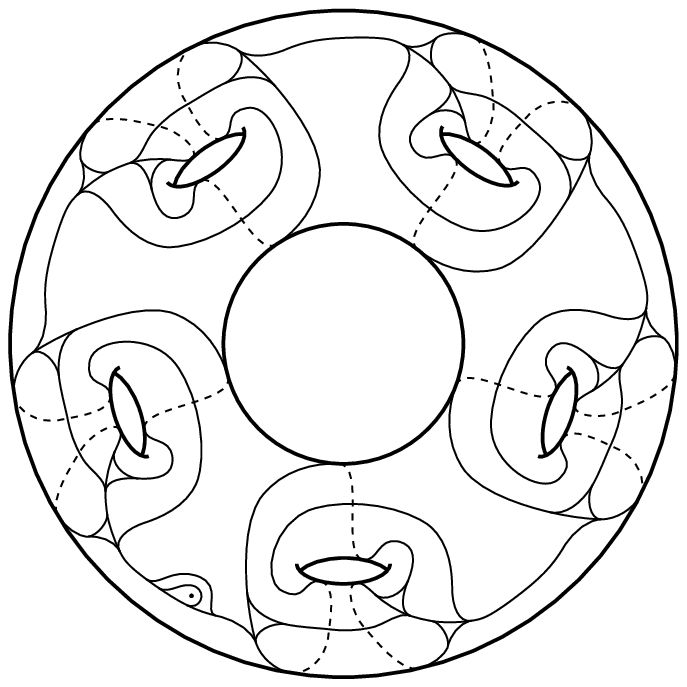}
\caption{The invariant train track $\pretrack{\gamma}$ for $\push(\gamma)\in \Mod(S_g,p)$.}
\label{fig:exampleTrack}
\end{figure}

We now calculate the action of $\push(\gamma)$ on the weight vector
\begin{equation*}
w = (v_e,v_1,v_2,\ldots,v_{g-1}),\quad\text{ where }\quad v_i=(\ewt{1}_i, \ewt{2}_i, \ewt{3}_i, \ewt{4}_i, \ewt{5}_i, \ewt{6}_i, \ewt{7}_i, \ewt{8}_i, \ewt{9}_i)
\end{equation*}
records the weights along the $i^{\text{th}}$ copy of $\gamma_0$ and $v_e = (\drag,\lside,\rside)$ records the weights in the eye of the track. Pushing the marked point along the $i^{\text{th}}$ copy of $\gamma_0$ transforms the weight vectors $v_e$ and $v_i$ according to \eqref{eqn:matrixFormula} and leaves the other vectors $v_j$ unchanged. Therefore, the full action of $\push(\gamma)$ on the weight vector $w$ is given by the block matrix product
\begin{equation}\label{eqn:bigMatrixProduct}
\begin{pmatrix}
A & 0 & 0 & \cdots & B \\
0 & I & 0 & \cdots & 0 \\
0 & 0 & I & \cdots & 0 \\
\vdots & \vdots & \vdots & \ddots & \vdots \\
C & 0 & 0 & \cdots & D
\end{pmatrix}
\cdots
\begin{pmatrix}
A & 0 & B & \cdots & 0 \\
0 & I & 0 & \cdots & 0 \\
C & 0 & D & \cdots & 0 \\
\vdots & \vdots & \vdots & \ddots & \vdots \\
0 & 0 & 0 & \cdots & I
\end{pmatrix}
\begin{pmatrix}
A & B & 0 & \cdots & 0 \\
C & D & 0 & \cdots & 0 \\
0 & 0 & I & \cdots & 0 \\
\vdots & \vdots & \vdots & \ddots & \vdots \\
0 & 0 & 0 & \cdots & I
\end{pmatrix}.
\end{equation}
Multiplying this out, we see that the incidence matrix for the carrying $\push(\gamma)(\pretrack{\gamma})\prec \pretrack{\gamma}$ is
\begin{equation}\label{eqn:bigIncidenceMatrix}
M_\gamma =
\begin{pmatrix}
A^{g-1}  & A^{g-2}B  & A^{g-3}B  & \cdots & A^2B   &  AB &  B \\
C        & D        & 0         & \cdots & 0      & 0  &  0 \\
CA       & CB       & D         & \cdots & 0      & 0  &  0 \\
\vdots   & \vdots   & \vdots    & \ddots & \vdots & \vdots & \vdots \\
CA^{g-3} & CA^{g-4}B & CA^{g-5}B & \cdots & CB     & D  &  0 \\
CA^{g-2} & CA^{g-3}B & CA^{g-4}B & \cdots & CAB    & CB &  D 
\end{pmatrix}.
\end{equation}

An elementary calculation shows that the first row and first column of $M_\gamma^3$ have strictly positive entries.
This implies that $M_\gamma^6$ has strictly positive entries and, consequently, that $M_\gamma$ is a Perron--Frobenius matrix. According to Lemma~\ref{lem:PerronFrobenius}, it now follows that the dilatation $\lambda_\gamma$ of $\push(\gamma)$ is bounded above by the largest row sum of $M_\gamma$. Using the fact that 
\begin{equation}\label{eqn:MtoTheN}
A^n = \begin{pmatrix}
11^n & 8q(n) & 2q(n)\\
0 & 1 & 0\\
0 & 0 & 1
\end{pmatrix},\quad\text{where} \quad q(n) = \sum_{j=0}^{n-1} 11^{j} = \frac{11^{n}-1}{11-1} \leq \frac{1}{10}11^{n},
\end{equation}
a direct comparison of the relevant matrix blocks shows that the first row of $M_\gamma$ has the largest sum. Indeed, since the first row of $A^{n-1}B$ sums to $20q(n)$, we may easily calculate the first row sum of $M_\gamma$ and conclude that
\begin{align*}
\lambda_{\gamma} \leq& 11^{g-1} + 10q(g-1) + \sum_{j=1}^{g-1}20q(j)\\
\leq& 11^{g-1} + 11^{g-1} + 2\sum_{j=1}^{g-1}11^{j}\\
=& 2\left(11^{g-1}\right) + 2\left(\tfrac{11^g -1}{11-1} - 1\right)\\
<& \tfrac{2}{5}11^{g}.
\end{align*}
Since $\push(\gamma)\in \pp(S_g)$, this example shows that
\[\least(\pp(S_g)) \leq \log(\lambda_\gamma) < g\log(11)\]
and completes the proof of Theorem~\ref{thm:leastDilatationFixedGenus}.\qedhere
\end{exmp}
\end{proof}

\subsection{Least dilatation vs. self-intersection number}

In keeping with the theme that the geometric structure of $\gamma$ controls the dynamical complexity of $\push(\gamma)$, we now refine our investigation of least dilatations to account for the dependence on self-intersection number. Accordingly, we restrict our attention to the filtration
\begin{equation*}
\pp_k(S) = \left\{\push(\gamma)\;\middle\vert\; \text{$\gamma\in \pi_1(S)$ with $\intnum(\gamma)=k$}\right\}
\end{equation*}
of $\pp(S)$ and strive to understand the least pseudo-Anosov dilatation achieved by pushing around a curve with prescribed self-intersection number. While Theorem~\ref{thm:generalBounds} gives general upper and lower bounds on the whole spectrum $\spec(\pp_k(S))$, the following theorem establishes a better upper bound on the least dilatation in $\pp_k(S_g)$ and proves that, asymptotically, $\least(\pp_k(S_g))$ grows like $\log(k)$.

\renewcommand{\thethmref}{\ref{thm:leastDilatationVaryN}}
\begin{thmref}
Let $S_g$ be a closed surface of genus $g\geq 3$. For any integer $k \geq 3g-1$, we have that
\[\tfrac{1}{5}\log(k+1) \leq \least(\pp_k(S_g)) < \log(k) + g\log(11).\]
\end{thmref}
\begin{proof}
The lower bound is a direct consequence of Theorem~\ref{thm:generalBounds}; the upper bound results from the following family of examples.

\begin{exmp}\label{ex:changeWinding}
For fixed numbers $n\geq 2$ and $g \geq 3$, we will construct a filling curve on $S_g$ whose self-intersection number depends on $n$. As in Example~\ref{ex:fixedSurface}, our construction uses the cyclic covering $S_g\to S_2$ and the filling curve $\gamma_0\subset S_2$ shown in Figure~\ref{fig:exampleCurvePiece}. To adjust the self-intersection number, we also consider the modified loop $\mu_0\colon[0,1]\to S_2$ illustrated in Figure~\ref{fig:windingCurvePiece}---$\mu_0$ is identical to $\gamma_0$ except in that it winds around a handle of $S_2$ an additional $n$ times; as such, it has self-intersection number $\intnum(\mu_0) = n+3$.

\begin{figure}
\centering
\subfigure[The loop $\mu_0$ winds around a handle of $S_2$.]
{
\labellist\small\hair 2pt
\pinlabel {$\mu_0$} [br] at 56 103
\pinlabel {$p$} [tr] <0pt,-1pt> at 40 49
\pinlabel {$n$-times} [b] <0pt,0.5pt> at 109 141
\endlabellist
\includegraphics[scale=1.0]{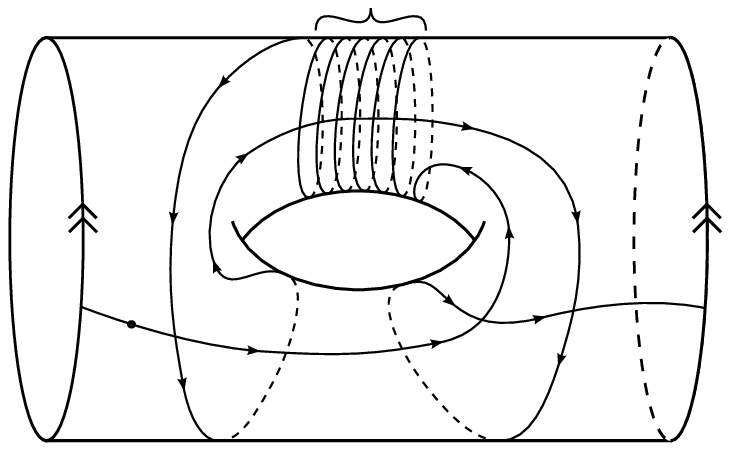}
\label{fig:windingCurvePiece}
}
\subfigure[The track $\sigma_0$ is invariant under $\push(\mu_0)$.]
{
\labellist\scriptsize\hair 2pt
\pinlabel {$\sigma_0$} [bl] <1pt,-1pt> at 154 93
\pinlabel {$\drag$}    [br] <2pt,-1pt> at 37 53
\pinlabel {$\lside$}   [bl] <0pt,-1pt> at 48 45
\pinlabel {$\rside$}   [tr] at 43 34
\pinlabel {$\ewt{1}$}  [r]  <0.5pt,0pt> at 62 40
\pinlabel {$\ewt{2}$}  [tl] <-1pt,-1pt> at 70 36
\pinlabel {$\ewt{3}$}  [bl] <0pt,-1pt> at 68 45
\pinlabel {$\ewt{4}$}  [tr] <1pt,-0.5pt> at 139 52
\pinlabel {$\ewt{5}$}  [l]  <0pt,0.5pt> at 145 56
\pinlabel {$\ewt{6}$}  [br] <1pt,0pt> at 139 58
\pinlabel {$\ewt{7}$}  [br] <1pt,-0.5pt> at 172 56
\pinlabel {$\ewt{8}$}  [tl] <-1pt,-1pt> at 175 53
\pinlabel {$\ewt{9}$}  [tr] <0pt,-1pt> at 167 50
\pinlabel {$\ewt{10}$} [tr] at 92 95
\pinlabel {$\ewt{11}$} [t]  <0.5pt,-0.5pt> at 103 100
\pinlabel {$\ewt{12}$} [l]  <0.5pt,0pt> at 110 116
\pinlabel {$\ewt{13}$} [tr] <0.5pt,0pt> at 107 113
\pinlabel {$\ewt{14}$} [br] <0.5pt,-0.5pt> at 106 120
\endlabellist
\includegraphics[scale=1.0]{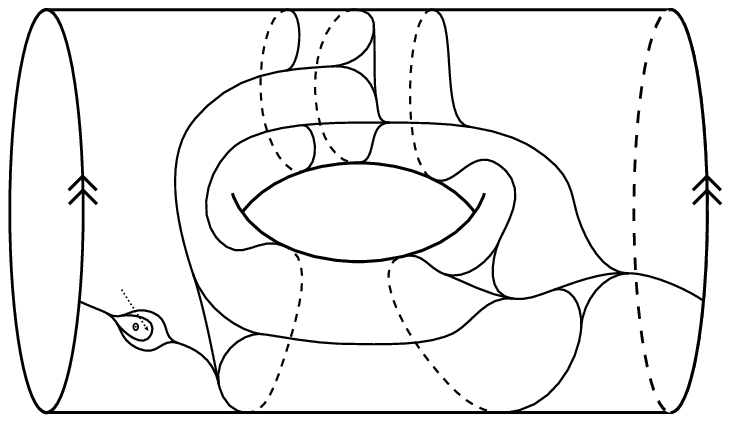}
\label{fig:windingTrackPiece}
}
\caption{A more complicated example on $S_2$.}
\label{fig:windingPieces}
\end{figure}

The winding structure of $\mu_0$ creates complementary monogons in the corresponding pretrack $\pretrack{\mu_0}\subset S_2$, so we instead consider the modified track $\sigma_0\subset S_2$ depicted in Figure~\ref{fig:windingTrackPiece}. This track cuts $S_2$ into six trigons, two bigons, one monogon, and one punctured monogon. Notice that $\sigma_0$ has the same basic structure of $\pretrack{\gamma_0}$ from Figure~\ref{fig:exampleTrackPiece}, but with additional branches to deal with the winding in $\mu_0$.

\begin{claim}\label{claim:anInvariantTrack}
The track $\sigma_0\subset S_2$ is invariant under $\push(\mu_0)$, that is, $\push(\mu_0)(\sigma_0)\prec \sigma_0$.
\end{claim}
\begin{proof}
The proof is in the same spirit as Proposition~\ref{prop:invariantPretrack}: as the marked point travels along $\mu_0$, it pushes the track $\sigma_0$ out of its way while we continually use isotopy and carrying moves (c.f. Figure~\ref{fig:moves}) to simplify the picture. After pushing the marked point once around $\mu_0$, the resulting track will be identical to $\sigma_0$.

Let $C\subset S_2$ be a compact cylinder containing all of the extra winding in the path $\mu_0$ so that the paths $\gamma_0$ and $\mu_0$ and tracks $\pretrack{\gamma_0}$ and $\sigma_0$ both agree on $X = S_2\setminus C$. It follows that, while the marked point is in $X$, we may apply the same simplifying moves as in the carrying $\push(\gamma_0)(\pretrack{\gamma_0})\prec \pretrack{\gamma_0}$ from Proposition~\ref{prop:invariantPretrack}. Namely, when the marked point travels along an edge of $\mu_0\cap X$, the eye of the track simply slides along the corresponding branch of $\sigma_0\cap X$, and when the marked point pushes through a self-intersection point in $X$, we perform the usual carrying from Figure~\ref{fig:invarianceAtIntersection}.

The only difficulty, then, is to navigate the portions of $\mu_0$ that lie in $C$; once this is done, the claim will follow. Notice that, in its journey around $\mu_0$, the marked point interacts with $C$ exactly twice. The first time, it winds $n$ times around $C$ as it progresses from the right boundary component to the left; the second time, the marked point simply traverses $C$ from left to right. The necessary carryings for these sections are a bit involved, so we demonstrate them explicitly in a sequence of $17$ snapshots spanning Figures~\ref{fig:crazyCarrying1} through \ref{fig:crazyCarrying4}. In the figures, the cylinder $C$ is depicted as a rectangle with top and bottom edges identified. 

\begin{figure}
\centering
\labellist\small\hair 8pt
\pinlabel {$1$} [bl] at 40 119
\pinlabel {$2$} [bl] at 41 37
\pinlabel {$3$} [bl] at 230 119
\pinlabel {$4$} [bl] at 230 37
\endlabellist
\includegraphics[scale=1.0]{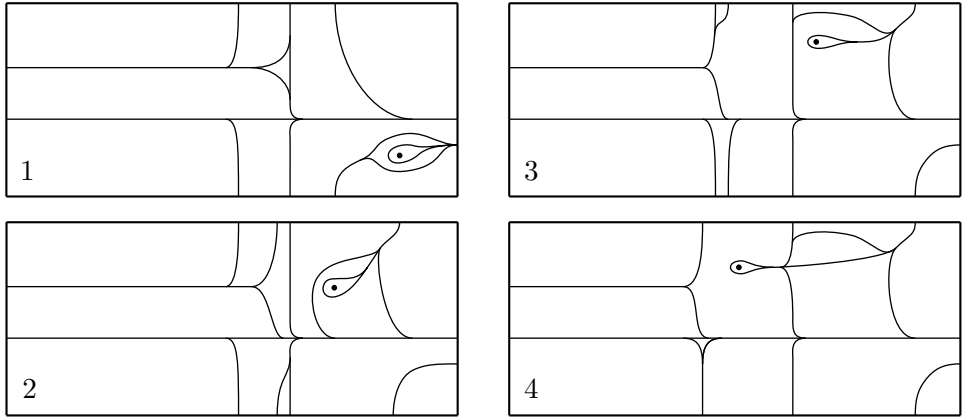}
\caption{The carrying moves for navigating a winding path -- part 1.}
\label{fig:crazyCarrying1}
\end{figure}

The first time the marked point crosses $C$ is illustrated in Figures~\ref{fig:crazyCarrying1} through \ref{fig:crazyCarrying3}. Frame $1$ shows the initial situation after the marked point, surrounded by the eye of the track, enters $C$ from the right. In frames $2$ and $3$, the marked point passes around the back of $C$, and we use a handful of slide moves to pull branches of the track apart. In frame $4$, the marked point pushes through a branch of the track, which we then pinch and collapse onto the monogon containing the marked point.

\begin{figure}
\centering
\labellist\small\hair 8pt
\pinlabel {$5$} [bl] at 40 112
\pinlabel {$6$} [bl] at 40 29
\pinlabel {$7$} [bl] at 229 112
\pinlabel {$8$} [bl] at 229 29
\endlabellist
\includegraphics[scale=1.0]{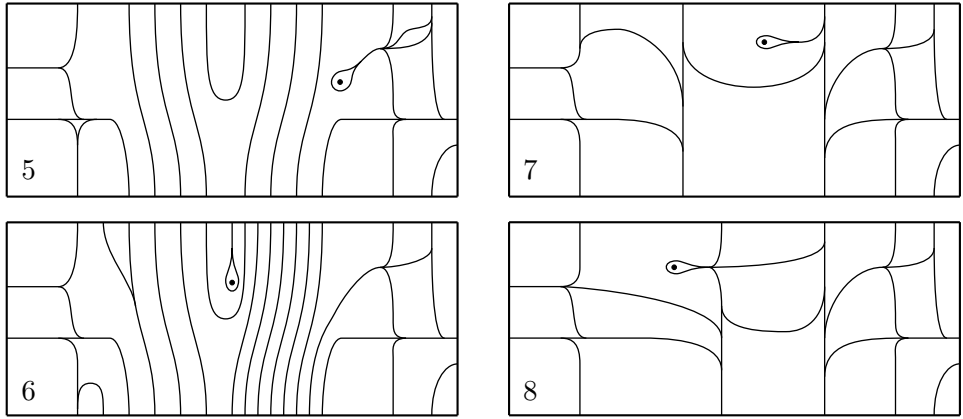}
\caption{The carrying moves for navigating a winding path -- part 2.}
\label{fig:crazyCarrying2}
\end{figure}

The next step is for the marked point to wind around $C$; however, the long horizontal branch of the track is blocking the way. In frame $5$, we clear a path for the marked point by twisting this branch around $C$ so that, in frame $6$, the marked point is able to wind $n$ times around $C$ by moving along the space between these twistings. In frame $7$, we pinch many of these branches together and collapse the resulting bigons. The marked point then pushes through another branch of the track in frame $8$. 

\begin{figure}
\centering
\labellist\small\hair 8pt
\pinlabel {$9$}  [bl] at 39 115
\pinlabel {$10$} [bl] at 39 33
\pinlabel {$11$} [bl] at 228 115
\pinlabel {$12$} [bl] at 228 33
\endlabellist
\includegraphics[scale=1.0]{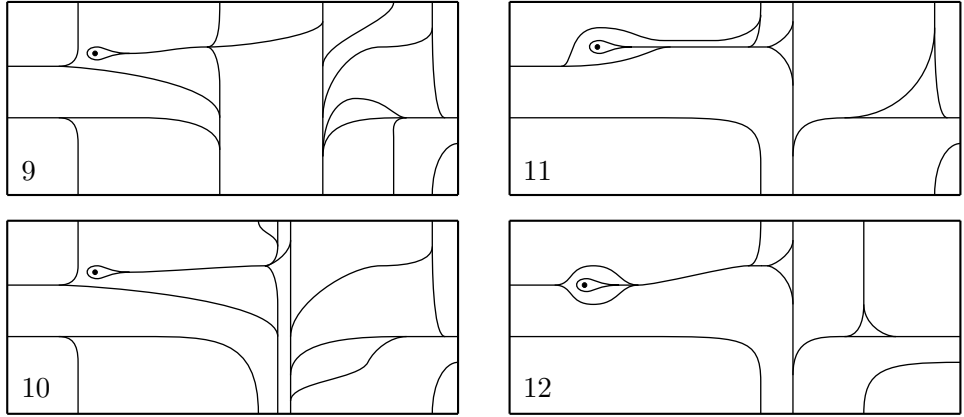}
\caption{The carrying moves for navigating a winding path -- part 3.}
\label{fig:crazyCarrying3}
\end{figure}

All that remains is to repackage the branches into an orderly configuration. In frame $9$, we collapse two bigons and use slide moves to pull apart some of the branches on the right side of $C$. In frames $10$ and $11$, we collapse two more bigons and use slides to combine several branches together. Frame $12$ shows the eye of the track reformed and ready to exit through the left boundary component of $C$.

\begin{figure}
\centering
\labellist\small\hair 8pt
\pinlabel {$13$} [bl] at 16 116
\pinlabel {$14$} [bl] at 16 33
\pinlabel {$15$} [bl] at 142 116
\pinlabel {$16$} [bl] at 142 33
\pinlabel {$17$} [bl] at 268 74
\endlabellist
\includegraphics[scale=1.0]{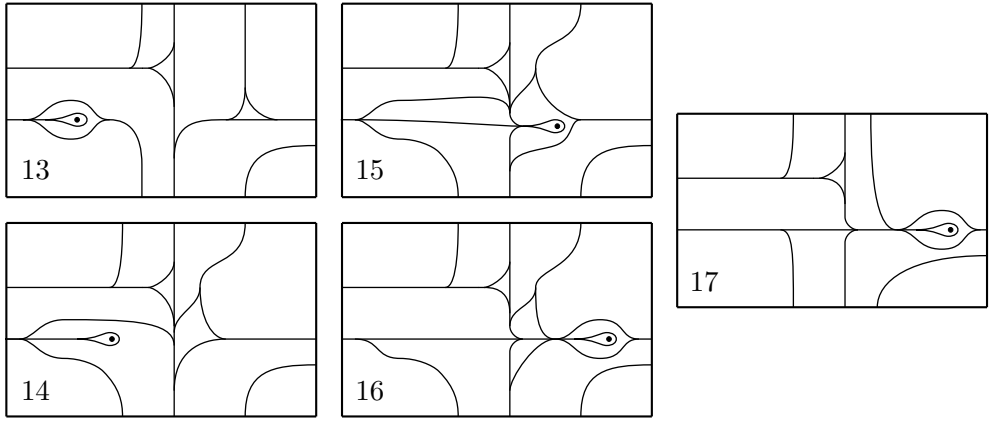}
\caption{The carrying moves for navigating a winding path -- part 4.}
\label{fig:crazyCarrying4}
\end{figure}

The second time the marked point crosses $C$ is illustrated in Figure~\ref{fig:crazyCarrying4}, with the initial configuration depicted in frame $13$. In frame $14$, we use slide moves to pull apart branches and make room for the marked point to move forward. In frame $15$, the marked point pushes through the vertical branch, and in frame $16$, a bigon is collapsed while branches are pinched to form an eye around the marked point. Upon collapsing two more bigons, frame $17$ shows the eye of the track ready to exit $C$, leaving the branches in $C$ exactly how they started in frame $1$. This completes the proof that $\push(\mu_0)(\sigma_0) \prec \sigma_0$.
\end{proof}

As in Example~\ref{ex:fixedSurface}, a weight function on $\sigma_0$ is determined by the weight vectors
\[v_e=(\drag,\lside,\rside)\quad \text{ and } \quad
v = (\ewt{1},\ewt{2},\ewt{3},\ewt{4},\ewt{5},\ewt{6},\ewt{7}, \ewt{8},\ewt{9},\ewt{10},\ewt{11},\ewt{12},\ewt{13},\ewt{14})\]
that are indicated in Figure~\ref{fig:windingTrackPiece}. By carefully keeping track of these weights throughout the carryings in Figures~\ref{fig:invarianceAtIntersection} and \ref{fig:crazyCarrying1}--\ref{fig:crazyCarrying4}, one finds that pushing the marked point around $\mu_0$ transforms the weight vector $(v_e,v)$ according to the block matrix

\begin{equation}\label{eqn:newMatrixFormula}
\begin{pmatrix}
\tilde{A} & \tilde{B}\\
\tilde{C} & \tilde{D}
\end{pmatrix}
=
\begin{pmatrix}
\left(\begin{smallmatrix}
6n+11 & 6n+11 & 0 \\
0     & 0     & 0 \\
0     & 0     & 0 \\
\end{smallmatrix}\right) & \hspace{-10pt}
\left(\begin{smallmatrix}
6n+8 & 3 & 6n+5 & 6n+4 & 6n+3 & 0 & 1 & 0 & 1 & 3 & 9n-4 & 6n+4 & 9n-1 & 9n-7 \\
0    & 0 & 0    & 0    & 0    & 1 & 0 & 0 & 0 & 0 & 0    & 0    & 0    & 0    \\
0    & 0 & 0    & 0    & 1    & 0 & 1 & 1 & 0 & 0 & 1    & 2    & 1    & 1    \\
\end{smallmatrix}\right)
\vspace{6pt} \\
\left(\begin{smallmatrix}
 2     & 2     & 0 \\
 2     & 2     & 0 \\
 0     & 0     & 0 \\
 4n+6  & 4n+6  & 0 \\
 2     & 2     & 0 \\
 0     & 0     & 0 \\
 4n+10 & 4n+10 & 0 \\
 4n+6  & 4n+6  & 0 \\
 0     & 0     & 0 \\
 0     & 0     & 1 \\
 6     & 5     & 0 \\
 2n-2  & 2n-2  & 0 \\
 0     & 0     & 0 \\
 2     & 3     & 0
\end{smallmatrix}\right) & \hspace{-10pt}
\left(\begin{smallmatrix}
2    & 0 & 2    & 2    & 2    & 0 & 0 & 0 & 0 & 1 & 2n   & 2     & 2n  & 2n-2 \\
1    & 1 & 0    & 0    & 0    & 0 & 0 & 0 & 0 & 1 & 0    & 0     & 0   & 0    \\
0    & 0 & 0    & 0    & 0    & 0 & 0 & 0 & 0 & 0 & 0    & 0     & 1   & 1    \\
4n+4 & 2 & 4n+2 & 4n+2 & 4n+2 & 0 & 2 & 0 & 1 & 2 & 6n-2 & 4n+4  & 6n  & 6n-4 \\
2    & 0 & 2    & 1    & 1    & 0 & 0 & 0 & 1 & 0 & 0    & 0     & 0   & 0    \\
0    & 0 & 0    & 0    & 0    & 0 & 0 & 1 & 0 & 0 & 0    & 0     & 0   & 0    \\
4n+8 & 2 & 4n+5 & 4n+3 & 4n+3 & 0 & 2 & 0 & 2 & 2 & 6n-2 & 4n+4  & 6n  & 6n-4 \\
4n+4 & 2 & 4n+3 & 4n+3 & 4n+3 & 0 & 1 & 1 & 0 & 2 & 6n-2 & 4n+4  & 6n  & 6n-4 \\
0    & 1 & 0    & 0    & 0    & 0 & 0 & 0 & 0 & 0 & 0    & 0     & 0   & 0    \\
0    & 0 & 0    & 0    & 0    & 0 & 0 & 0 & 0 & 0 & 0    & 0     & 0   & 0    \\
3    & 1 & 2    & 2    & 2    & 0 & 0 & 0 & 0 & 2 & 2n   & 2     & 2n  & 2n-2 \\
2n-2 & 0 & 2n-2 & 2n-2 & 2n-2 & 0 & 0 & 0 & 0 & 0 & n-1  & 2n-1  & n-1 & n-1  \\
0    & 0 & 0    & 0    & 0    & 0 & 0 & 0 & 0 & 0 & n-1  & 0     & n   & n-1  \\
3    & 1 & 2    & 2    & 2    & 0 & 0 & 0 & 0 & 0 & n    & 2     & n+1 & n      
\end{smallmatrix}\right)
\end{pmatrix}.
\end{equation}

\begin{figure}
\centering
\includegraphics[scale=1.0]{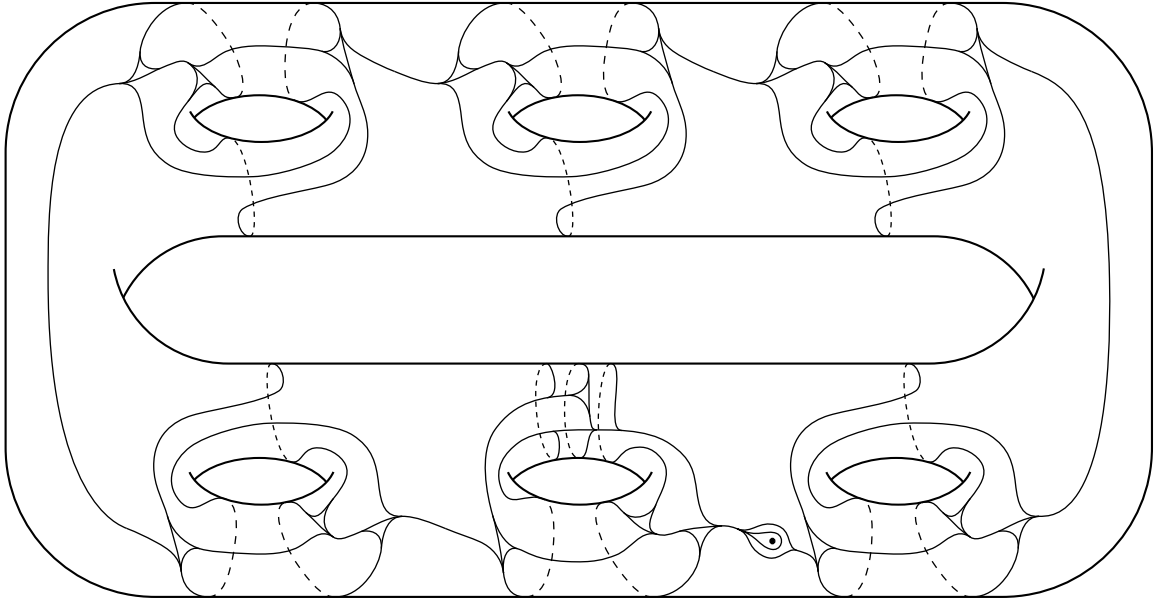}
\caption{The bigon track $\sigma$ is invariant under $\push(\mu)\in\Mod(S,p)$.}
\label{fig:windingExampleTrack}
\end{figure}

We are now ready to construct the example. Using the cyclic cover, $S_g\to S_2$, the loop $\gamma_0^{g-2}\mu_0\in\pi_1(S_2)$ lifts to a filling loop $\mu\subset S_g$. Alternately, if $T$ denotes the torus with two boundary components, then $S_g$ may be attained by gluing $g-1$ copies of $T$ together to form a ring, and $\mu\subset S_g$ may be constructed by concatenating copies of the arc $\gamma_0\subset T$ along the first $g-2$ copies of $T$ with the arc $\mu_0\subset T$ on the last copy of $T$. Note that $\intnum(\mu) = 3(g-1)+n$. Similarly, we may build an invariant track $\sigma\subset S_g$ for $\push(\mu)\in \Mod(S_g,p)$ by concatenating $g-2$ copies of $\pretrack{\gamma_0}\subset T$ with one copy of $\sigma_0\subset T$. As shown in Figure~\ref{fig:windingExampleTrack}, $\sigma$ is a bigon track that that cuts $S_g$ into $(2+ 3(g-1)+(g-3) +1)$ trigons, four bigons, and one punctured monogon (each of the $g-1$ pieces of $\sigma$ contributes three trigons, the copy of $\sigma_0$ contributes two additional trigons and two bigons, the eye contributes a trigon and a punctured monogon, and the junctions between the pieces form either bigons or trigons).

The track $\sigma$ allows us to estimate the dilatation of $\push(\mu)$ as follows. The action of $\push(\mu)$ on the weight space of $\sigma$ is determined by a matrix product that is analogous to the one in \eqref{eqn:bigMatrixProduct}. Using \eqref{eqn:newMatrixFormula} together with \eqref{eqn:bigIncidenceMatrix}, we find that the incidence matrix $M_\mu$ for the carrying $\push(\mu)(\sigma)\prec \sigma$ is given by the block matrix product
\begin{align*}
\begin{pmatrix}
\tilde{A} & 0 & 0 & \cdots & 0 & \tilde{B} \\
0         & I & 0 & \cdots & 0 & 0         \\
0         & 0 & I & \cdots & 0 & 0         \\
\vdots    & \vdots & \vdots & \ddots & \vdots & \vdots \\
0         & 0 & 0 & \cdots & I & 0         \\ 
\tilde{C} & 0 & 0 & \cdots & 0 & \tilde{D} 
\end{pmatrix}
\begin{pmatrix}
A^{g-2}   & A^{g-3}B  & A^{g-4}B  & \cdots & AB  & B      & 0 \\
C        & D         & 0         & \cdots & 0   & 0      & 0  \\
CA       & CB        & D         & \cdots & 0   & 0      & 0  \\
\vdots   & \vdots    & \vdots    &\ddots &\vdots&\vdots & \vdots \\
CA^{g-3}  & CA^{g-4}B & CA^{g-5}B & \cdots & CB  & D      & 0 \\
0        & 0         & 0         &  0     & 0   & 0      & I 
\end{pmatrix},
\end{align*}
which multiplies to give 
\begin{align*} 
M_\mu = \begin{pmatrix}
\tilde{A}A^{g-2}   & \tilde{A}A^{g-3}B  & \tilde{A}A^{g-4}B  & \cdots & \tilde{A}AB &  \tilde{A}B & \tilde{B} \\
C   & D         & 0         & \cdots & 0  &  0 & 0  \\
CA       & CB        & D         & \cdots & 0  &  0 & 0  \\
\vdots   & \vdots    & \vdots    & \ddots & \vdots & \vdots & \vdots \\
CA^{g-3}  & CA^{g-4}B & CA^{g-5}B & \cdots & CB &  D & 0 \\
\tilde{C}A^{g-2}   & \tilde{C}A^{g-3}B  & \tilde{C}A^{g-4}B  & \cdots & \tilde{C}AB &  \tilde{C}B & \tilde{D} \\
\end{pmatrix}.
\end{align*}

As before, one may show that $M_\mu$ is Perron--Frobenius by checking that the first row and first column of $M_\mu^3$ have strictly positive entries. Lemma~\ref{lem:PerronFrobenius} then implies that the dilatation of $\push(\mu)$ is bounded above by the sum of the entries in the largest row of $M_\mu$, which is evidently the first. Making use of \eqref{eqn:MtoTheN}, we see that the first row of $\tilde{A}A^j$ has sum $(6n+11)(11^j+10q(j)+1)$ and that the first row of $\tilde{A}A^{j}B$ sums to $(6n+11)(20q(j+1)+2)$. Since the first row of $\tilde{B}$ sums to $57n+20$, we find that the first row of $M_{\mu}$ has total sum
\begin{align*}
(6n+11)&(11^{g-2}+10q(g-2)+1) + (57n+20) + \sum_{j=1}^{g-2}(6n+11)(20q(j)+2)\\
&\leq (57n+20)+(6n+11)\left[11^{g-2} + 11^{g-2}+ 1 + \left(\sum_{j=1}^{g-2}2\cdot11^{j}\right) + 2(g-2)\right]\\
&\leq (57n+10n)+(6n+6n)\left[2(11^{g-2}) +2\left(\frac{11^{g-1}-1}{11-1} -1\right) + 2g-3 \right]\\
&\leq (67n)+(12n)\left[\tfrac{2}{11}11^{g-1} + \tfrac{2}{10}11^{g-1} + \tfrac{2}{11}11^{g-1} \right]\\
&\leq \left(11^{g-1}n\right)+(12n)\left[\tfrac{3}{5}11^{g-1}\right]\\
&< n 11^{g}.
\end{align*}
Since $\intnum(\mu) = 3(g-1) + n \geq n$, it follows that $\pushdil{\mu} < \intnum(\mu)11^{g}$. This construction shows that for each $k \geq 3g-1$ in the statement of Theorem~\ref{thm:leastDilatationVaryN}, we may take $n = k-3(g-1)\geq 2$ and produce a filling curve $\mu\subset S_g$ with $\intnum(\mu) = k$ whose corresponding mapping class $\push(\mu)$ has dilatation at most $k 11^g$. Since $\push(\mu)\in \pp_k(S_g)$ this shows that
\[\least(\pp_k(S_g)) < \log(k) + g\log(11)\]
and completes the proof of Theorem~\ref{thm:leastDilatationVaryN}.\qedhere
\end{exmp}
\end{proof}

{\small
\bibliography{surface_topology}{}
\bibliographystyle{alphanum}}

\bigskip

\noindent
Department of Mathematics\\
University of Illinois at Urbana-Champaign\\
1409 W. Green Street\\
Urbana, IL 61801\\
E-mail: {\tt dowdall@illinois.edu}

\end{document}